\documentclass[12pt,a4paper]{amsart}
\usepackage{latexsym,amsfonts,amsthm,amsmath,mathrsfs,amssymb}
\usepackage[british]{babel}
\usepackage[dvips]{color}
\usepackage[utf8]{inputenc}
\usepackage[T1]{fontenc}
\usepackage[dvips,final]{graphics}
\usepackage{xcolor}
\usepackage{mathrsfs}
\usepackage{mathtools} 
\usepackage{breqn}
\usepackage{epstopdf}
\usepackage{verbatim}
\usepackage{amscd}
\usepackage{hyperref}
\usepackage[english,capitalize]{cleveref}
\usepackage{aliascnt}

\makeatletter
\def\newaliasedtheorem#1[#2]#3{
	\newaliascnt{#1@alt}{#2}
	\newtheorem{#1}[#1@alt]{#3}
	\expandafter\newcommand\csname #1@altname\endcsname{#3}
}
\makeatother

\theoremstyle{plain}

\newtheorem{theorem}{Theorem}[section]
\newaliasedtheorem{prop}[theorem]{Proposition}
\newaliasedtheorem{lem}[theorem]{Lemma}
\newaliasedtheorem{coroll}[theorem]{Corollary}

\theoremstyle{definition}
\newaliasedtheorem{defi}[theorem]{Definition}
\newaliasedtheorem{quest}[theorem]{Question}
\newaliasedtheorem{fact}[theorem]{Fact}

\theoremstyle{remark}
\newaliasedtheorem{rem}[theorem]{Remark}
\newaliasedtheorem{exa}[theorem]{Example}

\newcommand{\R}{\mathbb{R}}

\newcommand{\W}{\mathbb{W}}
\newcommand{\LL}{\mathbb{L}}

\newcommand{\T}{\mathbb{T}}

\newcommand{\G}{\mathbb{G}}

\newcommand{\sW}{\mathbb{W}}
\newcommand{\sL}{\mathbb{L}}

\newcommand{\eps}{\varepsilon}
\let\altphi\phi
\let\phi\varphi
\let\varphi\altphi
\let\altphi\undefined

\newcommand{\graph}[1]{\mathrm{graph}\,(#1)}

\DeclareMathOperator{\de}{\mathrm d}

\newcommand{\average}{{\mathchoice {\kern1ex\vcenter{\hrule height.4pt
width 6pt
depth0pt} \kern-9.7pt} {\kern1ex\vcenter{\hrule height.4pt width 4.3pt
depth0pt}
\kern-7pt} {} {} }}

\newcommand{\res}{\mathop{\hbox{\vrule height 7pt width .5pt depth 0pt
\vrule height .5pt width 6pt depth 0pt}}\nolimits}

\address{\textsc{Gioacchino Antonelli}: 
	Scuola Normale Superiore, Piazza dei Cavalieri, 7, 56126 Pisa, Italy}
\email{gioacchino.antonelli@sns.it}
\address{\textsc{Daniela Di Donato}: 
Department of Mathematics and Statistics, P.O.\ Box 35 (MaD), FI--40014, University of Jyv\"askyl\"a, Finland.}
\email{daniela.d.didonato@jyu.fi}
\address{\textsc{Sebastiano Don}:
Department of Mathematics and Statistics, P.O.\ Box 35 (MaD), FI--40014, University of Jyv\"askyl\"a, Finland.}
\email{sedon@jyu.fi}
\address{\textsc{Enrico Le Donne}: 
	Dipartimento di Matematica, Universit\`a di Pisa, Largo B. Pontecorvo 5, 56127 Pisa, Italy \\
	\& \\
	University of Jyv\"askyl\"a, Department of Mathematics and Statistics, P.O. Box (MaD), FI-40014, Finland}
\email{enrico.ledonne@unipi.it}

\thanks{
	G.A., D.D.D., S.D. and E.L.D. are partially supported by the Academy of Finland (grant
	288501
	`\emph{Geometry of subRiemannian groups}' and by grant
	322898
	`\emph{Sub-Riemannian Geometry via Metric-geometry and Lie-group Theory}')
	and by the European Research Council
	(ERC Starting Grant 713998 GeoMeG `\emph{Geometry of Metric Groups}').
}

\title{Characterizations of uniformly differentiable co-horizontal intrinsic graphs in Carnot groups}

\date{\today}

\author{Gioacchino Antonelli, Daniela Di Donato, \\ Sebastiano Don and Enrico Le Donne}

\begin{document}

\begin{abstract}
	In arbitrary Carnot groups we study intrinsic graphs of maps with horizontal target. These graphs are $C^1_{\rm H}$ regular exactly when the map is uniformly intrinsically differentiable. Our first main result characterizes the uniformly intrinsic differentiability by means of H\"older properties along the projections of left-invariant vector fields on the graph. 
	
	We strengthen the result in step-2 Carnot groups for intrinsic real-valued maps by only requiring horizontal regularity. We remark that such a refinement is not possible already in the easiest step-3 group. 
	
	As a by-product of independent interest, in every Carnot group we prove an area-formula for uniformly intrinsically differentiable real-valued maps. We also explicitly write the area element in terms of the intrinsic derivatives of the map.
\end{abstract}

\maketitle 
\tableofcontents

\section{Notation}
\begin{tabular}{c p{0.7\textwidth} p{\textwidth}}	
	$(\phi,\widetilde\phi)$, $(\gamma,\widetilde\gamma)$, $\dots$ & $\square$ is the representation of $\widetilde{\square}$ in exponential coordinates. See \cref{def:coordinateconletilde}. & \\
	
	$L_g,R_g$ & Left-translation and right-translation by $g\in \mathbb G$. & \\
	
	$\deg $ & Holonomic degree. See \cref{sec:Prel}. & \\
		
	$\|\cdot\|,\|\cdot\|_{\mathbb G}$ & Homogeneous norm on $\mathbb G$. See \eqref{eqn:HomogeneousNorm}. & \\
	
	$\mathbb V_1$ & Horizontal bundle of $\mathbb G$. See \eqref{eqn:HorizontalBundle}. & \\
	
	$\pi_{\mathbb V_1}$ & Projection on the horizontal bundle. See \eqref{eqn:ProjectionOnV1}. & \\
	
	$(\mathbb W,\mathbb L)$ &  Complementary subgroups. See \cref{def:ComplementarySubgroups}. & \\
	
	$\pi_{\mathbb W}(g),g_{\mathbb W}$ & Projection of $g\in\mathbb G$ onto the homogeneous subgroup $\mathbb W$, given the splitting $\mathbb G=\mathbb W\cdot \mathbb L$. See \eqref{eqn:ComponentsSplitting}. & \\
	
	$h^\alpha (U;\R^k)$ & $\mathbb R^k$-valued $\alpha$-little H\"older functions defined on $U\subseteq \mathbb R^n$. See \cref{big3.3.11}. & \\
	
	$\widetilde\Phi(\widetilde U), \mathrm{graph}(\widetilde\phi)$ & Intrinsic graph of $\widetilde\phi:\widetilde U\subseteq \mathbb W\to\mathbb L$, given the splitting $\mathbb G=\mathbb W\cdot \mathbb L$. See \cref{def:IntrinsicGraph}. & \\
	
	$\widetilde\phi_q:\widetilde U_q\subseteq \mathbb W\to \mathbb L$ &  Intrinsic $q$-translation, with $q\in \mathbb G$, of the function $\widetilde\phi:\widetilde U\subseteq \mathbb W\to\mathbb L$, given the splitting $\mathbb G=\mathbb W\cdot \mathbb L$. See \cref{def:PhiQ}. &\\
	
	$\de_{\rm P}\! f$ & Pansu differential of the $C^1_{\rm H}$ function $f$. See \cref{def:PansuDifferentiability}. &\\
	
	$\de^{\phi}\!\phi_{a_0},\de^{\phi}\!\phi(a_0):\mathbb W\to\mathbb L$ & Intrinsic differential of the function $\widetilde\phi:\widetilde U\subseteq \mathbb W\to \mathbb L$, at a point $a_0\in\widetilde U$, given the splitting $\mathbb G=\mathbb W\cdot \mathbb L$. See  \cref{defiintrinsicdiff}. &
	\\
	
	${\rm (U)ID}(\widetilde U,\sW;\sL)$ & Set of (uniformly) intrinsically differentiable functions $\widetilde\phi:\widetilde U\subseteq \mathbb W\to\mathbb L$, given a splitting $\mathbb G=\mathbb W\cdot\mathbb L$. See \cref{defiintrinsicdiff}. &\\
	
	$\nabla^{\phi}\phi_{a_0}, \nabla^{\phi}\phi(a_0)$ & When $\mathbb L$ is horizontal, the map, identified with a $k\times (m-k)$-matrix, in $\mathrm{Lin}(\mathrm{Lie}(\mathbb W)\cap V_1,\mathrm{Lie}(\mathbb L))$, corresponding to the intrinsic differential $\de^{\phi}\!\phi_{a_0}$. See \cref{def:NablaPhiPhi} and \cref{rem:IntrinsicGradientInCoordinates}. &\\
	
	$\nabla^{\phi}_j\phi, \nabla^{\phi}\phi(X_j)$ & The $j$-th component, either a number or a vector, of the intrinsic gradient $\nabla^{\phi}\phi$. &\\ 

	$\nabla _{\rm H} f = \left(
	\nabla_{\mathbb L}f \;|\;  \nabla_{\mathbb W}f 
	\right)$ & Pansu differential of $f\in C^1_{\rm H}(\widetilde V;\mathbb R^k)$ in coordinates adapted to the splitting $\mathbb G=\mathbb W\cdot\mathbb L$. See \cref{remarkDPHI2Vectorial}. &\\
	
	$D^{\phi}_W$ & Intrinsic projected vector field on $\widetilde U\subseteq \mathbb W$, relative to the vector field $W\in\mathrm{Lie}(\mathbb W)$, and to the function $\widetilde\phi:\widetilde U\subseteq\mathbb W\to\mathbb L$, given the splitting $\mathbb G=\mathbb W\cdot\mathbb L$. See \cref{def:PhiJ}.   &\\
	
	$D^{\phi}_j$ & Intrinsic projected vector field when $W=X_j$, with $X_j\in \mathrm{Lie}(\mathbb W)$. See \cref{prop:CompatibleCoordinates}. & \\
	
	$D^{\phi}_{\square}\phi$ & The vector field $D^{\phi}_{\square}$ acting on the function $\phi$. &\\ 
	
	$\mathscr{L}^{n}$ &  $n$-dimensional Lebesgue measure. &\\ 
	
	$\mathscr{S}^n$ & $n$-dimensional spherical Hausdorff measure. &\\ 
\end{tabular}
 \section{Introduction}
\subsection{An historical account of the notion of $C^1_{\rm H}$-surface in Carnot groups} 
	 In these last twenty years there has been an increasing interest in a fine study of parametrized intrinsically regular surfaces in sub-Riemannian settings. The search for a good such notion was motivated by a negative result obtained in \cite{AK99}. Indeed, in the reference the authors show that the sub-Riemannian Heisenberg group $\mathbb H^1$ is not $k$-rectifiable in Federer's sense \cite{Fed69}, for every $k\geq 2$.
	
	A notion of intrinsic $C^1$ regular surface was firstly introduced and studied in \cite{FSSC01}, and then in \cite{FSSC03a} in arbitrary Carnot groups $\mathbb G$. Initially, the authors only took $C^1_{\rm H}$-hypersurfaces into account. A first step toward a general definition of $C^1_{\rm H}$-surfaces in arbitrary codimensions was done in \cite[Definition~3.1, Definition~3.2]{FSSC07} in the setting of Heisenberg groups $\mathbb H^n$. Then a general notion of  $(\mathbb G,\mathbb M)$ regular surface, where $\mathbb G$ and $\mathbb M$ are Carnot groups, was proposed by Magnani in \cite[Definition~3.5]{Mag06}. According to the latter definition, a $(\mathbb G,\mathbb M)$ regular surface is locally the zero level set of an $\mathbb M$-valued $C^1_{\rm H}$-function defined on an open subset of $\mathbb G$ and whose intrinsic Pansu differential $\de_{\rm P}\!f$ is surjective. 
	
	A first natural question one could try to answer is whether it is possible to (locally) write a $C^1_{\rm H}$-surface as an intrinsic graph of a function. An intrinsic graph in a Carnot group $\G$ is the set of points of the form $p\cdot\phi(p)$, given a function $ \phi\colon  U\subseteq \sW\to \sL$, where $\sW$ and $\sL$ are homogeneous and complementary subgroups, namely $\G=\sW\cdot\sL$, and $\sW\cap\sL=\{e\}$.
	The answer to this question is affirmative for $C^1_{\rm H}$-hypersurfaces. Moreover the graphing function is intrinsically Lipschitz according to the definition of \cite{FSSC06, FS16}, while it is in general neither Euclidean Lipschitz nor Lipschitz with respect to any sub-Riemannian distance, see \cite[Example~3.3 \& Proposition~3.4]{FSSC06}.
	
	A more general implicit function theorem was proved by Magnani in \cite[Theorem~1.4]{Mag13}. This theorem holds for arbitrary $(\mathbb G,\mathbb M)$ regular surfaces with the additional property that ${\rm Ker}(\de_{\rm P}\!f(x))$ has a complementary subgroup in $\mathbb G$, where $x$ is the point around which we want to parametrize the surface. From \cite[Eq.\,(1.8)]{Mag13} it follows that this parametrization is intrinsically Lipschitz. The validity of the implicit function theorem leads the way to a very general definition of $(\mathbb G,\mathbb M)$ regular sets for $\mathbb G$, where $\mathbb M$ is just a homogeneous group, given in \cite[Definition~10.2]{Mag13}. We will not deal with objects at this level of generality, but we refer the interested reader to \cite[Sections~10,11,12]{Mag13}. The class of intrinsically regular surfaces is also studied in \cite{JNGV20}, where area and coarea formulae are proved. For an alternative proof of the implicit function theorem, one can also see \cite[Section~2.5]{JNGV20}. 
	
	We will mainly deal with co-horizontal $C^1_{\rm H}$-surfaces, that have been studied in \cite{Koz15, DiDonato18, Cor19}, see \cref{def:CoAbelian}. 
	\begin{defi}[Co-horizontal $C^1_{\rm H}$-surface]\label{def:CoAbelianIntro}
		Let $\G$ be a Carnot group, and let $k\in \mathbb N$. We say that $\Sigma\subset \G$ is a \emph{co-horizontal $C^1_{\rm H}$-surface of codimension $k$} if, for any $p\in \Sigma$, there exist a neighborhood $ U$ of $p$ and a map $f\in C^1_{\rm H}( U;\R^k)$ such that
		\begin{equation*}
		\Sigma\cap  U=\{g\in U: f(g)=0\},
		\end{equation*}
		and the Pansu differential $\de_{\rm P}\!f(p)\colon \mathbb G\to \R^k$ is surjective.
		
		If, morevoer, the subgroup $\mathrm{Ker}(\de_{\rm P}\! f(p))$ admits a complementary subgroup, that is horizontal, we say that $\Sigma$ is a {\em co-horizontal $C^1_{\rm H}$-surface with complemented tangents}. We call $\mathrm{Ker}(\de_{\rm P}\! f(p))$ the {\em homogeneous tangent space to $\Sigma$ at $p$}. If the previous holds, the homogeneous subgroup at $p$ is independent of the choice of $f$, see \cite[Theorem 1.7]{Mag13}.
	\end{defi}
	\noindent\textbf{Uniform intrinsic differentiability of the parametrizing function}.  A fine study on the regularity of the parametrizing function of a $C^1_{\rm H}$-surface has been initiated in \cite{ASCV06} in the setting of Heisenberg groups $\mathbb H^n$, for the class of $C^1_{\rm H}$-hypersurfaces. For this study in arbitrary ${\rm CC}$-spaces, see also \cite{CM06}. In \cite{ASCV06}, the authors introduced the notion of uniform intrinsic differentiability. In this paper we abbreviate ``intrinsically differentiable'' and ``uniformly intrinsically differentiable'' with ID and UID, respectively. 
	
	From the analytic viewpoint, the  notion of (U)ID is defined in a translation invariant way, mimicking the Euclidean notion of derivative. For the sake of exposition, we here recall the definition of ID at the identity. Then, building on this definition, one can define the notion of (U)ID at any point by means of translations, see \cref{def:PhiQ}, and \cref{defiintrinsicdiff}. In the following, $\|\cdot\|$ is a homogeneous norm on $\mathbb G$.
	
	\begin{defi}[Intrinsic differentiability]\label{defiintrinsicdiffIntro}
		Let $\mathbb G$ be a Carnot group, with identity $e$, and a splitting $\mathbb G=\mathbb W\cdot\mathbb L$. Let $e\in U\subseteq \mathbb W$ be a relatively open subset, and $\phi\colon U\subseteq \mathbb W\to\mathbb L$ a function with $\phi(e)=e$. 
		
		We say that ${\phi}$ is {\em intrinsically differentiable at $e$} if there exists an intrinsically linear map $\de^{\phi}\!\phi_{e}\colon\mathbb W\to\mathbb L$ such that
		\begin{equation*}
		\lim_{\varrho\to 0}\left(\sup\left\{\frac{\|\de^{\phi}\!\phi_{e}(b)^{-1}\cdot{\phi}(b)\|}{\|b\|}: b\in U, 0<\|b\|<\varrho\right\}\right)= 0,
		\end{equation*}
		where we say that a function is {\em intrinsically linear} if its intrinsic graph is a homogeneous subgroup. The function $\de^{\phi}\!\phi_{e}$ is called {\em intrinsic differential} of $\phi$ at $e$.
	\end{defi}
	Building upon the implicit function theorem, the authors in \cite{ASCV06} prove that in $\mathbb H^n$ the graphing map $\phi$ for a $C^1_{\rm H}$-hypersurface is UID. The idea behind this implication is the following: a function $f\in C^1_{\rm H}$ not only has continuous derivatives, but also its horizontal gradient $\nabla_{\rm H}f$ uniformly approximates $f$ at first order, see \cite[Theorem 1.2]{Mag13}, and \cite[Proposition~2.4]{JNGV20}. This notion is often referred to as strict differentiability. This fact has a strong analogy with the Euclidean setting. Indeed, in the Euclidean framework, a function $f$ with continuous partial derivatives is Fréchet-differentiable, and the proof relies on a use of a mean value inequality, that is exactly what one finds in \cite[Theorem 1.2]{Mag13}, and \cite[Lemma 4.2]{ASCV06}. Eventually, the uniform differentiability of $f$ translates into the uniform intrinsic differentiability of $\phi$.
	
	The fact that the graphing function is {\rm UID} was proved in the case of co-horizontal $C^1_{\rm H}$-surfaces in $\mathbb H^n$ \cite{AS09}, and more in general for co-horizontal $C^1_{\rm H}$-surfaces with complemented tangents in any Carnot groups, in \cite{DiDonato18}. The inverse implication, i.e., the fact that the graph of a UID function is a $C^1_{\rm H}$-surface, was firstly shown to be true in \cite{ASCV06, AS09} in the setting of $\mathbb H^n$, and lately generalized in \cite{DiDonato18} for arbitrary Carnot groups $\mathbb G$ to functions with horizontal target, see our \cref{prop:UidC1H} for a precise statement. Notice that the lack of generality in the statement, namely, the fact that one restricts the target to be horizontal, is due to the fact that a generalized version of Whitney's extension theorem is not known to be true.

\subsection{Main theorems}

\textbf{Definitions \& statements}. For the notation we refer to \cref{sec:Prel}. Given an arbitrary Carnot group $\mathbb G$ of step $s$, with layers $V_i$, with $1\leq i\leq s$, we fix a splitting $\mathbb G=\mathbb W\cdot\mathbb L$, where $\mathbb L$ is horizontal. We denote by $\|\cdot\|$ a (fixed) homogeneous norm on $\mathbb G$. 

We first aim at characterizing the uniform intrinsic differentiability of a map ${\phi}\colon U\subseteq \mathbb W\to\mathbb L$ defined on a relatively open subset $ U\subseteq \mathbb W$. This is done by means of a correct Lipschitz/H\"older property of $\phi$ along the integral curves of projected vector fields on $ U\subseteq\mathbb W$, that we define here. See also \cref{def:PhiJ}.
\begin{defi}[Projected vector fields]\label{def:PhiJIntro}
	Let $U$ be a relatively open subset of $\mathbb W$. Given a continuous function ${\phi}\colon U\subseteq \mathbb W\to \mathbb L$, for every $W\in \mathrm{Lie}(\mathbb W)$, we take $D^{\phi}_W$ as the continuous vector field on $ U$ defined by
	\begin{equation}\label{eqn:DefinitionOfDjIntro}
	D^{\phi}_W(p)\coloneqq\left(\de\!\pi_{\mathbb W}\right)_{p\cdot\phi(p)}W_{p\cdot\phi(p)}. \qquad \forall p\in U,
	\end{equation}
	where $\pi_{\mathbb W}$ is the projection on $\mathbb W$ associated to the splitting $\mathbb G=\mathbb W\cdot \mathbb L$. Notice that if $\phi$ is $C^1$, then $D^{\phi}$ is a $C^1$ vector field.
\end{defi}

We define the notion of broad* regularity, mimicking the definition in \cite{ASCV06, BSC10b}, and we then introduce the notion of vertically broad* h\"older regularity, see \cref{defbroad*}, and \cref{def:vertically*Holder}, respectively. 

\begin{defi}[Broad* regularity]\label{defbroad*Intro}
	Let $\mathbb G$ be a Carnot group, with splitting $\mathbb G=\mathbb W\cdot\mathbb L$, and $\mathbb L$ horizontal. Let $U$ be a relatively open subset in $\mathbb W$,  $\phi\colon U\subseteq \mathbb W\to \mathbb L$ be a continuous function, and let $\omega\colon U\subseteq \mathbb W\to \mathrm{Lin}({\rm Lie}(\mathbb W)\cap V_1; \mathbb L)$ be a continuous function with values in the space of linear maps.
	
	We say that $D^{\phi}\phi=\omega$ {\em in the broad* sense on $ U$} if the following holds: for every $a_0\in U$, there exist a neighborhood $ U_{a_0}\Subset U$ of $a_0$ and $T>0$ such that, for every $a\in  U_{a_0}$ and every $W\in {\rm Lie}(\mathbb W)\cap V_1$ with $\|W\|\leq 1$, there exists
	an integral curve $\gamma\colon[-T,T]\to U$ of $D^{\phi}_W$ such that  $\gamma(0)=a$ and
	\[
	\phi(\gamma(s))^{-1}\cdot \phi(\gamma(t)) = \int_s^t \omega(\gamma(r))(W)\de\!r, \qquad \forall t,s\in [-T,T].
	\]
\end{defi}
\begin{defi}[Vertically broad* h\"older regularity]\label{def:vertically*HolderIntro}
	Let $\mathbb G$ be a Carnot group, with splitting $\mathbb G=\mathbb W\cdot\mathbb L$, and $\mathbb L$ horizontal. Let $U$ be a relatively open subset in $\mathbb W$, and let $\phi\colon U\subseteq \mathbb W\to \mathbb L$ be a continuous function.
	
	We say that $\phi$ is {\em vertically broad* h\"older on $ U$} if the following holds: for every $a_0\in U$ there exist a neighborhood $ U_{a_0}\Subset U$ of $a_0$ and $T>0$ such that for every $a\in U_{a_0}$ and every $W\in {\rm Lie}(\mathbb W)\cap V_d$, with $d>1$ and $\|W\|\leq 1$, there exists
	an integral curve $\gamma\colon [-T,T]\to U$ of $D_W^{\phi}$ such that $\gamma(0)=a$ and
	\[
	\lim_{\varrho\to 0}\left(\sup \left\{\frac{\|\phi(\gamma(s))^{-1}\cdot \phi(\gamma(t))\|}{|t-s|^{1/d}}:t,s\in [-T,T],|t-s|\leq \varrho\right\}\right) = 0,
	\]
	and the limit is uniform on $a\in {U}_{a_0}$ and on $W\in {\rm Lie}(\mathbb W)\cap V_d$, with $d>1$, and $\|W\|\leq 1$.
\end{defi}

We next state our first main result in a free-coordinate fashion. We refer the reader to \cref{thm:MainTheorem} for a coordinate-dependent, though equivalent, statement. We remark that the equivalence (a)$\Leftrightarrow$(b) of the forthcoming \cref{thm:MainTheoremIntro2} is not in the statement of \cref{thm:MainTheorem}, but it is an outcome of \cref{prop:UidC1H}, and the implicit function theorem \cite[Theorem~1.4]{Mag13}.

\begin{theorem}\label{thm:MainTheoremIntro2}
	Let $\mathbb G$ be a Carnot group with splitting $\mathbb G=\mathbb W\cdot\mathbb L$, and $\mathbb L$ horizontal. Let $U$ be a relatively open subset in $\mathbb W$, and let $\phi\colon U\subseteq\mathbb W\to\mathbb L$ be a continuous function. The following facts are equivalent.
	\begin{itemize}
		\item[(a)] ${\rm graph}(\phi)$ is a co-horizontal $C^1_{\rm H}$-surface with homogeneous tangent spaces complemented by $\mathbb L$ (see \cref{def:CoAbelianIntro}). 
		\item[(b)] $\phi$ is uniformly intrinsically differentiable on $ U$ (see \cref{defiintrinsicdiffIntro}).
		\item[(c)] $\phi$ is \textbf{vertically broad* h\"older} on $ U$ (see \cref{def:vertically*HolderIntro}), and there exists a continu\-ous function $\omega\colon U\to \mathrm{Lin}({\rm Lie}(\mathbb W)\cap V_1; \mathbb L)$, such that, for every $a\in  U$, there exist $\delta>0$ and a family of functions $\{\phi_\eps\in C^1( B(a,\delta);\mathbb L):\eps\in (0,1)\}$ such that, for every  $W\in {\rm Lie}(\mathbb W)\cap V_1$, one has
		\begin{equation}\label{eqn:ApproximatingIntro}
		\lim_{\eps\to0}\phi_\eps(p)=\phi(p) \quad\text{and}\quad\lim_{\eps\to0}(D_W^{\phi_\eps}\phi_\eps)(p)=\omega(p)(W) \quad \mbox{uniformly in $p\in B(a,\delta)$}.
		\end{equation}
		\item[(d)] $\phi$ is \textbf{vertically broad* h\"older} on $ U$ (see \cref{def:vertically*HolderIntro}) and there exists a continuous function $\omega\colon U\to \mathrm{Lin}({\rm Lie}(\mathbb W)\cap V_1; \mathbb L)$ such that $D^{\phi}\phi=\omega$ in the broad* sense on $U$ (see \cref{defbroad*Intro}). 
	\end{itemize}
\end{theorem}

Our second main result is an improvement of \cref{thm:MainTheoremIntro2} for step-2 Carnot groups, in the case $\mathbb L$ is one-dimensional. Again, we here give a coordinate-free statement. We refer the reader to \cref{thm:MainTheorem.0} for a coordinate-dependent, though equivalent, statement. We stress that in the forthcoming theorem we are removing the \textbf{vertically broad* h\"older} condition of \cref{thm:MainTheoremIntro2}, and we work with $\mathbb L$ that is \textbf{one-dimensional}. We also stress that in general in the statement of \cref{thm:MainTheoremIntro2} one cannot remove the vertically broad* h\"older condition, see \cref{rem:VerticallyBroadHolderNonSiToglie}, and \cite[Example 4.5.1]{Koz15} for a counterexample in the easiest step-3 group, namely the Engel group. 	

\begin{theorem}\label{thm:MainTheorem.0Intro2}
	Let $\G$ be a Carnot group of step $2$ with a splitting $\mathbb G=\mathbb W\cdot \mathbb L$, and $\mathbb L$ \textbf{one-dimensional horizontal}. Let $U$ be a relatively open subset in $\mathbb W$, and let $\phi\colon U\subseteq\mathbb W\to\mathbb L$ be a continuous function. The following facts are equivalent.
	\begin{itemize}
		\item[(a)] ${\rm graph}(\phi)$ is a $C^1_{\rm H}$-hypersurface with homogeneous tangent spaces complemented by $\mathbb L$ (see \cref{def:CoAbelianIntro}). 
		\item[(b)] $\phi$ is uniformly intrinsically differentiable on $U$ (see \cref{defiintrinsicdiffIntro}).
		\item[(c)] There exists a continuous map $\omega\colon U\to \mathrm{Lin}({\rm Lie}(\mathbb W)\cap V_1; \mathbb L)$, such that, for every $a\in  U$, there exist $\delta>0$ and a family of functions $\{\phi_\eps\in C^1(B(a,\delta);\mathbb L):\eps\in (0,1)\}$ such that, for every  $W\in {\rm Lie}(\mathbb W)\cap V_1$, one has
		\begin{equation*}
		\lim_{\eps\to0}\phi_\eps(p)=\phi(p) \quad\text{and}\quad\lim_{\eps\to0}(D_W^{\phi_\eps}\phi_\eps)(p)=\omega(p)(W) \qquad \mbox{uniformly in $p\in B(a,\delta)$}.
		\end{equation*}
		\item[(d)] There exists a continuous function $\omega\colon U\to \mathrm{Lin}({\rm Lie}(\mathbb W)\cap V_1; \mathbb L)$ such that $D^{\phi} {\phi}=\omega$ in the broad* sense on $U$ (see \cref{defbroad*Intro}).
	\end{itemize}
\end{theorem}

The last main result that we state is an area-formula for uniformly differentiable intrinsic real-valued maps, which we believe has its independent interest. We here give a coordinate-free statement. We refer the reader to \cref{prop2.22}, and \cref{rem:AreaFormula} for a coordinate-dependent statement. 

\begin{prop}\label{prop2.22Intro} 
	Let $\mathbb G$ be a Carnot group of homogeneous dimension $Q$ with a splitting $\mathbb G=\mathbb W\cdot \mathbb L$, with $\mathbb L$ horizontal and one-dimensional. Let $\langle \cdot,\cdot\rangle$ denote a scalar product on the first layer, and let $d$ be a homogeneous and left-invariant distance on $\G$. Consider $U$ a relatively open subset in $\mathbb W$, and a uniformly intrinsically differentiable function $\phi\colon U\subseteq\mathbb W\to\mathbb L$. 
	
	Then, the subgraph $E_{\phi}$ of $\phi$ has locally finite $\G$-perimeter\footnote{The $\mathbb G$-perimeter is defined with respect to the scalar product $\langle\cdot,\cdot\rangle$.} in $ U\cdot \mathbb L$ and its $\G$-perimeter measure $|\partial E_{\phi}|_{\G}$ is given by
	\begin{equation}\label{eqn:PerimeterSubgraphIntro}
	|\partial E_{\phi}|_{\G}( V) =\int _{\Phi^{-1}( V)} \sqrt{1+  |\de^{\phi}\!\phi|^2} \,  \de\!\mathscr{S}^{Q-1}\res \mathbb W, \qquad \mbox{for every Borel set $ V\subseteq U\cdot \mathbb L$},
	\end{equation}
	where $\Phi$ denotes the graph map of $\phi$, $\de^{\phi}\!\phi$ is the intrinsic differential of $\phi$, and $\mathscr S^{Q-1}\res \mathbb W$ is the spherical Hausdorff measure of dimension $Q-1$ restricted to $\mathbb W$. 
	
	Moreover, the reduced boundary of $ E_{\phi}$ coincides with ${\rm graph}(\phi)$, it is a $C^1_{\rm H}$-hypersurface, and there exists a positive measurable function $\beta$ on ${\rm graph}(\phi)$, only depending on the tangent to the hypersurface and on the homogeneous distance $d$, such that 
	\begin{equation}\label{eqn:SpiritIntegration}
	\int_{ V}\beta\de\! \mathscr S^{Q-1}\res {\rm graph}(\phi) = \int _{\Phi^{-1}( V)} \sqrt{1+  |\de^{\phi}\!\phi|^2} \,  \de\!\mathscr{S}^{Q-1}\res \mathbb W, \quad \mbox{for every Borel set $ V\subseteq U\cdot \mathbb L$}.
	\end{equation}
\end{prop} 

	\textbf{Comments on the statements}. We point out that in \cref{defbroad*Intro}, and \cref{def:vertically*HolderIntro} we give coordinate-free definitions of the broad* conditions, while in \cref{defbroad*}, and \cref{def:vertically*Holder}, we will give apparently weaker definitions, choosing an adapted basis. Nevertheless the broad* condition and the vertically broad* h\"older condition, when coupled together, are independent on the choice of the basis, see \cref{rem:Broad*InCoordinates}, and \cref{rem:VerticallyBroad*InCoordinates}.
	
	We comment on the statement of \cref{thm:MainTheoremIntro2}, and we refer the reader to the introduction of \cref{sec:Theorems} for a more detailed discussion. We notice that (a)$\Leftrightarrow$(d) is
	\cite[Theorem 4.3.1]{Koz15}. 
	In \cite{Koz15} the proof of this latter fact is heavily based on the characterization of co-horizontal $C^1_{\rm H}$-surfaces in terms of uniform convergence to Hausdorff tangents, see \cite[Theorem 3.1.12]{Koz15}. We give a self-contained different proof of this equivalence, with analytic flavor. Namely, we first show (b)$\Leftrightarrow$(d) in \cref{thm:MainTheoremIntro2}, whose proof is based on ideas coming from \cite{ASCV06} and \cite{DiDonato18}, and thus, as a corollary, we eventually recover \cite[Theorem 4.3.1]{Koz15} by using the latter equivalence and (a)$\Leftrightarrow$(b). 
	
	The approximating condition in \eqref{eqn:ApproximatingIntro} of item (c) of \cref{thm:MainTheoremIntro2} could be interpreted as a weak formulation of the equality $D^{\phi}\phi=\omega$ on $ U$, and it is the one that was first proposed and studied in \cite{ASCV06}, see \cref{rem:PointwiseAndLocalApproximability} for a detailed discussion about this condition. Indeed, in \cite{ASCV06}, in the case $\mathbb G=\mathbb H^n$ and $\mathbb L$ one-dimensional, the equivalence (b)$\Leftrightarrow$(c) of \cref{thm:MainTheoremIntro2} has been proved, even in the stronger version obtained by removing the vertically broad* h\"older regularity, see \cite[Theorem 5.1]{ASCV06}. We stress that the fact that $D^{\phi}\phi=\omega$ holds in the sense of distributions on $U$, that is part of \cite[Theorem 5.1]{ASCV06}, follows in general from the argument of item (c) of our \cref{prop2.22}.

	\vspace{0.2cm}
	We comment on the statement of \cref{thm:MainTheorem.0Intro2}, and we refer the reader to the introduction of \cref{sec:Step2} for a more detailed discussion on the idea behind the proof. First, we notice that the main difference between this statement and the one in \cref{thm:MainTheoremIntro2} is in the fact that, in all the equivalences, we are able to drop the vertically broad* h\"older regularity on $\phi$. The equivalence (b)$\Leftrightarrow$(c) of \cref{thm:MainTheorem.0Intro2} results in being a generalization of \cite[Theorem 5.1]{ASCV06} to all step-2 Carnot groups. 
	
	The equivalence (b)$\Leftrightarrow$(d) generalizes the result in \cite[Theorem 1.2]{BSC10b}, if one also takes item (d) of \cref{prop2.22} into account, in order to explicitly write the intrinsic normal of ${\rm graph}(\phi)$ in terms of the intrinsic derivatives of $\phi$.

	\vspace{0.2cm}
	We comment on the statement of \cref{prop2.22Intro}. First of all, let us notice that \eqref{eqn:SpiritIntegration} simply comes from \eqref{eqn:PerimeterSubgraphIntro} and the general area formula in \cite{Mag17}, see \cref{rem:AreaFormula} for details. Let us notice that for some particular choices of the homogeneous distance on $\mathbb G$, i.e., when it is vertically symmetric, the function $\beta$ is constant, thus simplifying \eqref{eqn:SpiritIntegration}, see again \cref{rem:AreaFormula}. 
	
	We also mention that in orthonormal coordinates we can explicitly compute the intrinsic normal, see item (d) of \cref{prop2.22}, thus generalizing the formulas already proved in Heisenberg groups and in Carnot groups of step 2 in \cite{ASCV06, DiDonato18}.
	
	We finally remark that a formula in the spirit of \eqref{eqn:SpiritIntegration} has been recently obtained in \cite{CM20} for parametrized co-horizontal $C^1_{\rm H}$-surfaces with complemented tangents of arbitrary codimensions in $\mathbb H^n$ (see \cite[Theorem 4.2]{CM20}), building on an upper blow-up thorem, see \cite[Theorem 1.1]{CM20}. Very recently, in \cite[Theorem 1.1]{JNGV20}, a general area formula for $C^1_{\rm H}$-surfaces has been proved in great generality. In \cite{JNGV20} the area element is left implicit, depending only on the tangent to the surface and on the homogeneous distance on the group. We stress that, with our formula \eqref{eqn:SpiritIntegration}, we explicitly write the area element in terms of the intrinsic derivatives of $\phi$, in the case the target is one-dimensional. For an area-formula for intrinsically Lipschitz functions in $\mathbb H^n$, one can also see \cite{CMPSC14}.
	
	\vspace{0.2cm}
	\noindent\textbf{Geometric characterizations of intrinsic differentiability}. The notion of (U)ID has a geometric meaning. Indeed, given $U\subseteq \mathbb W$ open, a function $\phi\colon U\subseteq\mathbb W\to\mathbb L$ is ID at $w_0\in U$ if and only if the Hausdorff tangent to ${\rm graph}({\phi})$ at $w_{0}\cdot{\phi}(w_0)$ is a homogeneous subgroup that is complementary to $\mathbb L$, see \cref{rem:InvarianceOfIdByTranslations}. 
	
	We stress that, at least in the case $\mathbb L$ is horizontal, the previous convergence to the tangent is uniform if and only if $\phi$ is UID, see \cite[Theorem~3.1.1]{Koz15}, and \cite[Theorem~3.1.12]{Koz15}. The proof of these statements are rather involved and based on the so-called four cones Theorem, see \cite{BK14}. We remark that we will not use this particular uniformity result throughout the paper. 
	
	We point out that the mere existence of Hausdorff tangents for $C^1_{\rm H}$ regular surfaces - without any information on the uniformity of convergence -  has been proved in great generality also in \cite[Theorem 1.7]{Mag13}, and in \cite[Lemma 2.14]{JNGV20}.
	
	Now a natural question can be raised. \textbf{Is it true that an ID function $\phi$ with continuous intrinsic gradient $\de^{\phi}\!\phi$ is UID?} Taking the geometric interpretation into account, the question can be reformulated, at least in the category of co-horizontal $C^1_{\rm H}$-surfaces: \textbf{is it true that, if a co-horizontal ${\rm graph}({\phi})$ has continuously varying Hausdorff tangents, then it is a co-horizontal $C^1_{\rm H}$-surface}? If true, this would be the counterpart of an already known result in the Euclidean setting that goes back to the beginning of twentieth century. We refer the reader to \cite[Proposition 2.1]{BNG14} and references therein for an historical account of the problem.
	
	The answer to the previous question is affirmative in Heisenberg groups $\mathbb H^n$, see \cite[Theorem 4.95]{SC16}, and \cite[Theorem 1.4]{Cor19}. In this paper we obtain a new result in this direction. We prove that the answer is affirmative also for hypersurfaces in every step-2 Carnot group, see (b)$\Rightarrow$(a) in \cref{thm:MainTheorem.0}, thus generalizing \cite[Theorem 4.95]{SC16}.
	
	We give also a partial affirmative answer for arbitrary Carnot groups, by requiring the additional hypothesis of the vertically broad* h\"older condition on $\phi$, see \cref{coroll:IDContinuousImpliesUid}. This weaker implication might not be so satisfactory. Indeed, the intrinsic differentiability (even if it is not uniform) by itself already implies a $1/d$-little H\"older continuity on integral curves of the vector fields $D^{\phi}_W$, with $W\in {\rm Lie}(\mathbb W)\cap V_d$. Nevertheless, this little H\"older continuity, a priori, might not be uniform, see \cref{prop:DerivativeOfIntegralCurves}, \cref{rem:StrategyIdGradContinuousUid}, and \cref{exa:Serapioni}.

	\subsection{Structure of the paper} In \cref{sec:Prel} we introduce the common terminology and notation we use throughout the paper. We introduce Carnot groups, little H\"older functions, intrinsic submanifolds, intrinsically Lipschitz functions, (uniformly) intrinsically differentiable functions and we describe their basic properties and relations.
	
	\noindent In \cref{sec:3} we introduce the projected vector fields and we study their basic properties: in particular, we show some invariance properties that will be crucial in the proof of the main theorems. We also show how the (uniformly) intrinsic differentiability affects metric properties along integral curves of the projected vector fields.
	
	\noindent In \cref{sec:Theorems} we prove the main results \cref{thm:MainTheoremIntro2} and \cref{prop2.22Intro} we discussed above.
	
	\noindent In \cref{sec:Applications} we construct examples and apply our results also to obtain different proofs of particular cases of theorems already contained in the literature. We refer the reader to the beginning of \cref{sec:Applications} for a more detailed discussion.
	
	\noindent In \cref{sec:Step2} we give the proof of \cref{thm:MainTheorem.0Intro2}. 
	
	\vspace{0.5cm}
	\textbf{Acknowledgments}. We are grateful to Katrin F\"assler for several discussions around the topic of the paper. We warmly thank Raul Serapioni for several stimulating discussions and for having shared and discussed with us the examples in \cref{exa:Serapioni} and \cref{rem:0ischaracteristic}. We also thank Sebastiano Nicolussi Golo for stimulating discussions, in particular regarding the last part of \cref{rem:InvarianceOfIdByTranslations}. We finally thank Francesco Serra Cassano and Davide Vittone for having discussed with us the content of \cref{rem:PointwiseAndLocalApproximability}.

\section{Preliminaries}\label{sec:Prel}

\subsection{Carnot groups}
A Carnot group $\mathbb G$ is a connected and simply connected Lie group, whose Lie algebra $\mathfrak g$ is stratified. Namely, there exist subspaces $V_1,\dots, V_s$ of the Lie algebra $\mathfrak g$ such that
\[
\mathfrak g = V_1\oplus \dots\oplus V_s,\qquad [V_1,V_j]=V_{j+1} \quad\forall j=1,\dots,s-1,\qquad [V_1,V_s]=\{0\}.
\]
 The integer $s$ is called {\em step} of the group $\G$, while $m\coloneqq\dim (V_1)$ is called \emph{rank} of $\G$. We call $n\coloneqq \mathrm{dim}\mathbb G$ the topological dimension of $\mathbb G$. We denote by $e$ the identity element of $\G$. 
  
 It is well known that the exponential map $\exp\colon\mathfrak g\to \G$ is a diffeomorphism. We call 
 \begin{equation}\label{eqn:HorizontalBundle}
 \mathbb V_1\coloneqq \exp(V_1),
 \end{equation}
 the {\em horizontal bundle} of $\mathbb G$.
 We write
 \begin{equation}\label{eqn:ProjectionOnV1}
 \pi_{\mathbb V_1}\coloneqq \exp\circ \pi_{V_1}\circ\exp^{-1},
 \end{equation}
 to denote the {\em projection on the horizontal bundle $\mathbb V_1$}, where $\pi_{V_1}$ is the linear projection in $\mathfrak g$ onto $V_1$.
 
 Every Carnot group has a one-parameter family of {\em dilations} that we denote by $\{\delta_\lambda: \lambda >0\}$. These dilations act on $\mathfrak g$ as 
 $$
 (\delta_{\lambda})_{|_{V_i}}=\lambda^{i}(\mathrm{id})_{|_{V_i}}, \qquad \forall \lambda>0, \quad \forall 1\leq i\leq s,
 $$
 and are extended linearly. We will indicate with $\delta_{\lambda}$ both the dilations on $\mathfrak g$ and the group automorphisms corresponding to them via the exponential map. 
 
 We fix a scalar product $\langle\cdot,\cdot\rangle$ in $V_1$, that can be extended left-invariantly on the horizontal bundle $\mathbb V_1= \exp(V_1)$, and a homogeneous norm $\|\cdot\|$ on $\G$. We recall that $\|\cdot\|$ is a {\em homogeneous norm} on $\mathbb G$ if 
 \begin{equation}\label{eqn:HomogeneousNorm}
 \begin{split}
 \|g\|&=0 \qquad \mbox{if and only if $g=0$}, \\
 \|\delta_{\lambda}g\|&=\lambda\|g\|, \qquad \forall \lambda>0, \quad \forall g\in\mathbb G,\\
 \|g\|&=\|g^{-1}\|, \qquad \forall g\in\mathbb G.
 \end{split}
 \end{equation}
 Sometimes we will also call the homogeneous norm $\|\cdot\|_{\mathbb G}$. We also fix on $\mathbb G$ a left-invariant $\delta_{\lambda}$-homogeneous distance $d$ and we denote by $B(g,r)$ (respectively $\overline{B(g,r)}$) the open (respectively closed) balls of center $g\in\mathbb G$ and radius $r>0$ according to this distance. We next give the definition of complementary subgroups. 
 
\begin{defi}[Complementary subgroups]\label{def:ComplementarySubgroups}
Given a Carnot group $\mathbb G$, we say that two subgroups $\mathbb W$ and $\mathbb L$ are \emph{complementary subgroups} in $\G$ if they are {\em homogeneous}, i.e., closed under the action of $\delta_{\lambda}$ for every $\lambda>0$, $\G=\mathbb W\cdot \mathbb L$ and $\mathbb W\cap \mathbb L=\{e\}$. 
\end{defi}

We say that the subgroup $\mathbb L$ is {\em horizontal and $k$-dimensional}, if there exist linearly independent $X_1,\dots$, $X_k \in V_1$ such that $\mathbb L=\exp({\rm span}\{X_1,\dots, X_k\})$.
Given $\mathbb W$ and $\mathbb L$ two complementary subgroups, we denote the {\em projection maps} from $\G$ onto $\mathbb W$ and onto $\mathbb L$ by $\pi_{\W}$ and $\pi_{\mathbb L}$, respectively. Defining $g_\sW\coloneqq\pi_\sW g$ and $g_\sL \coloneqq \pi_\sL g$ for any $g\in \G$, one has
\begin{equation}\label{eqn:ComponentsSplitting}
g=(\pi_{\W} g)\cdot(\pi_{\mathbb L} g)= g_{\sW}\cdot g_{\sL},
\end{equation}
and, whenever $\mathbb W$ is normal (for example this is true when $\mathbb L$ is horizontal), we have
\[
(g\cdot h)_{\sL}=g_{\sL}\cdot h_{\sL}, \qquad (g\cdot h)_{\sW}=g_{\sW}\cdot\left(g_{\sL}\cdot h_{\sW}\cdot\left(g_{\sL}\right)^{-1}\right),\qquad \forall g,h\in \G.
\]
\begin{rem} If $\sW$ and $\sL$ are complementary subgroups of $\G$ and $\sL$ is one-dimensional, then it is easy to see that $\sL$ is horizontal. For the sake of clarity, we will always write $\sL$ horizontal and one-dimensional even if one-dimensional is technically sufficient.
\end{rem}

Let us set $m_0\coloneqq 0$ and $m_j\coloneqq \dim{V_j}$ for any $j=1,\dots,s$. We recall that $m=m_1$. Let us define $n_0\coloneqq 0$, and $n_j\coloneqq \sum_{\ell=1}^j m_{\ell}$. The ordered set $(X_1,\dots, X_n)$ is an {\em adapted basis} for $\mathfrak g$ if the following facts hold. 
\begin{itemize}
	\item[(i)] The vector fields $X_{n_j+1},\dots,X_{n_{j+1}}$ are chosen among the iterated commutators of order $j$ of the vector fields $X_1,\dots, X_m$, for every $j=0,\dots,s-1$.
	\item[(ii)] The set $\{X_{n_j+1},\dots,X_{n_{j+1}}\}$ is a basis for $V_{j+1}$ for every $j=0,\dots, s-1$.
\end{itemize} 
If we fix an adapted basis $(X_1,\dots,X_n)$, and $\ell \in \{1,\dots,n\}$, we define the {\em holonomic degree of $\ell$} to be the unique $j^*\in\{1,\dots,s\}$ such that $n_{j^*-1}+1\leq \ell \leq n_{j^*}$. We denote $\deg\ell \coloneqq j^*$ and we also say that $j^*$ is the {\em holonomic degree of $X_{\ell}$}, i.e., $\deg(X_{\ell})\coloneqq j^*$.

\begin{defi}[Exponential coordinates]\label{def:coordinateconletilde}
	Let $\G$ be a Carnot group of dimension $n$ and let $(X_1,\dots,X_n)$ be an adapted basis of its Lie algebra. We define the {\em exponential coordinates of the first kind} associated with $(X_1,\dots,X_n)$ by the map $F\colon \R^n\to \G$ defined by
	\[
	F(x_1,\dots,x_n)\coloneqq\exp\left( x_1X_1+\ldots+x_nX_n\right).
	\]
	It is well known that $F$ is a diffeomorphism from $\R^n$ to $\G$. We will often need to consider maps in exponential coordinates. To avoid inconvenient notation we will use the following conventions.
	\begin{itemize}
		\item If $\widetilde U\subset\G$, then $U\coloneqq F^{-1}(\widetilde U)$.
		\item If $U\subseteq \R^n$, then $\widetilde U\coloneqq F(U)$.
		\item If $\sW$ and $\sL$ are complementary subgroups of $\G$ and $\mathbb L$ is horizontal and $k$-dimensional, we may assume that $\sL=\exp({\rm span}\{X_1,\dots,X_k\})$ and $\sW=\exp({\rm span}\{X_{k+1},\dots,X_n\})$ for an adapted basis $(X_1,\dots,X_n)$. Therefore $F$ is one-to-one from $\R^k\times\{0_{\R^{n-k}}\}$ onto $\sL$ and also from $\{0_{\R^k}\}\times \R^{n-k}$ onto $\sW$.
		\item If $\widetilde U\subseteq \mathbb W$ and $\widetilde\phi\colon\widetilde U\to \mathbb L$ is a function, then $\phi\colon U\subseteq \R^{n-k}\to\R^k$ denotes the composition of $\widetilde \phi$ with $F$, namely $\phi\coloneqq F^{-1}\circ \widetilde \phi\circ F$.
		\item If $\phi\colon U\subseteq \R^{n-k}\to \R^k$ is a function, then we denote by $\widetilde \phi\colon \widetilde U\subseteq \mathbb W\to \mathbb L$ the map defined by $\widetilde \phi\coloneqq F\circ\phi\circ F^{-1}$.
		\item If $p\in \G$ and $j=1,\dots, s$, then $p^j\in \R^{m_j}$ is the vector of the coordinates of $p$ in the $j^{\rm th}$ layer, namely $p^j\coloneqq (F^{-1}(p)_{n_{j-1}+1},\dots, F^{-1}(p)_{n_j})$.  
		\item If $p\in \G$ and $j=1,\dots, s$, then $\|p^j\|_{m_j}$ denotes the Euclidean norm of $p^j$ in $\R^{m_j}$.
	\end{itemize}
\end{defi}

 It is well known that all the homogeneous norms on $\mathbb G$ are bi-Lipschitz equivalent. Thus, when it will be convenient in the proofs, we work with the {\em anisotropic norm} that in exponential coordinates reads as $$\|(x_1,\dots,x_n)\|_{\mathbb G}=\sum_{\ell=1}^n |x_{\ell}|^{1/\mathrm{deg} \ell}.$$
 We remark that a slight variation of the previous homogeneous norm gives rise to a homogeneous norm that induces a left-invariant homogeneous distance, see \cite[Theorem~5.1]{FSSC03a}.

 We recall that the {\em homogeneous degree} of the monomial $x_1^{a_1}\cdot\dots\cdot x_n^{a_n}$ in exponential coordinates, is $\sum_{\ell=1}^n a_{\ell}\cdot \deg \ell $. 
  
 For the expression of the operation on the group $\mathbb G$ in exponential coordinates we refer to \cite[Proposition 2.1]{FSSC03a}. In the following result we point out a useful estimate for the norm of the conjugate. 
 
\begin{prop}[{\cite[Lemma 3.12]{FS16}}]\label{lemma333FS}
	There exist $\mathcal P= (\mathcal P^1,\dots,\mathcal P^s)\colon \G\times \G\to \R^{m_1}\times\cdots\times\R^{m_s}$ such that, for every $p,q \in \G$, one has
	\begin{equation}\label{eq:polinomi}
	F^{-1}(p^{-1}q p)= F^{-1}(q)+\mathcal{P} (p,q),
	\end{equation}
	where $\mathcal{P} ^1=0$ and, for each $i=2,\dots s$, $\mathcal{P} ^i $ is a vector valued $\delta_\lambda $-homogeneous polynomial of degree $i$. Moreover, for any bounded set $B  \subset \G$, there exists $C\coloneqq C (B,\G)>0$ such that 
	\begin{equation*}
	|\mathcal{P} ^i (p,q) | \leq C \bigl( \|q^1\|_{m_1}+\dots + \|q^{i-1}\|_{m_{i-1}} \bigl),
	\end{equation*}
	for every $p, q \in B$ and every $i=2,\dots , s$.
\end{prop}

\begin{rem}
	With a little abuse of notation, \eqref{eq:polinomi} will be always written as
	\begin{equation}\label{eq:abusodinotazione}
	p^{-1}qp=q+\mathcal P(p,q),
	\end{equation}
	where the identification of $\G$ with $\R^n$ has to be understood via exponential coordinates. Notice, however, that if one chooses a different diffeomorphism between $\R^n$ and $\G$, such as exponential coordinates of the second kind or of mixed type, the polynomial $\mathcal P$, the constant $C$ and the components $p^j$ have to be changed accordingly.
\end{rem}

\subsection{Little H\"older continuous functions}
We introduce and discuss the notion of $\alpha$-little H\"older continuous function.
\begin{defi}[little H\"older functions, \cite{Lunardi}]\label{big3.3.11}
	Let $U\subseteq\R^n$ be an open set. We denote by $h^\alpha (U;\R^k)$ the set of all  \emph{$\alpha$-little  H\"older continuous} functions of order $0<\alpha<1$, i.e., the set of maps  $\phi \in C(U;\R^k)$ satisfying
	\begin{equation}\label{equationluna}
	\lim_{r\to 0} \left(\sup \Biggl\{ \, \frac{|\phi (b')-\phi (b)|}{|b'-b|^{\alpha }} : \, b,b'\in U \, , \, 0<|b'-b| <r \, \Biggl\}\right)=0.
	\end{equation}
	We also define $h^\alpha _{{\rm loc}} (U;\R^k)$ the set of all functions $\phi \in C(U;\R^k)$ such that $\phi \in h^\alpha (U';\R^k)$ for any open set $U' \Subset  U $.
\end{defi}

The following example is, in some sense, ``pathological''. As it will be clear during the paper, it gives a flavor of the difference between intrinsically differentiable functions and uniformly intrinsically differentiable functions, see \cref{rem:StrategyIdGradContinuousUid}. We thank R.\ Serapioni for having shared this example with us. 
\begin{exa}\label{exa:Serapioni}
	We are going to construct a real-valued function $\phi\colon\mathbb R\to\mathbb R$ such that $\phi\in h^{1/2}_{\rm loc}(\mathbb R\setminus\{0\})$, $\phi\notin h^{1/2}_{\rm loc}(\mathbb R)$, but still it holds 
	\begin{equation}\label{limit}
	\lim_{x\to 0}\frac{|\phi(x)-\phi(0)|}{|x|^{1/2}} = 0.
	\end{equation}
	
	Let us first notice that, for $n\geq 2$, the intervals $I_n\coloneqq [1/n-1/n^3,1/n+1/n^3]$ are mutually disjoint. Let us define, for $n\geq 2$, the functions $\phi_n\colon\mathbb R\to\mathbb R$ as
	\begin{equation}
	\phi_n(x)\coloneqq \begin{cases*}
	n^3\left|x-\frac{1}{n}\right| & if $x\in I_n$, \\
	1        & otherwise.
	\end{cases*}
	\end{equation}
	Notice that for each $n\geq 2$, the function $\phi_n$ is globally Lipschitz. Define $\phi\colon\mathbb R\to\mathbb R$ as
	\[
	\phi(x)\coloneqq |x|\cdot \prod_{n=2}^{+\infty} \phi_n(x).
	\]
	Notice that, being $I_n$ pairwise disjoint for $n\geq 2$, the infinite product is well-defined, since, if $x\notin I_n$, then $\phi_n(x)=1$. Moreover, being each $\phi_n$ globally Lipschitz,  we get that $\phi\in {\rm Lip}_{\rm loc}(\mathbb R\setminus\{0\})$ and thus $\phi \in h^{1/2}_{\rm loc}(\mathbb R\setminus \{0\})$. We now prove that $\phi\notin h^{1/2}_{\rm loc}(\mathbb R)$. In particular this will follow from the fact that, for every compact neighborhood $U$ of the origin, $\phi \notin h^{1/2}(U)$. 
	
	Indeed, by definition of $\phi$, we get  the following equalities
	\begin{equation}\label{eqn:Contr}
	\frac{\left|\phi\left(\frac{1}{n}+\frac{1}{n^3}\right)-\phi\left(\frac{1}{n}\right)\right|}{\left(\frac{1}{n^3}\right)^{1/2}}=\frac{\frac{1}{n}+\frac{1}{n^3}}{\frac{1}{n^{3/2}}}=\frac{n^2+1}{n^{3/2}}, \qquad \forall n\geq 2.
	\end{equation}
	Thus, if $U$ is an arbitrary compact neighborhood of 0, we get that, for every $n$ large enough, one has $[1/n,1/n+1/n^3]\subset U$, and thus  \eqref{eqn:Contr} implies that \eqref{equationluna} cannot hold, because $1/n^3=\left(1/n+1/n^3\right)-1/n\to 0$ but $(n^2+1)/(n^{3/2})\to +\infty$ as $n\to +\infty$. Thus $\phi\not\in h^{1/2}(U)$. 
	
	Finally, by definition of $\phi$, we get that, for $x\neq 0$,  
	\[
	\frac{|\phi(x)-\phi(0)|}{|x|^{1/2}}=|x|^{1/2}\cdot \prod_{n=2}^{+\infty}\phi_n(x),
	\]
	and thus \eqref{limit} holds, because $\prod_{n=2}^{+\infty}\phi_n$ is bounded. 
	
	We remark that, by a little modification of this example, one can replace $1/2$ with any $0<\alpha<1$. 
\end{exa}

\subsection{Intrinsic surfaces, Intrinsically Lipschitz and Intrinsically differentiable functions} In this section we recall the notion of intrinsic graph of a function, and see what happens to the defining map if we translate the graph. Then we recall the definitions of intrinsically Lipschitz and intrinsically differentiable maps. Finally we discuss the notion of co-horizontal $C^1_{\rm H}$-surface. 

 \begin{defi}[Intrinsic graph of a function]\label{def:IntrinsicGraph}
 	Given $\mathbb W$ and $\mathbb L$ two complementary subgroups in $\mathbb G$, and $\widetilde\phi:\widetilde U\subseteq \mathbb W\to\mathbb L$ a function, we denote
 	$$
 	\widetilde\Phi(\widetilde U)=\mathrm{graph}(\widetilde\phi)\coloneqq\{\widetilde\Phi(w):=w\cdot\widetilde\phi( w): w\in\widetilde U\}.
 	$$
 \end{defi}
 \begin{defi}[Intrinsic translation of a function]\label{def:PhiQ}
 	Given $\mathbb W$ and $\mathbb L$ two complementary subgroups of a Carnot group $\mathbb G$ and a map $\widetilde{\phi}\colon\widetilde{U}\subseteq\mathbb W\to\mathbb L$, we define, for every $q\in\mathbb G$,
 	\[
 	\widetilde{U}_q\coloneqq\{a\in\mathbb W: \pi_{\sW}(q^{-1}\cdot a)\in \widetilde{U}\},
 	\]
 	and $\widetilde{\phi}_q\colon\widetilde{U}_q\subseteq \mathbb W\to\mathbb L$ by setting
 	\begin{equation}\label{eqn:Phiq}
 	\widetilde{\phi}_q(a)\coloneqq\left(\pi_{\mathbb L}(q^{-1}\cdot a)\right)^{-1}\cdot \widetilde{\phi}\left(\pi_{\mathbb W}(q^{-1}\cdot a)\right).
 	\end{equation}
 \end{defi}
 \begin{prop}\label{prop:PropertiesOfIntrinsicTranslation}
 	Let $\mathbb W$ and $\mathbb L$ be two complementary subgroups of a Carnot group $\mathbb G$ and let $\widetilde{\phi}\colon\widetilde{U}\subseteq\mathbb W\to\mathbb L$ be a function. Then, for every $q\in\mathbb G$, the following facts hold.
 	\begin{itemize}
 		\item[(a)] ${\rm graph}(\widetilde{\phi}_q)=q\cdot{\rm graph}(\widetilde{\phi})$;
 		\item[(b)] $(\widetilde{\phi}_q)_{q^{-1}}=\widetilde{\phi}$;
 		\item[(c)] If $\mathbb W$ is normal, then $\widetilde U_q=q_{\mathbb W}\cdot\left(q_{\mathbb L}\cdot \widetilde U\cdot q_{\mathbb L}^{-1}\right)$ and
 		\begin{equation*}
 		\widetilde{\phi}_q(a)=q_{\mathbb L}\cdot\widetilde{\phi}(q_{\mathbb L}^{-1}\cdot q_{\mathbb W}^{-1}\cdot a\cdot q_{\mathbb L}),
 		\end{equation*}
 		for any $a\in \widetilde U_q$;
 		\item[(d)] If $q=\widetilde{\phi}(a)^{-1}\cdot a^{-1}$ for some $a\in \widetilde{U}$, then 
 		\begin{equation*}
 		\widetilde{\phi}_q(e)=e.
 		\end{equation*}
 	\end{itemize}
 \end{prop}
 \begin{proof}
 	The proof of (a), directly follows from \eqref{eqn:Phiq}, which yields
 	\begin{equation}\label{eqn:GraphOfPhiQEqualsQGraphOfPhi}
 	a\cdot\widetilde{\phi}_q(a)=q\cdot\pi_{\sW}(q^{-1}\cdot a)\cdot \widetilde{\phi}\left(\pi_{\mathbb W}(q^{-1}\cdot a)\right),\qquad \forall a\in \widetilde U_q.
 	\end{equation}
 	To prove (b), it is enough to apply twice \eqref{eqn:Phiq}.
 	For the proof of (c), decompose $q=q_{\sW}\cdot q_{\mathbb L}$. Then, for every $a\in\widetilde U_q$,
 	\[
 	q^{-1}\cdot a = (q_{\mathbb L}^{-1}\cdot q_{\sW}^{-1}\cdot a\cdot q_{\mathbb L})\cdot q_{\mathbb L}^{-1},
 	\]
 	and whenever $\mathbb W$ is normal one gets 
 	\begin{equation}\label{eqn:ProjectionOfq-1a}
 	\pi_{\mathbb L}(q^{-1}\cdot a)=q_{\mathbb L}^{-1}, \qquad \pi_{\sW}(q^{-1}\cdot a)=q_{\mathbb L}^{-1}\cdot q_{\sW}^{-1}\cdot a\cdot q_{\mathbb L}.
 	\end{equation}
 	As a consequence we get $\widetilde{U}_q=q_{\mathbb W}\cdot\left(q_{\mathbb L}\cdot \widetilde U\cdot q_{\mathbb L}^{-1}\right)$ and, using again \eqref{eqn:Phiq}, we obtain (c).
 	
 	To prove (d), it is enough to evaluate \eqref{eqn:GraphOfPhiQEqualsQGraphOfPhi} in $a=e$ and $q=\widetilde{\phi}(a)^{-1}\cdot a^{-1}$.
 \end{proof}

 We introduce the notion of intrinsically Lipschitz function and state some properties. See \cite[Section 3]{FS16}. 

\begin{defi}[Intrinsic Cone]\label{def:Cone}
	Let $\mathbb W$ and $\mathbb L$ be two complementary subgroups of a Carnot group $\mathbb G$. The {\em intrinsic cone} $C_{\mathbb W,\mathbb L}(q,\alpha)$ of basis $\mathbb W$ and axis $\mathbb L$, centered at $q$ and of opening $\alpha\geq 0$, is defined by
	\[
	C_{\mathbb W,\mathbb L}(q,\alpha) \coloneqq q\cdot \{p\in\mathbb G:\|p_{\mathbb W}\|\leq \alpha \|p_{\mathbb L}\| \}.
	\]
\end{defi}

\begin{defi}[Intrinsically Lipschitz function]\label{def:IntrinsicLipschitz}
	Let $\mathbb W$ and $\mathbb L$ be complementary subgroups of a Carnot group $\mathbb G$. We say that a function $\widetilde\phi\colon\widetilde U \subseteq \mathbb W\to \mathbb L$ is {\em intrinsically $L$-Lipschitz} in $\widetilde U$, with $L>0$, if 
	\[
	C_{\mathbb W,\mathbb L}(p,L^{-1}) \cap \mbox{graph}(\widetilde\phi)=\{p\}, \qquad \forall p\in\mbox{graph}(\widetilde\phi).
	\] 
\end{defi}
\begin{prop}[{\cite[Theorem~3.2 \& Proposition~3.3]{FS16}}]\label{prop:IntrinsicLipschitz}
	Let $\mathbb W$ and $\mathbb L$ be two complementary subgroups in a Carnot group $\mathbb G$ and let $\widetilde\phi\colon\widetilde U\subseteq\mathbb W\to\mathbb L$ be a function. Then the following facts are equivalent.
	\begin{itemize}
		\item[(a)] $\widetilde\phi$ is intrinsically $L$-Lipschitz in $\widetilde U$;
		\item[(b)] $\|\pi_{\mathbb L}\left(p^{-1}\cdot q\right)\|\leq L\|\pi_{\mathbb W}\left(p^{-1}\cdot q\right)\|$ for every $p, q\in{\rm graph}(\widetilde\phi)$;
		\item[(c)] for any $a\in\widetilde U$, setting $q\coloneqq\widetilde\phi(a)^{-1}\cdot a^{-1}$, one has $\|\widetilde\phi_q(b)\|\leq L\|b\|$ for every $b\in\widetilde U_q$, where $\widetilde\phi_q$ and $\widetilde U_q$ are defined in \cref{def:PhiQ}.
	\end{itemize}
	Moreover, for every $a\in\widetilde U$, setting $q\coloneqq\widetilde\phi(a)^{-1}\cdot a^{-1}$, one has that $\widetilde\phi$ is intrinsically $L$-Lipschitz in $\widetilde U$ if and only if $\widetilde\phi_q$ is intrinsically $L$-Lipschitz in $\widetilde U_q$.
\end{prop}

We now define the notion of intrinsically linear function, intrinsically differentiable function and uniformly intrinsically differentiable function. General properties are studied in \cite{FMS14}, see for example \cite[Proposition~3.1.3 \& Proposition~3.1.6]{FMS14}. For the forthcoming definitions and properties of intrinsically differentiable functions we follow also \cite{DiDonato18}.

The notion of intrinsic differentiability was first given in \cite[Definition~4.4]{FSSC06} and then first studied in \cite{ASCV06}, see \cite[Definition~1.1]{ASCV06}. In this last reference the notion of intrinsic differentiability is given in terms of the graph distance.
 We here give a slightly different definition of intrinsic differentiability that is indeed equivalent to ours, by \cite[Proposition~4.76]{SC16}, when $\mathbb W$ is a normal subgroup, that will always be in our case. 

\begin{defi}[Intrinsically linear function]
	Let $\mathbb W$ and $\mathbb L$ be complementary subgroups in $\mathbb G$. Then $\ell: \W \to \mathbb L$ is {\em intrinsically linear} if $\mathrm{graph}(\ell)$ is a homogeneous subgroup of $\mathbb G$.
\end{defi}

\begin{defi}[Pansu differentiability]\label{def:PansuDifferentiability}
	Let $\G$ and $\G'$ be two Carnot groups endowed with left-invariant homogeneous distances $d_\G$ and $d_{\G'}$ and let $\Omega\subseteq \G$ be an open set. A function $f\colon\G\to\G'$ is said to be {\em Pansu differentiable} at a point $p\in \Omega$ if there exists a Carnot homomorphism $L\colon \G\to\G'$, i.e., a group homomorphism that commutes with the dilations $\delta_{\lambda}$, such that
	\[
	\lim_{x\to p}\frac{d_{\G'}(f(p)^{-1}f(x), L(p^{-1}x))}{d_\G(x,p)}=0.
	\]
	The map $L$ is uniquely determined, whenever it exists, and it is called the {\em Pansu differential} of $f$ at $p$ and it is denoted by $\de_{\rm P}\! f(p)$.
\end{defi}

\begin{defi}[$C^1_{\rm H}$-function]
	Let $\Omega\subseteq \G$ be an open subset of a Carnot group $\G$. A map $f\colon \Omega\to \R^k$ is said to be {\em of class $C^1_{\rm H}$} if it is Pansu differentiable and the Pansu differential $\de_{\rm P}\! f\colon \G\to \R^k$ is continuous. We denote by $C^1_{\rm H}(\Omega;\R^k)$ the set of $\R^k$-valued functions of class $C^1_{\rm H}$ in $\Omega$.
\end{defi}

\begin{defi}[(Uniformly) intrinsic differentiability]\label{defiintrinsicdiff}
	Let $\W$ and $\mathbb L$ be complementary subgroups of a Carnot group $\mathbb G$ and let $\widetilde{\phi}\colon\widetilde{U}\subseteq \W \to\mathbb L$ be a function with $\widetilde{U}$ open in $\mathbb W$. For $a_0\in\widetilde{U}$, let $p_0\coloneqq\widetilde{\phi}(a_0)^{-1}\cdot a_0^{-1}$ and denote by $\widetilde{\phi}_{p_0}\colon \widetilde{U}_{p_0}\subseteq \mathbb W\to\mathbb L$ the shifted function defined in \cref{def:PhiQ}. 
	
	We say that $\widetilde{\phi}$ is {\em intrinsically differentiable} at $a_0$ if the shifted function $\widetilde{\phi}_{p_0}$
	is {\em intrinsically differentiable} at $e$, i.e., if there is an intrinsically linear map $\de^{\phi}\!\phi_{a_0}\colon\mathbb W\to\mathbb L$ such that
	\begin{equation}\label{eqn:IdInCoordinates2}
	\lim_{r\to 0}\left(\sup\left\{\frac{\|\de^{\phi}\!\phi_{a_0}(b)^{-1}\cdot\widetilde{\phi}_{p_0}(b)\|}{\|b\|}: b\in \widetilde U_{p_0}, 0<\|b\|<r\right\}\right)= 0.
	\end{equation}
	The function $\de^{\phi}\!\phi_{a_0}$, sometimes denoted also by $\de^{\phi}\!\phi(a_0)$, is called {\em intrinsic differential} of $\widetilde\phi$ at $a_0$, and we say that $\widetilde \phi$ is {\em intrinsically differentiable} if it is intrinsically differentiable at any point $a_0\in \widetilde U$. We also denote by ${\rm ID}(\widetilde U,\sW;\sL)$ the set of intrinsically differentiable functions $\widetilde\phi\colon\widetilde U\subseteq \sW\to\sL$.
	
	We say that $\widetilde{\phi}$ is {\em uniformly intrinsically differentiable} at $a_0$ if, setting $p_{a}\coloneqq \widetilde \phi(a)^{-1}\cdot a^{-1}$ for any $a\in \widetilde U$, we have 
	\begin{equation}
	\lim_{r\to 0}\left(\sup\left\{\frac{\|\de^{\phi}\!\phi_{a_0}(b)^{-1}\cdot\widetilde{\phi}_{p_a}(b)\|}{\|b\|}: a\in \widetilde U\cap B(a_0,r), b\in\widetilde U_{p_a}\cap B(a_0,r), a\neq b\right\}\right)= 0.
	\end{equation}
	We say that $\widetilde \phi$ is {\em uniformly intrinsically differentiable} on $\widetilde U$ if it is uniformly intrinsically differentiable at any $a_0\in \widetilde U$. We finally denote by ${\rm UID}(\widetilde U,\sW;\sL)$ the set of uniformly intrinsically differentiable functions $\widetilde \phi\colon\widetilde U\subseteq \sW\to \sL$.
\end{defi}

	\begin{rem}[Intrinsic difference quotients]
		In the papers \cite{FSSC06, Ser17}, the authors introduce and study the following two notions, giving characterizations for intrinsically Lipschitz continuity, see \cite[Proposition~3.11 \& Theorem~3.21]{Ser17}. For a continuous function $\widetilde\phi\colon \widetilde U\subseteq \mathbb W\to \mathbb L$, defined on $\widetilde U$ open, the \emph{intrinsic difference quotients} of $\widetilde\phi$ at the point $w\in \widetilde U$ in the direction $Y\in {\rm Lie}(\mathbb W)$ at time $t>0$, are defined as
		\[
		\nabla_Y\widetilde\phi(w,t)\coloneqq \delta_{1/t}\widetilde\phi_{p}(\delta_t\exp Y),
		\]
		for every $t>0$, where $p\coloneqq \widetilde\phi(w)^{-1}\cdot w^{-1}$, and whenever $\delta_t\exp Y$ is in $\widetilde U_p$. The \emph{intrinsic directional derivative} of $\widetilde\phi$ at $w\in \widetilde U$ in the direction $Y\in {\rm Lie}(\mathbb W)$ is defined by
		\[
		D_Y\widetilde\phi(w)\coloneqq \lim_{t\to0}\nabla_Y\widetilde\phi(w,t),
		\]
		whenever the limit exists. In analogy with Euclidean Calculus, we notice that, if $\widetilde\phi\in {\rm ID}(\widetilde U, \mathbb W; \mathbb L)$, then it admits intrinsic directional derivatives at any $w\in \widetilde U$ along any $Y\in {\rm Lie}(\mathbb W)$, and, moreover, one has $D_{Y}\widetilde\phi(w)=\de^{\phi}\!\phi(w)(\exp Y)$, for any $w\in \widetilde U$ and every $Y\in \mathrm{Lie}(\mathbb W)$. Indeed, this is a consequence of the following identity
		\[
		\|\de^\phi\!\phi(w)(\exp Y)^{-1}\delta_{1/t}\widetilde \phi_{p}(\delta_t\exp Y)\|=\frac{\|\left(\de^\phi\!\phi(w)(\delta_t\exp Y)\right)^{-1}\widetilde \phi _{p}(\delta_t\exp Y)\|}{t},
		\]
		that simply comes from the fact both the norm $\|\cdot\|$ and $\de^\phi\!\phi$ are $\delta_{\lambda}$-homogeneous. Then from the previous equality and \eqref{eqn:IdInCoordinates2} with $a_0=w$, and $b=\delta_t\exp Y$, we get the sought claim taking $t\to 0$.
	\end{rem}

\begin{prop}[{\cite[Proposition 3.4]{DiDonato18}}]\label{p2.5}
	Let $\W$ and $\mathbb L$ be two complementary subgroups of a Carnot group $\mathbb G$ with $\mathbb L$ horizontal and $k$-dimensional and let $\ell\colon\sW\to\sL$ be an intrinsically linear function. Then  $\ell$ only depends on the horizontal components of the elements in $\sW$, namely on $\W_1\coloneqq\W\cap \mathbb V_1$, where $\mathbb V_1=\exp(V_1)$. In particular, if $\pi_{\mathbb V_1}$ denotes the projection from $\G$ to $\mathbb V_1$, see \eqref{eqn:ProjectionOnV1}, one has
	\begin{equation*}\label{lineare0}
	\ell(a)=\ell(\pi_{\mathbb V_1} a), \qquad \forall a\in \mathbb W.
	\end{equation*}
	As a consequence, $\exp^{-1}\circ \ell\circ\exp_{|_{{\rm Lie}(\sW)\cap V_1}}\colon {\rm Lie}(\sW)\cap V_1\to {\rm Lie}(\sL)$ is linear, and there exists a constant $C\coloneqq C(\ell)>0$  such that
	\begin{equation}\label{linea2.0}
	\|\ell(a)\|\leq C \|\pi_{\mathbb V_1}a\|,\qquad \forall a\in \W.
	\end{equation}
\end{prop}

\begin{defi}[Intrinsic gradient]\label{def:NablaPhiPhi}
	Let $\W$ and $\mathbb L$ be two complementary subgroups of a Carnot group $\mathbb G$ with $\mathbb L$ horizontal and $k$-dimensional, let $\widetilde U\subseteq \mathbb W$ be open, and let $\widetilde{\phi}\colon\widetilde U \to\sL$ be intrinsically differentiable at $a_0\in\widetilde{U}$. By \cref{p2.5}, the map $\exp^{-1}\circ (\de^{\phi}\!\phi_{a_0})\circ\exp_{|_{{\rm Lie}(\sW)\cap V_1}}$ is linear and thus there exists a linear map $\nabla^{\phi}\phi_{a_0}\in\mathrm{Lin}({\rm Lie}(\sW)\cap V_1; {\rm Lie}(\sL))$ such that 
	\[
	\de^{\phi}\!\phi_{a_0}(\exp W)=\exp\left(\nabla^{\phi}\phi_{a_0}(W)\right),\qquad \forall W\in{\rm Lie}(\sW)\cap V_1.
	\]
\end{defi}
\begin{rem}[Intrinsic gradient in exponential coordinates]\label{rem:IntrinsicGradientInCoordinates}
	Assume $(X_1,\dots, X_n)$ is an adapted basis of the Lie algebra $\mathfrak g$ such that $\mathbb L=\mbox{span}\{X_1,\dots,X_k\}$ and $\mathbb W=\mbox{span}\{X_{k+1},\dots,X_n\}$ and identify $\sW$ and $\sL$ with $\R^{n-k}$ and $\R^k$, respectively, through exponential coordinates as explained in \cref{def:coordinateconletilde}.
	Then, by \cref{def:NablaPhiPhi}, with a little abuse of notation, we get a $k\times (m-k)$ matrix $\nabla^{\phi}\phi_{a_0}$  such that, in coordinates, one has
	\begin{equation*}\label{DISSUdifferential}
	\de^{\phi}\!\phi_{ a_0}(a)=\left(\nabla^{\phi}\phi_{a_0}(a_{k+1},\dots,a_{m})^{\rm T}, 0,\dots,0\right), \qquad \forall a=(a_{k+1},\dots,a_n)\in \sW\equiv\R^{n-k}.
	\end{equation*}  
\end{rem}

The following proposition gives us a more explicit way to write the definition of functions in ${\rm ID}(\widetilde U,\sW;\sL)$ and in ${\rm UID}(\widetilde U,\sW;\sL)$, whenever $\mathbb L$ is horizontal. 
\begin{prop}[{\cite[Proposition 3.5]{DiDonato18}}]\label{rem:IdAndUidInCoordinates} Let $\W$ and $\mathbb L$ be complementary subgroups in a Carnot group $\mathbb G$, with $\mathbb L$ horizontal, and let $\widetilde{\phi}\colon\widetilde{U}\subseteq \W \to\mathbb L$ with $\widetilde{U}$ open in $\mathbb W$. Then the following facts hold.
	\begin{itemize}
		\item[(a)] $\widetilde \phi$ is intrinsically differentiable at $\widetilde a_0\in \widetilde U$ if and only if 
	\begin{equation}\label{eqn:IdInCoordinates}
	\lim_{r\to 0}\left(\sup\left\{\frac{|\phi(b)-\phi(a_0)-\nabla^{\phi}\phi_{a_0}(a_0^{-1}\cdot b)|}{\|\widetilde{\phi}(a_0)^{-1}a_0^{-1}\cdot b\,\widetilde{\phi}(a_0)\|}: b\in U, 0<\|a_0^{-1}\cdot b\|<r \right\}\right)= 0.
	\end{equation}
	\item[(b)] $\widetilde{\phi}$ is uniformly intrinsically differentiable at $\widetilde a_0$ if and only if
	\begin{equation}\label{eqn:UidInCoordinates}
	\lim_{r\to 0}\left(\sup\left\{\frac{|\phi(b)-\phi(a)-\nabla^{\phi}\phi_{a_0}(a^{-1}\cdot b )|}{\|\widetilde{\phi}(a)^{-1}a^{-1}b\,\widetilde \phi(a)\|}: a, b\in B(a_0,r)\cap U, a\neq b\right\}\right)= 0.
	\end{equation}
	\end{itemize}
\end{prop}

\begin{rem}\label{rem:abuse}
	We notice that in \eqref{eqn:IdInCoordinates} and \eqref{eqn:UidInCoordinates} there is a little abuse of notation, for the sake of simplicity. First we are identifying $\mathbb L$ with $\mathbb R^k$ in order to write the differences in the numerators, and moreover we write $\nabla^{\phi}\phi_{a_0}(a_0^{-1}\cdot b)$ but we mean $\nabla^{\phi}\phi_{a_0}(\pi_{V_1}\exp^{-1}(\widetilde a_0^{-1}\cdot\exp b))$. 
\end{rem}
\begin{rem}[Intrinsic differentiability \& tangent subgroups]\label{rem:InvarianceOfIdByTranslations}
 Let us collect the following observations about \cref{defiintrinsicdiff}.
 
(i) If $\widetilde \phi$ is intrinsically differentiable at $a_0\in \widetilde U$, there is a unique  intrinsically linear function $\de^{\phi}\!\phi_{a_0}$ satisfying $\eqref{eqn:IdInCoordinates2}$.  Moreover $\widetilde \phi$ is continuous at $ a_0$, see \cite[Theorem 3.2.8 and Proposition 3.2.3]{FMS14}.

(ii) The notion of intrinsic differentiability is invariant under group translations. More precisely, let $a,b$ be in $\widetilde U$ and let $p\coloneqq \widetilde \phi (a)^{-1}\cdot a^{-1}$ and $q\coloneqq \widetilde \phi (b)^{-1}\cdot b^{-1}$. Then $\widetilde \phi $ is intrinsically differentiable at $a$ if and only if $\widetilde \phi _{q^{-1}p} = (\widetilde \phi _{p})_{q^{-1}}$ is intrinsically differentiable at $b$, see (\cite[Remark 3.2.2]{FMS14}).

(iii) The analytic definition of intrinsic differentiability has an equivalent geometric formulation. Indeed, the intrinsic differentiability at one point is equivalent to the existence of a tangent subgroup to the graph, see \cite[Theorem 4.15]{FSSC11} for the proof in the case of Heisenberg groups $\mathbb H^n$. If we have $\widetilde \phi: \widetilde U\subseteq \W \to \mathbb L$, and $w_0\in\widetilde U$, we say  that a homogeneous subgroup $\T$ of $\G$ is a {\em tangent subgroup} to $\graph {\widetilde \phi }$ at $w_0\cdot\widetilde{\phi}(w_0)$ if the following facts hold.
\begin{enumerate}
\item[(i)] $\T$ is a complementary subgroup of $\mathbb L$;
\item[(ii)] In any compact subset of $\G$, the limit
\begin{equation*}
\lim_{\lambda \to \infty } \delta _\lambda \left((w_0\cdot \widetilde\phi(w_0))^{-1}\cdot\graph {\widetilde \phi } \right)  =\T
\end{equation*}
holds in the sense of Hausdorff convergence.
\end{enumerate}
In the introduction of \cite{FMS14} the authors say that $\widetilde \phi$ is intrinsically differentiable at $w_0$ if and only if $\graph {\widetilde \phi }$ has a tangent subgroup $\T$ in $w_0\cdot\widetilde{\phi}(w_0)$ and in this case $\T=\graph {\de ^\phi\! \phi_{w_0}}$. The complete proof can be given building on \cite[Theorem 3.2.8]{FMS14}, that shows one part of the statement, and generalizing \cite[Theorem 4.15]{FSSC11}, that holds verbatim in the context of arbitrary Carnot groups. We thank Sebastiano Nicolussi Golo for having shared with us some notes containing a detailed proof of the previously discussed statement. 
\end{rem}

\begin{prop}[{\cite[Proposition 3.7]{DiDonato18}}]\label{prop:ContinuityOfDifferentialUid}
Let $\W$ and $\mathbb L$ be complementary subgroups of a Carnot group $\G$ with $\mathbb L$ horizontal and $k$-dimensional, let $\widetilde U\subseteq \mathbb W$ be open and let $\widetilde\phi \in {\rm UID}(\widetilde U,\mathbb W;\mathbb L)$. Then the following facts hold.
\begin{enumerate}
\item[(a)] $\widetilde \phi$ is intrinsically Lipschitz continuous on every relatively compact subset of $\widetilde U$. 
\item[(b)] the function $a \mapsto \nabla ^{\phi}\phi_{a}$ is continuous from $\widetilde U$ to the space of matrices $\R^{k\times(m-k)}$. Here $\nabla^{\phi}\phi$ is the intrinsic gradient, see \cref{def:NablaPhiPhi}.
\end{enumerate}
\end{prop}
 
 \begin{defi}[$\nabla_{\mathbb W},\nabla_{\mathbb L}$]\label{remarkDPHI2Vectorial} 
 	Let $\W$ and $\mathbb L$ be two complementary subgroups of a Carnot group $\mathbb G$, with $\mathbb L$ horizontal and $k$-dimensional and let $f\in C^1_{\rm H}(\widetilde U;\R^k)$.
 	Consider an adapted basis $(X_1,\dots,X_n)$ of the Lie algebra $\mathfrak g$ such that $\mathbb L=\exp({\rm span} \{X_1,\dots, X_k\})$ and \\ $\W=\exp({\rm span} \{X_{k+1},\dots, X_n\})$. Then, we  define $ \nabla _\sL f$ and $\nabla_\sW f$ by setting  
 	\begin{equation*}
 	\nabla_{\mathbb L}f\coloneqq \begin{pmatrix}
 	X_1f^{(1)}& \dots&  X_kf^{(1)} \\
 	\vdots & \ddots & \vdots \\
 	X_1f^{(k)} &\dots  &X_kf^{(k)}
 	\end{pmatrix},\qquad \nabla_{\mathbb W}f\coloneqq \begin{pmatrix}
 	X_{k+1}f^{(1)}&\dots&  X_{m}f^{(1)} \\
 	\vdots & \ddots & \vdots \\
 	X_{k+1}f^{(k)}& \dots&  X_{m}f^{(k)}
 	\end{pmatrix}.
 	\end{equation*}
 	In particular, one has that, in exponential coordinates, $\nabla _{\rm H} f = \left(
 	\nabla_{\mathbb L}f \;|\;  \nabla_{\mathbb W}f 
 	\right)$.
 \end{defi}
 
We recall the notion of co-horizontal $C^1_{\rm H}$-surface of arbitrary codimension, see ~\cite[Definition~3.3.4]{Koz15}. \textbf{We stress that we changed the terminology with respect to \cite[Definition~3.3.4]{Koz15}. What he calls co-Abelian surface, for us is a co-horizontal surface}. For a very general definition of $C^1_{\rm H}$-surface, we refer the reader to \cite[Definition~3.1]{Mag06}, \cite[Definition~10.2]{Mag13} and to \cite[Section~2.5]{JNGV20}. 
 
\begin{defi}[co-horizontal $C^1_{\rm H}$-surface]\label{def:CoAbelian}
Let $\G$ be a Carnot group of rank $m$ and let $1\leq k\leq m$. We say that $\Sigma\subset \G$ is a \emph{co-horizontal $C^1_{\rm H}$-surface of codimension $k$} if, for any $p\in \Sigma$, there exist a neighborhood $\widetilde U$ of $p$ and a map $f\in C^1_{\rm H}(\widetilde U;\R^k)$ such that
\begin{equation}\label{eqn:Representation}
\Sigma\cap  \widetilde U=\{g\in \widetilde U: f(g)=0\},
\end{equation}
and the Pansu differential $\de_{\rm P}\!f(p)\colon \mathbb G\to \R^k$ of $f$ is surjective.  

We say that $\Sigma$ is a codimension $k$ \emph{co-horizontal $C^1_{\rm H}$-surface with complemented tangents} if, in addition, given a representation around $p$ as in \eqref{eqn:Representation}, the homogeneous subgroup $\mathrm{Ker}(\de_{\rm P}\! f(p))$ admits a horizontal complement (of dimension $k$).
In this case, we call ${\rm Ker}(\de_{\rm P}\!f(p))$ the {\em homogeneous tangent space to $\Sigma$ at $p$}. This homogeneous subgroup at $p$ is independent of the choice of $f$, see \cite[Theorem~1.7]{Mag13}.
\end{defi}

\noindent We remark that, if $\Sigma\subseteq \G$ is a co-horizontal $C^1_{\rm H}$-surface with complemented tangents, then one can use the implicit function Theorem, see \cite[Theorem~2.1]{FSSC03b} for the one-codimensional case, and see \cite[Theorem~1.4]{Mag13} for the more general statement, to locally represent the surface as a graph of a function $\widetilde \phi\colon \widetilde U\subseteq \sW\coloneqq \mathrm{Ker}(\de_{\rm P}\! f(p))\to \sL$, with $\mathbb W$ and $\mathbb L$ complementary subgroups. 

The following proposition follows from \cite[Theorem~4.1 \& Theorem~4.6]{DiDonato18} and relates level sets of $\mathbb R^k$-valued $C^1_{\rm H}$-functions, and ultimately co-horizontal $C^1_{\rm H}$-surfaces with complemented tangents, with uniformly intrinsically differentiable functions.
\begin{prop}[{\rm \cite[Theorem~4.1 \& Theorem~4.6]{DiDonato18}}]\label{prop:UidC1H}
Let $\W$ and $\mathbb L$ be two complementary subgroups of a Carnot group $\mathbb G$, with $\mathbb L$ horizontal and $k$-dimensional,  take $\widetilde U\subseteq \mathbb W$ open and $\widetilde{\phi}\in\mathrm{UID}(\widetilde U,\mathbb W;\mathbb L)$. Then, for every $a\in\widetilde{U}$, there exist a neighborhood $\widetilde V$ of $ a\cdot\widetilde{\phi}(a)$ in $\mathbb G$, and $f\in C^1_{\rm H}(\widetilde V;\R^k)$, such that
\[
\widetilde{\Phi}(\widetilde{U})\cap \widetilde V= \{g\in \widetilde V: f(g)=0\},
\]
and, for every $g\in \widetilde V$, the Pansu differential $\de_{\rm P}\! f(g)_{|_{\mathbb L}}\colon\mathbb L\to\mathbb R^k$ is bijective. As a consequence $\graph{\widetilde \phi}$ is a co-horizontal $C^1_{\rm H}$-surface of codimension $k$, with tangents complemented by $\mathbb L$. Moreover, if $(X_1,\dots,X_n)$ is an adapted basis of the Lie algebra $\mathfrak g$ such that $\mathbb L=\exp({\rm span} \{X_1,\dots, X_k\})$ and $\W=\exp({\rm span} \{X_{k+1},\dots, X_n\})$, then $\det \nabla_{\sL}f\neq0$ and, in exponential coordinates, one has
\begin{equation}\label{DPHI2Vectorial}
\nabla^{\phi} \phi (a)=- \left(\nabla_{\mathbb L}f (\widetilde\Phi (a))\right)^{-1}\nabla_{\mathbb W}f (\widetilde\Phi (a)),\qquad \forall a\in \widetilde U.
\end{equation}
For the definition of $\nabla^{\phi}\phi,\nabla_{\mathbb W}$ and $\nabla_{\mathbb L}$ we refer to \cref{def:NablaPhiPhi} and \cref{remarkDPHI2Vectorial}.

On the other hand, if $1\leq k\leq m$ and $\Sigma$ is a codimension $k$ co-horizontal $C^1_{\rm H}$-surface with complemented tangents, then, for every $p\in \Sigma$, there exist two complementary subgroups $\sW$ and $\sL$ of $\G$ with $\sL$ horizontal and $k$-dimensional, a neighborhood $\widetilde V\subseteq \G$ of $p$ and $\widetilde \phi\in {\rm UID}(\widetilde U,\sW;\sL)$, with $\widetilde U=\pi_\sW(\widetilde V)$, such that 
\[
\Sigma\cap \widetilde V=\graph{\widetilde \phi}.
\]
\end{prop}

\begin{rem}
Notice that, in the setting of \cref{prop:UidC1H}, in the case $k=1$, one may assume $X_1f\neq 0$ on $\widetilde V$, and, in coordinates, formula \eqref{DPHI2Vectorial} reads as
\begin{equation}\label{DPHI2}
\nabla^{\phi} \phi(a)=-\left (\frac{X_2f}{X_1f} ,\dots , \frac{X_{m}f}{X_1f}\right)\circ\Phi(a),\qquad \forall a\in \widetilde U.
\end{equation}
\end{rem}

\begin{rem}[Tangent subgroups to $C^1_{\rm H}$-surfaces]\label{rem:ConvergenceLocallyUniform}
	From the previous \cref{prop:UidC1H} and \cref{rem:InvarianceOfIdByTranslations} it directly follows that every co-horizontal $C^1_{\rm H}$-surface with complemented tangents has Hausdorff tangent everywhere. For a proof of this property in a more general context  one can see \cite[Theorem~1.7]{Mag13}, or \cite[Lemma 2.14, point (iii)]{JNGV20}.
	This convergence is moreover locally uniform: we will not use this information, but this comes from \cite[Theorem~3.1.1]{Koz15}.
\end{rem}

\section{Intrinsic projected vector fields on subgroups}\label{sec:3}
In this section we mainly deal with complementary subgroups $\mathbb W$ and $\mathbb L$ of a Carnot group $\mathbb G$ along with a continuous map $\widetilde{\phi}\colon\widetilde{U}\subseteq\mathbb W\to\mathbb L$, where $\widetilde{U}$ is open in $\mathbb W$.

In \cref{sub:DefinitionDPhi} we shall define, for some $W\in {\rm Lie}(\mathbb W)$, the projected vector field $D^{\phi}_W$ on $\mathbb W$ by taking the projection on $\mathbb W$ of $W$ restricted the graph $\widetilde{\Phi}(\widetilde{U})$ of $\widetilde{\phi}$ (see \cref{def:PhiJ}), and we discuss some basic properties of these vector fields. We give explicit formulas for these vector fields in Heisenberg groups $\mathbb H^n$, in Carnot groups of step 2, and in the Engel group $\mathbb E$ (see \cref{example:Heisenberg}, \cref{example:Step2}, and \cref{example:Engel}, respectively). In \cref{prop:CompatibleCoordinates} we show an explicit expression of such vector fields in exponential coordinates. 
The definition of the projected vector fields appeared first in \cite{Koz15}, see \cref{rem:OnTheDefinitionOfDPhij}. In \cite{Koz15} the author gives equivalent conditions for $\widetilde{\Phi}(\widetilde{U})$ to be an intrinsically Lipschitz graph (respectively a co-horizontal $C^1_{\rm H}$-surfaces with complemented tangents) in terms of H\"older properties of the integral curves of the vector fields $D^{\phi}_W$, see \cite[Theorem 4.2.16]{Koz15} (respectively \ \cite[Theorem~4.3.1]{Koz15}). Within our context we recover these results by using invariance properties of such vector fields, see the introduction to \cref{sec:Theorems}.

In \cref{sub:InvarianceProperties} we prove some invariance properties of the projected vector fields with respect to the intrinsic translations (see \cref{def:PhiQ}) of $\widetilde\phi$. In particular we write how the vector field $D^{\phi_q}$ changes with respect to $D^{\phi}$ and how the integral curves of $D^{\phi_q}$ change with respect to the integral curves of $D^{\phi}$, see \cref{lem:Invariance1DPhi} and \cref{lem:Invariance2IntegralCurve} when $\mathbb L$ is horizontal, and \cref{rem:CurveInvarianceTrueAlsoWhenWNormal} for the general case in which $\mathbb W$ is normal. These invariance properties will be crucial for the proof of the results in \cref{sec:Theorems}.

In \cref{sub:MetricProperties} we recall that if $\widetilde\phi$ is intrinsically Lipschitz, then $\widetilde\phi\circ\widetilde\gamma$ is $1/j$-H\"older whenever $\widetilde\gamma$ is an integral curve of $D^{\phi}_W$ with $\deg W=j$, see \cref{prop:IntrinsicLipschitzIntegralCurve}. We stress that this property was already known from \cite[Theorem~4.2.16]{Koz15}. We improve this result when $\widetilde\phi$ is more regular. Namely if $\widetilde\phi$ is intrinsically differentiable, then $\widetilde\phi\circ\widetilde\gamma$ is Euclidean differentiable whenever $\widetilde\gamma$ is an integral curve of $D^{\phi}_W$ with $\deg W=1$, while if $\deg W>1$ we obtain a \textbf{pointwise} little H\"older continuity of $\widetilde\phi\circ\widetilde\gamma$, see \cref{prop:DerivativeOfIntegralCurves}. This pointwise little H\"older continuity improves to a \textbf{uniform} little H\"older continuity if $\widetilde\phi$ is \textbf{uniformly} intrinsically differentiable, see \cref{prop3.6cont}, and \cref{prop:UidImpliesH1k} for a more refined conclusion.

In \cref{sub:Broad*} we recall the notion of broad* solution to the system $D^{\phi}\phi=\omega$, with a continuous datum $\omega$. The study of the relation between being a broad* solution to the system $D^{\phi}\phi=\omega$ with a continuous $\omega$ and the intrinsic regularity of ${\rm graph}(\widetilde\phi)$ was first initiated, in $\mathbb H^n$ for $\mathbb L$ one-dimensional in \cite[Section~5]{ASCV06}, and then continued in \cite{BSC10a, BSC10b, BCSC14}. For the case $\G=\mathbb H^n$ and $\mathbb L$ horizontal and $k$-dimensional see also \cite{Cor19} and for the general case of Carnot groups of step 2 and $\mathbb L$ one-dimensional, see \cite{DiDonato18,DiDonato19}. In \cref{prop:IdGradContImpliesBroad} we give a sufficient condition for the map $\phi$ to be a broad* solution to the system $D^{\phi}\phi=\omega$, with $\omega$ continuous. This condition is the intrinsic differentiability plus the continuity of the intrinsic gradient. 

\subsection{Definition of $D^{\phi}$ and main properties}\label{sub:DefinitionDPhi}
In this subsection we define the projected vector fields $D^{\phi}_W$ and state some of their properties.
\begin{defi}[Projected vector fields]\label{def:PhiJ}
	Given two complementary subgroups $\mathbb W$ and $\mathbb L$ in a Carnot group $\mathbb G$, and a continuous function $\widetilde{\phi}\colon\widetilde{U}\subseteq\mathbb W\to\mathbb L$ defined on an open set $\widetilde{U}$ of $\mathbb W$, we define, for every $W\in \mathrm{Lie}(\mathbb W)$,  the continuous {\em projected vector field} $D^\phi_W$, by setting 
	\begin{equation}\label{eqn:DefinitionOfDj}
	(D^{\phi}_{W})_{|_w} (f)\coloneqq W_{|_{w\cdot\widetilde\phi(w)}}( f\circ \pi_{\sW}),
	\end{equation}
	for all $w\in \widetilde U$ and all $ f\in C^{\infty}(\W)$. When $W$ is an element $X_j$ of an adapted basis $(X_1,\dots,X_n)$ we also denote $D^{\phi}_j\coloneqq D^{\phi}_{X_j}$.
\end{defi}
\begin{rem}\label{rem:OnTheDefinitionOfDPhij}
	The \cref{def:PhiJ} is well posed since the projection $\pi_{\mathbb W}$ is polynomial and hence $C^\infty$ for every arbitrary splitting. Notice that if $\widetilde{\phi}$ is $C^{\infty}$ then  $D^{\phi}_W$ is a vector field with $C^\infty$ coefficients.
	
	 \cref{def:PhiJ} has been given in \cite[Definition~4.2.12]{Koz15} and it has been studied in the case $\mathbb W$ is a homogeneous normal subgroup and, more specifically, when $\mathbb L$ is horizontal and $k$-dimensional. We refer to the discussion in the introduction of \cref{sec:3}.
	 \textbf{From now on $\mathbb W$ denotes a homogeneous normal subgroup of $\mathbb G$}. 
\end{rem}
\begin{rem}
	Notice that \eqref{eqn:DefinitionOfDj} is equivalent to
	\begin{equation}\label{eqn:EquivalentDefinitionDj}
	(D^{\phi}_{W})_{|_w}=\de (\pi_{\sW})_{\widetilde{\Phi}(w)}(W_{|_{\widetilde{\Phi}(w)}}),
	\end{equation}
	that is, $D^{\phi}_W$ is the push-forward of the vector field $W$ towards the map $(\pi_{\mathbb W})_{|_{\widetilde{\Phi}(\widetilde{U})}}\colon\widetilde{\Phi}(\widetilde{U})\to \widetilde{U}$.
	Thus, as already observed in \cite[Equation 4.4]{Koz15}, one has
	\begin{equation}\label{eqn:DPhiJAsDifferential}
	\begin{aligned}
	(D^{\phi}_{W})_{|_w} &=\frac{\de}{\de t}_{|_{t=0}} \pi_{\sW}(\widetilde{\Phi}(w)\cdot\exp(tW))=\frac{\de}{\de t}_{|_{t=0}}w\cdot\widetilde{\phi}(w)\cdot\exp(tW)\cdot\widetilde{\phi}(w)^{-1}= \\
	&=\frac{\de}{\de t}_{|_{t=0}}L_w\circ L_{\widetilde{\phi}(w)}\circ R_{\widetilde{\phi}(w)^{-1}}(\exp(tW))= \\
	&= \de(L_w)_e\circ \de(L_{\widetilde{\phi}(w)})_{\widetilde{\phi}(w)^{-1}}\circ \de(R_{\widetilde{\phi}(w)^{-1}})_e (W_{|_e})=\de(L_w)_e\circ\mathrm{Ad}_{\widetilde\phi(w)}(W_{|_e}).
	\end{aligned}
	\end{equation}
	For the previous computation, we used the definition of the differential and, in the second equality, the fact that 
	\[
	\begin{aligned}
	\pi_{\sW}(\widetilde{\Phi}(w)\cdot\exp(tW))&=\pi_{\sW}(w\cdot\widetilde{\phi}(w)\cdot\exp(tW)\cdot\widetilde{\phi}(w)^{-1}\cdot\widetilde{\phi}(w))=\\
	&=w\cdot\widetilde{\phi}(w)\cdot\exp(tW)\cdot\widetilde{\phi}(w)^{-1},
	\end{aligned}
	\]
	 for every $w\in \mathbb W$ and $W\in \mathrm{Lie}(\W)$, where the last equality holds since $\mathbb W$ is normal.
\end{rem}
\begin{exa}[Projected vector fields on Heisenberg groups]\label{example:Heisenberg}
	Consider the Heisenberg group $\mathbb H^n$, with an adapted basis $(X_1,\dots,X_{2n+1})$ of its Lie algebra such that the only nonvanishing relations are $[X_i,X_{n+i}]=X_{2n+1}$, for every $1\leq i\leq n$. Fix $1\leq k\leq n$,  identify $\mathbb H^n$ with $\mathbb R^{2n+1}$ by means of exponential coordinates associated with $(X_1,\dots,X_{2n+1})$ and define $\mathbb W\coloneqq\{x_1=\dots=x_k=0\}$, and $\mathbb L\coloneqq\{x_{k+1}=\dots=x_{2n+1}=0\}$. Then, for a continuous $\widetilde{\phi}\colon\widetilde U\subseteq \mathbb W\to\mathbb L$, with $\widetilde U$ open, one can compute in exponential coordinates
	\begin{equation}\label{eqn:ProjectedVectorFieldsHn1}
	\begin{split}
	(D^{\phi}_{X_j})_{|_w}&=(X_j)_{|_w}, \qquad  k+1\leq j\leq n \vee k+n+1\leq j\leq 2n, \\ (D^{\phi}_{X_{n+i}})_{|_w}&=(\partial_{x_{n+i}})_{|_w}+\phi^{(i)}(w)(\partial_{x_{2n+1}})_{|_w}, \quad i= 1,\dots, k.
	\end{split}
	\end{equation}
	\begin{equation}\label{eqn:ProjectedVectorFieldsHn2}
	(D^{\phi}_{X_{2n+1}})_{|_w}=(\partial_{x_{2n+1}})_{|_w}.
	\end{equation}
	for every $w\in U$, where $\phi$ denotes the composition of $\widetilde \phi$ with the exponential coordinates and $\phi^{(i)}$ is its $i$-th component of $\phi$. Notice that we do not have the first condition in case $k=n$. 
\end{exa}
\begin{rem}\label{rem:ProjectedVectorFieldsHeisenberg}
	For the computations of \cref{example:Heisenberg}, we refer to \cite[Section 4.4.2]{Koz15}, where the constant is slightly different from ours because of the fact that the author considers the model of $\mathbb H^n$ with relations $[X_i,X_{n+i}]=-4X_{2n+1}$, for every $1\leq i\leq n$. The expression of the projected vector fields in case $\mathbb L$ is $k$-dimensional is also in \cite[Definition~3.6]{Cor19}.
	
	It is by now well known, from the papers \cite{ASCV06}, \cite{BSC10a}, \cite{BSC10b}, and \cite{BCSC14}, that in every Heisenberg group $\mathbb H^n$, in case $\mathbb L$ is one-dimensional, the intrinsic regularity of $\mathrm{graph}(\widetilde{\phi})$ depends on the regularity of the vector field $D^{\phi}$ applied to $\phi$, i.e., $D^{\phi}\phi\coloneqq(D^{\phi}_{X_2}\phi,\dots,D^{\phi}_{X_{2n}}\phi)$, which has to be considered in the sense of distributions. For the full results we refer to \cite[Theorems 4.90 \& 4.92]{SC16}. In particular $\mathrm{graph}(\widetilde{\phi})$ is an intrinsically Lipschitz graph (respectively a $C^1_{\rm H}$-hypersurface) if and only if $D^{\phi}\phi=\omega$, in the distributional sense, for some $\omega\in L^{\infty}(U)$ (respectively $\omega\in C(U)$). 
	
	A step towards obtaining analogous results in $\mathbb H^n$, in case $\mathbb L$ has higher dimension and $\omega$ is continuous, has been recently done by Corni in \cite{Cor19}. In particular, the author proves that, if $\mathbb L$ is horizontal $k$-dimensional, the set $\mathrm{graph}(\widetilde{\phi})$ is a co-horizontal $C^1_{\rm H}$-surface if and only if $\phi$ is a broad* solution to $D^{\phi}\phi=\omega$ for some $\omega\in C(U)$. We shall recall the definition of broad* in \cref{defbroad*}.
\end{rem}
\begin{exa}[Projected vector fields on Carnot groups of step 2]\label{example:Step2}
	 Consider a Carnot group $\mathbb G$ of rank $m$ and step 2 with an adapted basis $(X_1,\dots,X_{n})$. For $1\leq s,\ell \leq m$ and $m+1\leq i \leq n$, let us define the {\em structure constants} $c_{\ell s}^i$ by means of the relation $[X_{\ell},X_s]\eqqcolon\sum_{i=m+1}^n c_{\ell s}^i X_i$.   Identify $\mathbb G$ with $\mathbb R^n$ by means of exponential coordinates and take $\mathbb W\coloneqq\{x_1=0\}$, and $\mathbb L\coloneqq\{x_2=\dots=x_{n}=0\}$. Then, if $\widetilde{\phi}\colon \widetilde U\subseteq \mathbb W\to\mathbb L$ is continuous, with $\widetilde U$ open, by explicit computations one has in exponential coordinates
	\begin{equation}\label{ProjectedVectorFieldsStep21}
	 (D^{\phi}_{X_{j}})_{|_w}=(X_j)_{|_w}+\sum_{i=m+1}^{n} c_{1j}^i \phi(w)(\partial_{x_i})_{|_w}, \qquad \mbox{ for } j=2,\dots,m;
	\end{equation}
	\begin{equation}\label{ProjectedVectorFieldsStep22}
	(D^{\phi}_{X_{j}})_{|_w}=(X_j)_{|_w}=(\partial_{x_j})_{|_w}, \qquad \mbox{ for } j=m+1,\dots,n,
	\end{equation}
	for every $w\in U$, where $\phi$ denotes the composition of $\widetilde \phi$ with the exponential coordinates.
\end{exa}
\begin{rem}\label{rem:ProjectedVectorFieldsStep2}
For the expression of the projected vector fields in \cref{example:Step2} we refer also to \cite[Definition 5.2]{DiDonato18}. In the papers \cite{DiDonato18} and \cite{DiDonato19} the author started to generalize the results already proved in the Heisenberg groups $\mathbb H^n$ (see \cref{rem:ProjectedVectorFieldsHeisenberg}) to Carnot groups of step 2, in case $\mathbb L$ is one-dimensional. 

In particular, in \cite{DiDonato18}, in the setting of Carnot groups of step 2, the author deals with the characterization of maps $\widetilde{\phi}$ such that $\mathrm{graph}(\widetilde{\phi})$ is a $C^1_{\rm H}$-hypersurface. In \cite[Theorem~5.8]{DiDonato18}, the author recovers partially the result in \cite[Theorem 5.1]{ASCV06}, thus making the first step through the complete characterization in step 2 Carnot groups analogous to the one discussed in \cref{rem:ProjectedVectorFieldsHeisenberg} for the Heisenberg group. We stress that in this paper of ours we generalize \cite[Theorem~5.1]{ASCV06} to all step-2 Carnot groups, thus improving also \cite[Theorem~5.8]{DiDonato18}, see \cref{thm:MainTheorem.0}.

In \cite[Theorem 7.1 \& 7.2]{DiDonato19}, in the setting of Carnot groups of step 2, with $\mathbb L$ one-dimensional, the author recovers \cite[Theorem 1.1]{BCSC14} with an additional assumption: $\mathrm{graph}(\widetilde{\phi})$ is intrinsically Lipschitz if and only if $D^{\phi}\phi=\omega$ in the sense of distribution for some $\omega\in L^{\infty}(U)$ \textbf{and $\phi$ is locally $1/2$-H\"older along the vertical coordinates}. We expect that the techniques of \cref{sec:Step2} can be used to drop the previous additional assumption on the locally 1/2-H\"older continuity along the vertical coordinates. This will be subject of further investigations. 
\end{rem}
\begin{exa}[Projected vector fields on Engel group]\label{example:Engel}
	Consider the Engel group $\mathbb E$, which is the Carnot group of topological dimension $4$  whose Lie algebra $\mathfrak e$ admits an adapted basis $(X_1,X_2,X_3,X_4)$ such that
	\[
	\mathfrak e\coloneqq\mathrm{span}\{X_1,X_2\}\oplus\mathrm{span}\{X_3\}\oplus\mathrm{span}\{X_4\},
	\]
	where the only nonvanishing relations are $[X_1,X_2]=X_3, [X_1,X_3]=X_4$. We identify $\mathbb E$ with $\mathbb R^4$ by means of exponential coordinates, and we define $\mathbb W\coloneqq\{x_1=0\}$, and $\mathbb L\coloneqq\{x_2=x_3=x_4=0\}$. Then, by explicit computations that can be found in \cite[Section 4.4.1]{Koz15}, we get that, given a continuous function $\widetilde\phi:\widetilde U\subseteq \mathbb W\to\mathbb L$, with $\widetilde U$ open, the projected vector fields on $\mathbb W$ are
	\begin{equation}\label{eqn:ProjectedVectorFieldsEngel}
	D_{X_2}^{\phi}=\partial_{x_2}+\phi\partial_{x_3}+\frac{\phi^2}{2}\partial_{x_4}, \quad D_{X_3}^{\phi}=\partial_{x_3}+\phi\partial_{x_4}, \quad D^{\phi}_{X_4}=\partial_{x_4}.
	\end{equation}
	
	In \cite[Setion 4.4.1]{Koz15}, one can find the computations of the projected vector fields also for the pair of complementary subgroups $\left(\{x_2=0\},\{x_1=x_3=x_4=0\}\right)$. It is worth mentioning that in \cite[Section 4.5]{Koz15} there are some counterexamples, for the Engel group $\mathbb E$ with the splitting discussed here, to some of the statements, discussed in \cref{rem:ProjectedVectorFieldsHeisenberg}, and \cref{rem:ProjectedVectorFieldsStep2}, that holds for Carnot groups of step 2. For one of these examples, see also \cref{rem:VerticallyBroadHolderNonSiToglie}.
\end{exa}

We now prove a proposition about the general form of the projected vector fields, in an arbitrary Carnot group $\mathbb G$, in exponential coordinates. The proposition below can be found in \cite[Proposition 4.1.15]{Koz15}, but we however write the proof for the sake of completeness.
\begin{prop}[Projected vector fields in coordinates]\label{prop:CompatibleCoordinates}
	Let $\mathbb W$ and $\mathbb L$ be complementary subgroups of a Carnot group $\G$ such that  $\mathbb L$ is horizontal and $k$-dimensional, and let $\widetilde \phi\colon \widetilde U\subseteq \mathbb W\to \mathbb L$ be a continuous function on an open set $\widetilde U$. Fix an adapted basis $(X_1,\dots,X_n)$ of the Lie algebra $\mathfrak g$ such that $\mathbb W=\exp\left(\mathrm{span}\{X_{k+1},\dots,X_n\}\right)$ and $\mathbb L=\exp\left(\mathrm{span}\{X_1,\dots,X_k\}\right)$. If we identify $\mathbb G$ with $\mathbb R^n$ by means of exponential coordinates associated with $(X_1,\dots,X_n)$, then the vector fields $D^{\phi}_j\coloneqq D^{\phi}_{X_j}$ defined in \eqref{eqn:DefinitionOfDj} have the following expression:
	\begin{equation}\label{eqn:CoordinateDPhiJ}
	D^{\phi}_{j|_{(0,\dots,0,x_{k+1},\dots,x_n)}}=\partial_{x_j}+\sum_{i=n_{\deg j}+1}^n P_i^j(\phi^{(1)},\dots,\phi^{(k)},x_{k+1},\dots,x_{n_{\deg i-1}})\partial_{x_i}, \quad \forall j=k+1,\dots,n,
	\end{equation}
	where $\phi$ is the composition of $\widetilde \phi$ with the exponential map, $\phi^{(i)}=\phi^{(i)}(x_{k+1},\dots,x_n)$ denotes the $i$-th component of $\phi$, and $P_i^j$ is a polynomial of homogeneous degree $\deg i-\deg j$, with the convention that the degree of the $\phi^{(i)}$ components in the polynomial is 1. For the notation $n_{\square}$ and $\deg$, see the discussion before \cref{def:coordinateconletilde}.
\end{prop}
\begin{proof}
Since $\mathbb L$ is horizontal, then $\mathbb W$ is normal and \eqref{eqn:DPhiJAsDifferential} holds. Now the result follows from \eqref{eqn:DPhiJAsDifferential} and the following general fact that holds for arbitrary Carnot groups. Let us fix, in exponential coordinates, $x\in \mathbb G\equiv \mathbb R^n$. The differential of the left (respectively right) translation evaluated at a point $x'\in\mathbb G\equiv \mathbb R^n$, that we denote by $\de(L_x)_{x'}$ (respectively $\de(R_x)_{x'}$), is a matrix with identity $(m_\ell\times m_\ell)$-blocks  on the diagonal, with $1\leq \ell\leq s$, and moreover the element of position $ij$ is a polynomial in the coordinates of $x, x'$ of homogeneous degree $\deg i-\deg j$, if $\deg i>\deg j$. Instead if $\deg i<\deg j$ the element of position $ij$ is zero. This last statement about the structure of the differential of left and right translations  follows by the explicit expression of the product in coordinates, see \cite[Proposition~2.1]{FSSC03a}.
\end{proof}
\begin{rem}\label{rem:SameProofWNormal}
	Notice that, for notational purposes, \cref{prop:CompatibleCoordinates} is stated just for $\mathbb L$ horizontal but an expression similar to \eqref{eqn:CoordinateDPhiJ} holds also in the more general case in which $\mathbb W$ is normal, see also \cite[Proposition~4.1.15]{Koz15}. The difference is that the zeros should be put in the components of $\mathbb L$, which are not necessarily all in the first layer, and the $\phi^{(i)}$ components in the polynomial are not necessarily of degree 1. 
\end{rem}
\begin{lem}[Locally connectible with projected vector fields]\label{rem:ConnectingCurves}
	Let $\mathbb W$ and $\mathbb L$ be complementary subgroups of $\G$ with $\mathbb W$ normal, and let $\widetilde{\phi}\colon\widetilde{U}\subseteq\mathbb W\to\mathbb L$ be a continuous function on an open set $\widetilde U$. Then, for every $\widetilde U'\Subset \widetilde U$ and every $\widetilde w\in \widetilde U'$, there exists a neighborhood $\widetilde V\Subset \widetilde U'$ of $\widetilde w$ such that, for every $\widetilde v,\widetilde v'\in\widetilde V$ there exists a path, entirely contained in $\widetilde U'$, connecting $\widetilde v$ to $\widetilde v'$, made of a finite concatenation of integral curves for the vector fields $D^{\phi}_W$, for $W\in{\rm Lie}(\mathbb W)$.
\end{lem}
\begin{proof}
	 We give the proof in the case $\mathbb L$ is horizontal and $k$-dimensional. The same proof can be given in the general case in which $\mathbb W$ is normal taking \cref{rem:SameProofWNormal} into account. We fix an adapted basis $(X_1,\dots,X_n)$ of the Lie algebra $\mathfrak g$, such that $\mathbb W=\exp(\mbox{span}\{X_{k+1},\dots,X_n\})$ and $\mathbb L=\exp(\mbox{span}\{X_{1},\dots,X_k\})$. We denote by $\phi\colon U\subseteq \mathbb R^{n-k}\to\mathbb R^k$ the composition of $\widetilde \phi$ with the exponential coordinates and we set $D^{\phi}_j\coloneqq D_{X_j}^{\phi}$ for every $j=k+1,\dots,n$. Using the particular form of $D^{\phi}_j$, see \eqref{eqn:CoordinateDPhiJ}, we shall prove that $U$ is locally connectible by means of integral curves of the vector fields $D^{\phi}_j$, with $j=k+1,\dots,n$. Indeed, if we fix $w\in U$, and $j\in\{k+1,\dots,n\}$, Peano's Theorem \cite[Theorem 1.1]{Hal80} and the estimate for the existing time \cite[Corollary 1.1]{Hal80} imply that, for every open neighborhood $U'\Subset U$ of $w$, there exist $0<\alpha\coloneqq\alpha(U',\phi)$ and an open neighborhood $V''\coloneqq V''(U',\phi)$ of $w$ with $V''\Subset U'$, such that, for every $v\in V''$, there exists at least one integral curve $\gamma$ of $D^{\phi}_j$ starting from $v$, defined in $[-\alpha,\alpha]$, and such that $\gamma([-\alpha,\alpha])\subseteq U'$. 
	
	We can iterate the argument for each $j$, eventually considering a smaller neighborhood $V''$ and a smaller time at each stage of the iteration. Thus we obtain a neighborhood $V'\Subset U'$ of $w$ and a time $\delta>0$ such that, starting from an arbitrary point in $V'$,  there exists a concatenation of $n-k$ integral curves of $D^{\phi}_j$, with $j=k+1,\dots,n$, with interval of definition containing $[-\delta,\delta]$, and this concatenation is supported in $U'$. Eventually, up to reducing $V'$, we deduce that for every $U'\Subset U$ containing $w$ there exists a neighborhood $V\Subset V'\Subset U'$ of $w$ such that, for every $v,v'\in V$, there exists at least one path from $v$ to $v'$, made of integral curves of the vector fields $D^{\phi}_j$, with $j=k+1,\dots,n$, entirely contained in $U'$. We remark that this can be done taking into account that any integral curve of $D^{\phi}_j$ is a line along the $j$-th coordinate, and adjusting one coordinate at time. When we adjust one coordinate, we do not have the check the previous ones because of the triangular form of the vector fields $D^{\phi}_j$, see \eqref{eqn:CoordinateDPhiJ}.
\end{proof}

\subsection{Invariance properties of $D^{\phi}$}\label{sub:InvarianceProperties}
Now we prove some invariance properties of the vector fields $D^{\phi}$ under the operation of translating graphs that we have introduced in \cref{def:PhiQ}. 
\begin{lem}\label{lem:Invariance1DPhi}
	Let $\mathbb W$ and $\mathbb L$ be two complementary subgroups of a Carnot group $\mathbb G$, with $\mathbb L$ $k$-dimensional and horizontal and let $\widetilde{\phi}\colon\widetilde{U}\subseteq \mathbb W\to \mathbb L$ be a continuous function defined on $\widetilde U$ open. Take $W\in{\rm Lie}(\mathbb W)$, let $\widetilde{f}\colon\widetilde{U}\subseteq \mathbb W\to\mathbb L$ be a $C^{\infty}$ function and let us denote $D^{\phi}\coloneqq D^{\phi}_W$. Fix $q\in\mathbb G$ and denote by $\widetilde{\phi}_q$ and $\widetilde{f}_q$ the translated functions defined as in \eqref{eqn:Phiq} with domain $\widetilde{U}_q$. Denote by $f,\phi\colon U\subseteq \mathbb R^{n-k}\to \mathbb R^k$ and $f_q,\phi_q\colon U_q\subseteq\mathbb R^{n-k}\to\mathbb R^k$ the composition of $\widetilde f,\widetilde \phi,\widetilde f_q,\widetilde \phi_q$ with the exponential coordinates, respectively. Then, for every $w\in U_q$, the following equality holds in exponential coordinates
	\begin{equation}\label{eqn:InvarianceDPhi}
	D^{\phi_q}_{|_w}(f_q) = D^{\phi}_{|_{\pi_{\sW}(q^{-1}\cdot w)}}(f), 
	\end{equation}
	where $D^{\phi}(f)$, for a vector valued $f$, stands for the vector $(D^{\phi}(f^{(1)}),\dots,D^{\phi}(f^{(k)}))$. 
\end{lem}
\begin{proof}
	We stress a little abuse of notation throughout the proof, for the sake of simplicity. We exploit the identifications as in \cref{def:coordinateconletilde} without explicitly write the exponential map $F$ or $F^{-1}$.
	
	By \eqref{eqn:DefinitionOfDj} and \eqref{eqn:GraphOfPhiQEqualsQGraphOfPhi}, if we fix $w\in U_q$ and set $g\coloneqq\pi_{\mathbb W}(q^{-1}\cdot w)\cdot \widetilde\phi(\pi_{\sW}(q^{-1}\cdot w))$, we get 
	\begin{equation}\label{eqn:1}	
	D^{\phi_q}_{|_w}(f_q)=W_{|_{w\cdot\widetilde{\phi}_q(w)}}(f_q\circ \pi_{\sW})=W_{|_{q\cdot g }}(f_q\circ \pi_{\sW})=W_{|_g}(f_q\circ \pi_{\sW}\circ L_q),
	\end{equation}
	where in the last equality we used the fact that $W$ is left-invariant, namely $W_{|_{q\cdot g}}=\de(L_q)_g(W_{|_g})$. By definition of $g$ and the definition of $D^\phi$, see \eqref{eqn:DefinitionOfDj}, one has
	\begin{equation}\label{eqn:2}
	D^{\phi}_{|_{\pi_{\sW}(q^{-1}\cdot w)}}(f)=W_{|_g}(f\circ \pi_{\sW}). 
	\end{equation}
	Thus, taking \eqref{eqn:1} and \eqref{eqn:2} into account, we are left to show that 
	\begin{equation}\label{eqn:3}
	W_{|_g}(f_q\circ \pi_{\sW}\circ L_q)=W_{|_g}(f\circ \pi_{\sW}).
	\end{equation}
	Indeed, if $a\in\mathbb G$, we have 
	\begin{equation}\label{eqn:ComputationFqCircPiWCircLq}
	\begin{aligned}
	\widetilde f_q\circ\pi_{\sW}\circ L_q(a)&=\widetilde f_q\circ\pi_{\sW}(q\cdot a)=\widetilde f_q\circ \pi_{\sW}\left(\left(q_{\mathbb W}\cdot q_{\mathbb L}\cdot a_{\mathbb W}\cdot q_{\mathbb L}^{-1}\right)\cdot q_{\mathbb L}\cdot a_{\mathbb L}\right)\\
	&=\widetilde f_q(q_{\mathbb W}\cdot q_{\mathbb L}\cdot a_{\mathbb W}\cdot q_{\mathbb L}^{-1})=q_{\mathbb L}\cdot \widetilde f(a_{\mathbb W})=q_{\mathbb L}\cdot \widetilde f\circ\pi_{\sW}(a),
	\end{aligned}
	\end{equation}
	where in the third equality we used that $\mathbb W$ is a normal subgroup, and in the fourth equality we used (c) of \cref{prop:PropertiesOfIntrinsicTranslation}. Then, the functions $\widetilde f_q\circ\pi_{\sW}\circ L_q$ and $\widetilde f\circ \pi_{\sW}$ differ only by a left translation of the element $q_{\mathbb L}$. Thus, in exponential coordinates, they are $\mathbb R^k$-valued functions that differ by the fixed Euclidean translation of the $\mathbb R^k$-vector corresponding to $q_{\mathbb L}$. This last observation comes from the fact that, in exponential coordinates, the operation of the group restricted to $\mathbb L$ is the Euclidean sum, being $\mathbb L$ horizontal, see \cite[Proposition 2.1]{FSSC03a}. Finally, \eqref{eqn:3} holds true by the fact that, component by component, we are differentiating along a vector field two functions that differ by a fixed constant.
\end{proof}
\begin{lem}\label{lem:Invariance2IntegralCurve}
	Consider the setting of the \cref{lem:Invariance1DPhi} above, let $T>0$ and let $\widetilde{\gamma}\colon [0,T]\to \widetilde U$ be a $C^1$ regular solution of the Cauchy problem 
	\begin{equation}\label{eqn:CauchyProblem}
	\begin{cases}
	\widetilde{\gamma}'(t)=D^{\phi}\circ\widetilde{\gamma}(t), \\
	\widetilde{\gamma}(0)= w.
	\end{cases}
	\end{equation}
	Then for every $q\in \G$ there exists a unique $C^1$ map $\widetilde{\gamma}_q\colon[0,T]\to\widetilde U_q$ such that
	\begin{equation}\label{eqn:PropertyOfGammaTilde}
	\pi_{\sW}(q^{-1}\cdot \widetilde{\gamma}_q(t))=\widetilde{\gamma}(t), \qquad \forall t\in[0,T].
	\end{equation}
	In addition, $\widetilde\gamma_q$ is a solution of the Cauchy problem 
	\begin{equation}\label{eqn:IntegralCurveOfDPhiQ}
	\begin{cases}
	\widetilde{\gamma}_q'(t)=D^{\phi_q}\circ\widetilde{\gamma}_q(t), \\
	\widetilde{\gamma}_q(0)=q_{\mathbb W}\cdot q_{\mathbb L}\cdot  w\cdot q_{\mathbb L}^{-1}.
	\end{cases}
	\end{equation}
	
	Moreover, if there exists a continuous function $\omega\colon U\subseteq \mathbb R^{n-k}\to\mathbb R^k$ such that
	\[
	\phi(\gamma(t))-\phi(\gamma(0))=\int_{0}^t \omega(\gamma(s))\de\!s, \qquad \forall t\in[0,T],
	\]
	then, there exists a continuous function $\bar \omega_q\colon U_q\subseteq\mathbb R^{n-k}\to\mathbb R^k$ such that 
	\[
	\phi_q(\gamma_q(t))-\phi_q(\gamma_q(0))=\int_{0}^t \bar \omega_q(\gamma_q(s))\de\!s, \qquad \forall t\in [0,T].
	\]
\end{lem}
\begin{proof}
	We stress a little abuse of notation throughout the proof, for the sake of simplicity. We exploit the identifications as in \cref{def:coordinateconletilde} without explicitly write the exponential map $F$ or $F^{-1}$.
	
	For every $q\in \G$, we define 
	\[
	\widetilde{\gamma}_q(t)\coloneqq q_{\sW}\cdot q_{\mathbb L}\cdot \widetilde{\gamma}(t)\cdot q_{\mathbb L}^{-1}, \qquad \forall t\in[0,T].
	\]
	 Then $\widetilde\gamma_q$ takes values in $\widetilde U_q$, see item (c) of \cref{prop:PropertiesOfIntrinsicTranslation}, and $\widetilde \gamma_q(0)=q_{\mathbb W}\cdot q_{\mathbb L}\cdot w\cdot q_{\mathbb L}^{-1}$. Moreover, one also has
	\[
	\pi_{\mathbb W}(q^{-1}\cdot\widetilde \gamma_q(t))=\pi_{\mathbb W}(q_{\mathbb L}^{-1}\cdot q_{\mathbb W}^{-1}\cdot q_{\mathbb W}\cdot q_{\mathbb L}\cdot \widetilde\gamma(t)\cdot q^{-1}_{\mathbb L})=\widetilde \gamma(t), \quad \forall t\in [0,T].
	\]
	Moreover, if we impose \eqref{eqn:PropertyOfGammaTilde}, the uniqueness of $\widetilde\gamma_q$ is guaranteed by the second equation of \eqref{eqn:ProjectionOfq-1a}, and the equivalence
	\begin{equation}\label{eqn:HowGammaChanges}
	q_{\mathbb L}^{-1}\cdot q_{\sW}^{-1}\cdot \widetilde{\gamma}_q(t)\cdot q_{\mathbb L}=\widetilde{\gamma}(t) \Leftrightarrow \widetilde{\gamma}_q(t)=q_{\sW}\cdot q_{\mathbb L}\cdot \widetilde{\gamma}(t)\cdot q_{\mathbb L}^{-1}, \qquad \forall t\in[0,T].
	\end{equation}
	Now we want to check that $\widetilde \gamma_q$ is a solution to the Cauchy problem \eqref{eqn:IntegralCurveOfDPhiQ}. We work in exponential coordinates. For every vector-valued function $f\colon U_q\subseteq \mathbb R^{n-k}\to\mathbb R^k$ of class $C^{\infty}$, we have that, for every $t\in [0,T]$, it holds, in exponential coordinates, the following chain of equalities:
	\begin{equation}\label{eqn:EqnInvarianceIntegralCurve}
	\begin{aligned}
	D^{\phi_q}_{|_{{\gamma}_q(t)}}(f)&=
	D^{\phi_q}_{|_{{\gamma}_q(t)}}((f_{q^{-1}})_q) = D^{\phi}_{|_{{\gamma}(t)}}(f_{q^{-1}})\\
	&={\gamma}'(t)(f_{q^{-1}})=\frac{\de}{\de t}(f_{q^{-1}}\circ \gamma(t)) \\
	&=\frac{\de}{\de t}\left(f_{q^{-1}}\circ \pi_{\sW}\circ L_{q^{-1}}({\gamma}_q(t))\right) \\
	&=\frac{\de}{\de t}\left((q^{-1})_{\mathbb L}\cdot f\circ \pi_{\mathbb W}({\gamma}_q(t))\right)=\frac{\de}{\de t}\left((q^{-1})_{\mathbb L}\cdot f\circ {\gamma}_q(t))\right)\\
	&=\frac{\de}{\de t}\left(f\circ {\gamma}_q(t)\right)={\gamma}'_q(t)(f),
	\end{aligned}
	\end{equation}
	where in the first equality we used item (b) of \cref{prop:PropertiesOfIntrinsicTranslation}, and in the second one we used \eqref{eqn:InvarianceDPhi} and the fact that $\pi_{\sW}(q^{-1}\cdot\widetilde{\gamma}_q(t))=\widetilde{\gamma}(t)$. In the fifth equality we used the coordinate version of $\pi_{\sW}(q^{-1}\cdot\widetilde{\gamma}_q(t))=\widetilde{\gamma}(t)$ with a little abuse of notation, and in the sixth equality we used the coordinate version of \eqref{eqn:ComputationFqCircPiWCircLq} with $q^{-1}$ in place of $q$. Finally, in the eighth equality we used the fact that, being $\mathbb L$ horizontal, the product with a fixed element of $\mathbb L$, in exponential coordinates, is a Euclidean translation and hence it does not affect the time derivative. This comes from the explicit expression of the product in coordinates, see \cite[Proposition~2.1]{FSSC03a}. Thus the proof of \eqref{eqn:IntegralCurveOfDPhiQ} is finished by \eqref{eqn:EqnInvarianceIntegralCurve}.
	
	By a further inspection, following the equalities starting from the right hand side of the second line to the right hand side of the fourth line of \eqref{eqn:EqnInvarianceIntegralCurve}, we have proved, \textbf{only by exploiting the fact that $\mathbb W$ is normal}, that if $\widetilde U$ is open then for every function $\widetilde G\colon\widetilde U\subseteq \mathbb W\to\mathbb L$, for every $q\in \G$, and every $t\in [0,T]$, it holds
	\begin{equation}\label{eqn:Translation1}
	\widetilde G\circ\widetilde\gamma(t)=(q^{-1})_{\mathbb L}\cdot  \widetilde G_q\circ\widetilde \gamma_q(t).
	\end{equation}
 	Since  $\mathbb L$ is horizontal, in exponential coordinates this equality reads as  
 	\begin{equation}\label{eqn:Translation}
 	 G\circ\gamma(t)= G_q\circ \gamma_q(t)+ (q^{-1})_{\mathbb L}.
 	\end{equation} 
 	Assume there exists a continuous map $\omega\colon U\subseteq \mathbb R^{n-k}\to\mathbb R^k$ as in the statement. Then, by composing with exponential coordinates we get a continuous $\widetilde\omega\colon\widetilde U\subseteq \mathbb W\to\mathbb L$. We then define $\widetilde\omega_q$ as in \eqref{eqn:Phiq} and we set $\omega_q\colon U_q\subseteq \mathbb R^{n-k}\to\mathbb R^k$ to be the composition of $\widetilde \omega_q$ with the exponential coordinates. We are then in a position to define $\overline \omega_q$ by setting
	\begin{equation}\label{eqn:HowOmegaChanges}
	\overline\omega_q(x)\coloneqq\omega_q(x)+(q^{-1})_{\mathbb L}, \quad\forall x\in U_q.
	\end{equation}
	Then, we have, for every $t\in [0,T]$,
	\begin{equation}\label{eqn:FundamentalTheoremOfCalculusOnCurves}
	\begin{aligned}
	\phi_q(\gamma_q(t))-\phi_q(\gamma_q(0))&=\phi(\gamma(t))-\phi(\gamma(0))=\int_0^t \omega(\gamma(s))\de\!s = \\
	&= \int_0^t \left(\omega_q(\gamma_q(s))+(q^{-1})_{\mathbb L}\right)\de\! s=\int_0^t\bar{\omega}_q(\gamma_q(s))\de\!s,
	\end{aligned}
	\end{equation}
	where in the first and in the third equality we used \eqref{eqn:Translation}. This completes the proof. 
\end{proof}
\begin{rem}\label{rem:CurveInvarianceTrueAlsoWhenWNormal}
	The curve $\widetilde\gamma_q$ defined in \eqref{eqn:HowGammaChanges} solves \eqref{eqn:IntegralCurveOfDPhiQ} also in the general case when $\mathbb W$ is normal, but one should run more involved computations, because the invariance property \eqref{eqn:InvarianceDPhi} might not be true. We presented the invariance in \eqref{eqn:InvarianceDPhi} in the specific case $\mathbb L$ is horizontal and $k$-dimensional because it will be frequently used in the last theorems of this section and in \cref{sec:Theorems}.
	
	We give a sketch of the proof in the general case. For a reference, one can also read the first item of \cite[Proposition 4.2.15]{Koz15}. Given $q\in\mathbb G$, define the map $\sigma_q$ on $\mathbb W$ by $\sigma_q(w)\coloneqq q_{\mathbb W}\cdot q_{\mathbb L}\cdot w\cdot (q_{\mathbb L})^{-1}$. By \eqref{eqn:ProjectionOfq-1a}, we get that $\sigma_{q^{-1}}(w)=\pi_{\mathbb W}(q^{-1}\cdot w)$. Then, in case $\mathbb L$ is horizontal, the invariance property proved in \eqref{eqn:InvarianceDPhi} reads as $D^{\phi_q}_{|_w}=D^{\phi}\circ\sigma_{q^{-1}}(w)$. In case $\mathbb L$ is not horizontal, using the properties stated in \cref{prop:PropertiesOfIntrinsicTranslation}, one can prove that
	\[
	D^{\phi_q}_{|_w} = \de(\sigma_q)_{\sigma_{q^{-1}}(w)} \left(D^{\phi}\circ\sigma_{q^{-1}}(w)\right),
	\]
	for every $w\in\widetilde U$ and $q\in \mathbb G$. One can hence use the previous equality and the definition of $\widetilde\gamma_q$ as in \eqref{eqn:HowGammaChanges} to show that, if a curve $\widetilde\gamma$ satisfies \eqref{eqn:CauchyProblem},  then $\widetilde\gamma_q$ satisfies \eqref{eqn:IntegralCurveOfDPhiQ}.
\end{rem}

\subsection{Metric properties of integral curves of $D^{\phi}$}\label{sub:MetricProperties}

In this subsection we study how the intrinsically Lipschitz regularity of $\widetilde\phi$ affects the metric regularity of the integral curves of the vector fields $D^{\phi}$ (see \cref{prop:IntrinsicLipschitzIntegralCurve}). We also see how the conclusions obtained can be improved when we assume $\widetilde\phi$ to be intrinsically differentiable (\cref{prop:DerivativeOfIntegralCurves}) or uniformly intrinsically differentiable (see \cref{prop3.6cont} and \cref{prop:UidImpliesH1k}). 

The following lemma is essentially the implication $(1)\Rightarrow (3)$ of \cite[Theorem~4.2.16]{Koz15}. For the reader convenience, and for some benefit toward \cref{rem:LeDonne}, we give the proof here, without going through all the precise estimates, and claiming no originality. 
\begin{lem}\label{lem:EstimateOnTheNormOfGamma}
	Let  $\mathbb W$ and $\mathbb L$ be two complementary subgroups of a Carnot group $\G$, with $\mathbb W$ normal, let $\widetilde{U}\subseteq\mathbb W$ be an open neighborhood of the identity $e$ and let $\widetilde{\phi}\colon\widetilde{U}\to\mathbb L$ be a function such that $\widetilde{\phi}(e)=e$. Assume there exists a constant $C>0$ with 
	\begin{equation}\label{eqn:LipschitzCondition}
	\|\widetilde\phi(w)\|_{\mathbb G}\leq C\|w\|_{\mathbb G}, \qquad \forall  w\in\widetilde{U}.
	\end{equation}
	Then for every integer $d\geq 1$, every $W\in{\rm Lie}(\mathbb W)\cap V_d$, and every integral curve $\widetilde{\gamma}\colon[0,T]\to\widetilde{U}$ of $D_W^{\phi}$ starting from $e$, there exists $C'$ depending only on $C$ and $W$, such that
	\begin{equation}\label{eqn:NormOfGamma}
	\|\widetilde\gamma(t)\|_{\mathbb G} \leq C't^{1/d}, \qquad \|\widetilde\phi(\widetilde\gamma(t))\|_{\mathbb G}\leq C't^{1/d}, \qquad \forall t\in [0,T].
	\end{equation}
\end{lem}
\begin{proof}
	We give the proof in case $\mathbb L$ is $k$-dimensional and horizontal, just for notational purposes. Then, taking \cref{rem:SameProofWNormal} into account, the proof can be given in the general case. We recall that, from the fact that all the homogeneous norms are equivalent, we may fix $\|(x_1,\dots,x_n)\|_{\mathbb G}:=\sum_{i=1}^n |x_i|^{1/\deg i}$. Then, in the case $\mathbb L$ is horizontal, $ \|\widetilde\phi\circ\widetilde\gamma\|_{\mathbb G}=|\phi\circ\gamma|$.  

	We fix an adapted basis $(X_1,\dots,X_n)$ of the Lie algebra $\mathfrak g$ such that $W=X_j$ for some $j\in\{k+1,\dots,n\}$. Then $\deg j=d$ and, according to \cref{prop:CompatibleCoordinates}, we have, in exponential coordinates adapted to this basis, that $D^{\phi}_j\coloneqq D^{\phi}_{X_j}=D^{\phi}_W$ writes as
	\begin{equation}\label{eqn:eq1}
	D^{\phi}_j|_{(0,\dots,0,x_{k+1},\dots,x_n)}=\partial_{x_j}+\sum_{i=n_{d}+1}^n P_i^j(\phi^{(1)},\dots,\phi^{(k)},x_{k+1},\dots,x_{n_{\deg i-1}})\partial_{x_i},
	\end{equation}
	for some polynomials $P^j_i$ of homogeneous degree $\mathrm{deg}i-d$. By the use of  triangle inequality and Young's inequality, we get that for every $i=n_{d}+1,\dots,n$ there exists a constant $C_{1,i}>0$ depending only on the polynomial $P_i^j$, and a constant $C_{2,i}>0$ that depends only on $C_{1,i}$ and on $n$, such that 
	\begin{equation}\label{eqn:Young}
	\begin{aligned}
	|P_i^j(\phi^{(1)},\dots,\phi^{(k)},x_{k+1},\dots,x_{n_{\deg i-1}})|&\leq C_{1,i}\|(0,\dots,0,\phi^{(1)},\dots,\phi^{(k)},x_{k+1},\dots,x_{n_{\deg i-1}})\|_{\mathbb G}^{\deg i-d}  \\
	&\leq C_{2,i}(\|(0,\dots,0,\phi^{(1)},\dots,\phi^{(k)},0,\dots,0)\|_{\mathbb G}^{\deg i-d}\\
	&\hphantom{\leq C_3}+ \|(0,\dots,0,x_{k+1},\dots,x_n)\|_{\mathbb G}^{\deg i-d}).
	\end{aligned}
	\end{equation}
	Fix $t\in [0,T]$ and define
	\[
	m_{\gamma}(t)\coloneqq\max_{s\in [0,t]}\|\widetilde{\gamma}(s)\|_{\mathbb G}, \qquad m_{\phi}(t)\coloneqq\max_{s\in [0,t]}|\phi(\gamma(s))|.
	\]
	Then, by the fact that $\widetilde\gamma$ is an integral curve of $D^{\phi}_j$ and the particular form of $D^{\phi}_j$ in \eqref{eqn:eq1}, we get that $\gamma_j(t)=t$ for every $t\in [0,T]$, and $\gamma_\ell(t)\equiv 0$ for every $t\in [0,T]$ and every $\ell\neq j$ with $1\leq \ell\leq n_d$. The estimate \eqref{eqn:Young} implies that, for every $i\geq n_d+1$, one has 
	\begin{equation}\label{eqn:EstimateKoz}
	\begin{split}
	|\gamma_i(s)| &\leq \int_0^s\left|P_i^j\left(\phi^{(1)}(\gamma(r)),\dots,\phi^{(k)}(\gamma(r)),\gamma_{k+1}(r),\dots,\gamma_{n_{\deg i-1}}(r)\right)\right|\de\!r\\
	&\leq tC_{2,i}\left((m_{\gamma}(t))^{\deg i-d}+(m_{\phi}(t))^{\deg i-d}\right)\leq  tC_{3,i}(m_{\gamma}(t))^{\deg i-d}, \qquad 	\forall s\in [0,t],
	\end{split}
	\end{equation}
	where in the last inequality we used $m_{\phi}(t)\leq Cm_{\gamma}(t)$ for every $t\in [0,T]$, that comes from the hypothesis \eqref{eqn:LipschitzCondition}, and where $C_{3,i}$ depends only on $C_{2,i}$ and $C$. Thus, from \eqref{eqn:EstimateKoz} and the fact that the only other nonzero component of $\gamma$ is $\gamma_j(t)=t$, we get 
	\[
	\begin{split}
	\|\widetilde\gamma(s)\|_{\mathbb G}&\leq C_4 \left( t^{1/d}+\sum_{i=m_d+1}^nt^{1/\deg i}(m_{\gamma}(t))^{1-d/\deg i}\right) \\
	&\leq nC_4\max_{i\in\{n_d+1,\dots,n\}}\left\{t^{1/d},t^{1/\deg i}(m_{\gamma}(t))^{1-d/\deg i}\right\},
	\end{split}
	\]
	for every $s\in [0,t]$, where $C_4$ depends only on all the constants $C_{3,i}$, with $i\geq m_d+1$. Maximizing the previous inequality with respect to $s\in [0,t]$, gives
	\[
	m_{\gamma}(t) \leq nC_4\max_{i\in\{n_d+1,\dots,n\}}\left\{t^{1/d},t^{1/\deg i}(m_{\gamma}(t))^{1-d/\deg i}\right\}, 
	\]
	for every $t\in [0,T]$. As a consequence of this inequality we get $m_{\gamma}(t) \leq C_5 t^{1/d}$ for all $t\in [0,T]$, for a constant $C_5$ depending only on $C_4$ and on $\G$. We have thus proved
	\[
	\|\widetilde \gamma(t)\|_\G\leq C_5 t^{1/d},\qquad  \forall t\in [0,T].
	\]
	To conclude the proof, it is enough to use \eqref{eqn:LipschitzCondition} and choose $C'\coloneqq\max\{C_5,CC_5\}$. Finally, taking into account all the dependencies of the constants, the constant $C'$ only depends on $C$ and on the coefficients of $D^{\phi}_j$ in coordinates, and thus ultimately on the vector field $W\in \mathrm{Lie}(\mathbb W)$.
\end{proof}
\begin{rem}
	Condition \eqref{eqn:LipschitzCondition} in \cref{lem:EstimateOnTheNormOfGamma} can be deduced as soon as $\widetilde{\phi}\colon\widetilde{U}\subseteq \mathbb W\to\mathbb L$ is intrinsically Lipschitz at $e$, see the item (c) of \cref{prop:IntrinsicLipschitz}. We will give a general statement in this direction in the forthcoming \cref{prop:IntrinsicLipschitzIntegralCurve}, that is a restatement of the implication $(1)\Rightarrow(2) \& (3)$ of \cite[Theorem 4.2.16]{Koz15}.
	
	Notice that in \cref{lem:EstimateOnTheNormOfGamma} we only exploited the particular triangular form of $D^{\phi}_j$, see \eqref{eqn:CoordinateDPhiJ}. The same result as in \cref{lem:EstimateOnTheNormOfGamma} holds if we take the integral curves, starting from $e$, of any vector field of the same form as the right hand side of \eqref{eqn:eq1}, satisfying the  homogeneity conditions on the polynomials $P^j_i$ given in \cref{prop:CompatibleCoordinates}. On the contrary, for the forthcoming \cref{prop:IntrinsicLipschitzIntegralCurve}, we need that we are dealing precisely with the vector fields $D^{\phi}_j$ in order to use the invariance properties in \cref{lem:Invariance1DPhi}, and \cref{lem:Invariance2IntegralCurve}.
\end{rem}

	\begin{prop}\label{prop:IntrinsicLipschitzIntegralCurve}
		Let $\mathbb W$ and $\mathbb L$ be two complementary subgroups of a Carnot group $\mathbb G$, with $\mathbb W$ normal, let $\widetilde{U}$ be an open subset of $\mathbb W$, and let $\widetilde{\phi}\colon\widetilde{U}\to\mathbb L$ be an intrinsically $L$-Lipschitz function with Lipschitz constant $L>0$.
	
	Then for every integer $d\geq 1$, every $W\in{\rm Lie}(\mathbb W)\cap V_d$, and every integral curve $\widetilde{\gamma}\colon[0,T]\to\widetilde{U}$  of $D_W^{\phi}$, there exists a constant $C'>0$ depending only on $L$ and $W$ such that
	\begin{equation}\label{eqn:NormOfGammaGeneral1}
	\|\widetilde\phi(\widetilde\gamma(s))^{-1}\cdot\widetilde\gamma^{-1}(s)\cdot\widetilde\gamma(t)\cdot\widetilde\phi(\widetilde\gamma(s))\|_{\mathbb G} \leq C'|t-s|^{1/d}, \quad \mbox{for }\quad 0\leq s<t\leq T;
	\end{equation}
	\begin{equation}\label{eqn:NormOfGammaGeneral2}
	  \|\widetilde\phi(\widetilde\gamma(s))^{-1}\cdot\widetilde\phi(\widetilde\gamma(t))\|_{\mathbb G}\leq C'|t-s|^{1/d}, \quad \mbox{for }\quad 0\leq s<t\leq T.
	\end{equation}
	\end{prop}

\begin{proof}
	 Fix $s\in [0,T]$ and define $q\coloneqq\widetilde\phi(\widetilde\gamma(s))^{-1}\cdot\widetilde\gamma(s)^{-1}$. By exploiting item (d) of \cref{prop:PropertiesOfIntrinsicTranslation}, item (c) of \cref{prop:IntrinsicLipschitz}, and the fact that $\widetilde\phi$ is intrinsically $L$-Lipschitz, we get that $\widetilde\phi_q(e)=e$ and $\|\widetilde\phi_q(w)\|_{\mathbb G}\leq L\|w\|_{\mathbb G}$ for every $w\in\widetilde U_q$. We now apply \cref{lem:Invariance2IntegralCurve} to get that the curve $\widetilde \gamma_q$ defined by $\widetilde\gamma_q(t)\coloneqq\widetilde\phi(\widetilde\gamma(s))^{-1}\cdot\widetilde\gamma^{-1}(s)\cdot\widetilde\gamma(t)\cdot\widetilde\phi(\widetilde\gamma(s))$ is an integral curve of the vector field $D^{\phi_q}_W$. Notice that this curve takes values in $\widetilde U_q$ as noticed in \cref{lem:Invariance2IntegralCurve}. We stress that \cref{lem:Invariance2IntegralCurve} is stated only for $\mathbb L$ horizontal, but it also holds in case $\mathbb W$ is normal, see \cref{rem:CurveInvarianceTrueAlsoWhenWNormal}.
	
	Define $\widetilde\gamma^+_q(\cdot)\coloneqq\widetilde\gamma_q(\cdot+s)$. Since  $\widetilde\gamma^+_q(0)=e$, we are in a position to apply \cref{lem:EstimateOnTheNormOfGamma} to the function $\widetilde\phi_q$ and to the  curve $\widetilde\gamma^+_q$. Evaluating the first inequality of \eqref{eqn:NormOfGamma} at time $t-s$ we get \eqref{eqn:NormOfGammaGeneral1}. Finally, by \eqref{eqn:Translation1} - that holds in the general case in which $\mathbb W$ is normal - and by evaluating the second inequality of \eqref{eqn:NormOfGamma} at time $t-s$, we get  \eqref{eqn:NormOfGammaGeneral2}.
\end{proof}
	
	\begin{rem}[An improvement of {\cite[Proposition 6.6]{ALD19}}]\label{rem:LeDonne}
	A simple modification of the proof of \cref{prop:IntrinsicLipschitzIntegralCurve} provides a general argument for the second part of \cite[Proposition~6.6]{ALD19}, that was proved only in the setting of Carnot groups of step 2 and in case $\mathbb L$ is 1-dimensional. The generalization reads as follows. Fix two complementary subgroups $\mathbb W$ and $\mathbb L$ of a Carnot group $\mathbb G$, with  $\mathbb W$ normal, and an intrinsically Lipschitz function $\widetilde\phi\colon\widetilde U\subseteq \mathbb W\to\mathbb L$, where $\widetilde U$ is open. Consider a solution $\widetilde\gamma:I\to\widetilde U$ of 
	\begin{equation}\label{eq:controlli}
	\widetilde\gamma'(t)=\sum_{j=k+1}^m a_j(t)(D^{\phi}_{X_j})_{|_{\widetilde\gamma(t)}},
	\end{equation}
	for some controls $a_j(t)$ of class $L^{\infty}(I)$, where $\{X_{k+1},\dots,X_m\}$ is a basis of $\mathrm{Lie}(\mathbb W)\cap V_1$. Then the curve $\widetilde\phi\circ\widetilde\gamma$ is Lipschitz. 
	
	To prove this last statement one first proves the analogous of \cref{lem:EstimateOnTheNormOfGamma} for curves satisfying \eqref{eq:controlli}. The estimates are done in the same way but the constant $C'$ also depends on a uniform bound of the controls $a_j(\cdot)$ in $L^{\infty}(I)$. In order to conclude, one can run the same argument of \cref{prop:IntrinsicLipschitzIntegralCurve}, taking into account that the invariance property shown in \cref{lem:Invariance2IntegralCurve} also holds for curves satisfying \eqref{eq:controlli} with exactly the same proof. For the sketch of the proof of \cref{lem:Invariance2IntegralCurve} in the general case when $\mathbb W$ is normal we refer the reader to \cref{rem:CurveInvarianceTrueAlsoWhenWNormal}.
	\end{rem}
	
	The forthcoming \cref{prop:DerivativeOfIntegralCurves} is, to our knowledge, new. It tells us what are the metric properties of the integral curves $\widetilde\gamma$ of $D^{\phi}$ whenever $\widetilde\phi\in{\rm ID}(\widetilde U,\mathbb W;\mathbb L)$. The counterpart of \cref{prop:DerivativeOfIntegralCurves} in the setting of the Heisenberg group $\mathbb H^n$ is already known: for the case in which $\mathbb L$ is  one-dimensional, the proof follows from the argument of \cite[Theorem~4.95, $(3)\Rightarrow(2)$]{SC16}, while for the case in which $\mathbb L$ is $k$-dimensional, the proof is in \cite[Proposition~4.6]{Cor19}. A weaker version of this proposition, which also requires $\widetilde\phi\circ\widetilde\gamma$ to be $C^1$, has appeared in \cite[Proposition~3.7]{ASCV06} in the Heisenberg group, for $\mathbb L$ one-dimensional, and in \cite[Proposition~5.6]{DiDonato18} in Carnot groups of step 2, for $\mathbb L$ one-dimensional.

\begin{prop}\label{prop:DerivativeOfIntegralCurves}
	Let $\mathbb W$ and $\mathbb L$ be two complementary subgroups in a Carnot group $\mathbb G$, with $\mathbb L$ horizontal and k-dimensional, let $\widetilde U$ be an open set in $\mathbb W$ and let $\widetilde{\phi}\colon\widetilde{U}\to\mathbb L$ be an intrinsically differentiable function at $w_0\in\widetilde{U}$. Then the following facts hold.
	\begin{itemize}
	\item[(i)] For every $W\in {\rm Lie}(\mathbb W)\cap V_1$ and every integral curve $\widetilde{\gamma}\colon[0,T]\to\widetilde{U}$ of $D_W^{\phi}$ starting from $w_0$, the composition $\widetilde\phi\circ\widetilde\gamma$ is differentiable at 0 and 
	\begin{equation}\label{eqn:IdBehaviourOf1Curves}
	\frac{\de}{\de t}_{|_{t=0}}(\widetilde\phi\circ\widetilde\gamma)(t) = \de^{\phi}\!\phi(w_0)(\exp W).
	\end{equation}
	
	\item[(ii)] For every $d>1$, $W\in{\rm Lie}(\mathbb W)\cap V_d$ and every integral curve $\widetilde \gamma\colon[0,T]\to\widetilde U$ of $D^\phi_W$ starting from $w_0$, the following holds:
	\begin{equation}\label{eqn:IdBehaviourOfkCurves}
	\lim_{t\to 0}\frac{\|\widetilde\phi(w_0)^{-1}\cdot\widetilde\phi(\widetilde\gamma(t))\|_{\mathbb G}}{t^{1/d}}=0.
	\end{equation}
\end{itemize}
\end{prop}
\begin{proof}
	Notice that, since $\mathbb L$ is horizontal, the homogeneous norm $\|\cdot\|_{\mathbb G}$ restricted to $\mathbb L$ is equivalent to the Euclidean norm of the exponential coordinates. First of all, by item (i) of \cref{rem:InvarianceOfIdByTranslations}, $ \widetilde\phi$ is continuous at $w_0$.  If we define $q\coloneqq{\widetilde\phi}(w_0)^{-1}\cdot w_0^{-1}$ we get, from item (d) of \cref{prop:PropertiesOfIntrinsicTranslation}, that $\widetilde{\phi}_q(e)=e$ and, from item (ii) of \cref{rem:InvarianceOfIdByTranslations}, that $\widetilde{\phi}_q$ is intrinsically differentiable at $e$ with $\de^{\phi_q}\!\phi_q (e)=\de^{\phi}\!\phi (w_0)$. By  \eqref{eqn:IdInCoordinates}, \eqref{linea2.0} and the triangle inequality, for every $\widetilde V\Subset \widetilde {U}_q$ containing $e$, there exists a constant $C$ such that $|\phi_q(w)|\leq C\|w\|_{\mathbb G}$ for every $w\in \widetilde V$. 
	
	(i) Fix $W\in {\rm Lie}(\mathbb W)\cap V_1$ and $\widetilde\gamma$ as in the assumption. By the first part of \cref{lem:Invariance2IntegralCurve},  there exists $\widetilde{\gamma}_q\colon[0,T]\to \widetilde{U}_q\subseteq \mathbb W$ such that $\widetilde{\gamma}_q$ is an integral curve of $D^{\phi_q}_W$ starting from $e$. Then, since $|\phi_q(w)|\leq C\|w\|_{\mathbb G}$ for every $w\in\widetilde V$, and since for sufficiently small times $t>0$ it holds $\widetilde\gamma_q([0,t])\subseteq \widetilde V$, we are in the setting of \cref{lem:EstimateOnTheNormOfGamma}. Thus, from the first inequality in \eqref{eqn:NormOfGamma}, we can write
	\[
	\limsup_{t\to 0}\frac{\|\widetilde\gamma_q(t)\|_{\mathbb G}}{t} <+\infty.
	\] 
	 Fix an adapted basis $(X_1,\dots, X_n)$ of the Lie algebra $\mathfrak g$ such that $\mathbb L=\exp(\mathrm{span}\{X_1,\dots,X_k\})$ and $\mathbb W=\exp(\mathrm{span}\{X_{k+1},\dots,X_n\})$. We use the notation of \cref{def:coordinateconletilde} and \cref{def:NablaPhiPhi}. By using the previous inequality we get that there is a constant $\bar{C}>0$ such that, for every small enough $t\in [0,T]$, 
	\begin{equation}\label{eqn:Final1}
	\frac{|\phi_q(\gamma_q(t))-\phi_q(\gamma_q(0))-t\nabla^{\phi_q}\phi_q(e)(W)|}{t} \leq \bar{C}\frac{|\phi_q(\gamma_q(t))-\phi_q(\gamma_q(0))-t\nabla^{\phi_q}\phi_q(e)(W)|}{\|\widetilde\gamma_q(t)\|_{\mathbb G}}.
	\end{equation}
	Notice that $\phi_q(\gamma_q(0))=\phi_q(0)=0$. Moreover, using the particular form of the projected vector fields in \eqref{eqn:CoordinateDPhiJ} and the fact that $W\in V_1$, it is easy to see that $\pi_{\mathbb V_1}(\widetilde{\gamma}_q(t))=\exp(tW)$ for all $t\in[0,T]$. By exploiting the fact that the intrinsic differential is linear on the horizontal components (see \cref{p2.5}), we get that for all $t\in [0,T]$
	\[
	t \nabla^{\phi_q}\phi_q(e)(W)= \nabla^{\phi_q}\phi_q(e)(tW) =\exp^{-1}\left(\de^{\phi_q}\!\phi_q(e)(\exp (t W))\right).
	\]
	Let us conclude the proof. The intrinsic differentiability of $\widetilde\phi_q$ at $e$ provides \eqref{eqn:IdInCoordinates}. Thus, by exploiting $\pi_{\mathbb V_1}(\widetilde\gamma_q(t))=\exp(t W)$ for all $t\in [0,T]$, the previous equality, and the fact that the intrinsic differential depends only on the projection on $\mathbb V_1$ (see \cref{p2.5}), the right hand side of \eqref{eqn:Final1} goes to zero as $t\to 0$. Thus, also the left hand side goes to zero as $t\to 0$ and this means that 
	\[
	\frac{\de}{\de t}_{|_{t=0}}(\phi_q\circ\gamma_q)(t)=\nabla^{\phi_q}\phi_q(e)(W).
	\]
	By using \eqref{eqn:Translation}, and since $\nabla^{\phi_q}\phi_q(e)=\nabla^{\phi}\phi(w_0)$, we get \eqref{eqn:IdBehaviourOf1Curves}. This concludes the proof of (i).
	
	(ii) Assume $W\in{\rm Lie}(\mathbb W)\cap V_d$ with $d>1$, and $\widetilde\gamma$ as in the assumption. We proceed with the same argument as in (i). Then, following the lines of the proof in item (i) and by the first inequality in \eqref{eqn:NormOfGamma}, we obtain that there exists $\bar C>0$ such that for sufficiently small $t\in[0,T]$
	\begin{equation}\label{eq:final2}
	\frac{|\phi_q(\gamma_q(t))-\phi_q(\gamma_q(0))|}{t^{1/d}} \leq \bar C\frac{|\phi_q(\gamma_q(t))-\phi_q(\gamma_q(0))|}{\|\widetilde\gamma_q(t)\|_{\mathbb G}}.
	\end{equation}
	Since $W\in V_d$ with $d>1$ and $\phi_q(0)=0$, the projection of every integral curve of $D^{\phi_q}_W$, starting from 0, on the horizontal bundle is zero. This follows by exploiting the particular form of $D^{\phi_q}$ in coordinates, see \eqref{eqn:CoordinateDPhiJ}. Then the intrinsic differentiability of $\widetilde \phi_q$ at $e$ jointly with \eqref{eq:final2}, the fact that the projection of $\widetilde\gamma_q$ on $\mathbb V_1$ is zero, and that the intrinsic gradient is linear on $V_1$ (see \cref{p2.5}) yields, with the same reasoning as before,
	\[
	\lim_{t\to0}\frac{|\phi_q(\gamma_q(t))-\phi_q(\gamma_q(0))|}{t^{1/d}}=0.
	\]
	Then, by using \eqref{eqn:Translation} we conclude \eqref{eqn:IdBehaviourOfkCurves} and thus the proof. 
\end{proof} 
\begin{rem}\label{rem:EaseOfNotation}
	For the ease of notation, we considered in \cref{prop:DerivativeOfIntegralCurves} only intervals $[0,T]$, and thus we got conclusions only on the right limits and the right derivatives. The same proof provides the same conclusion on the full limit, or the full derivative, whenever the interval is centered at the origin. 
\end{rem}
Now we want to deduce metric properties of $\widetilde\phi$ when we know that it is {\rm UID}. The following proposition shows that any uniformly intrinsically differentiable function $\widetilde\phi$ is $\frac 1s$-little H\"older continuous on any Carnot group of step $s$, when read in exponential coordinates. It is a generalization of \cite[Proposition 4.4]{ASCV06}.
\begin{prop}\label{prop3.6cont}
	Let $\mathbb W$ and $\mathbb L$ be two complementary subgroups of a Carnot group $\mathbb G$ with $\LL$ horizontal and k-dimensional, and let $\widetilde U\subseteq \mathbb W$ be an open set. If $\widetilde \phi \in {\rm UID}(\widetilde U,\mathbb W;\LL)$, then such a function read in exponential coordinates is in $h^{1/s}_{\rm loc}(U; \R ^k)$, that is
	$\phi \in C(U;\R ^k)$ and for all $U ' \Subset U$ one has
	\begin{equation}\label{big3.3.11nuova}
	\lim_{r\to 0}\left( \sup_{} \left\{ \frac{|\phi (b)-\phi (a)|}{|b-a|^{1/s}}: a,b \in U', 0<|b-a|<r \right\}\right) =0.
	\end{equation}
\end{prop}
\begin{proof} 
	We fix an adapted basis $(X_1,\dots,X_n)$ such that $\mathbb L=\exp({\rm span}\{X_1,\dots,X_k\})$ and $\mathbb W=\exp({\rm span}\{X_{k+1},\dots,X_n\})$.  We use the convention in \cref{def:coordinateconletilde} and \cref{rem:IdAndUidInCoordinates}, taking into account the little abuse of notation as in \cref{rem:abuse}. In these coordinates, up to  bi-Lipschitz equivalence, we can suppose to work with the anisotropic norm. If $a\in U$, we denote by $a^1,\dots, a^s$ the vector of components of $a$ in each layer, so $a^j\in \R^{m_j}$ for every $j=1,\dots, s$ and $a=(a^1,\dots, a^s)$. For any $a_0\in U$ and $r>0$ we set
	\[
	\rho_{a_0}( r)\coloneqq\sup \left\{ \frac{ | \phi (b)-  \phi (a)-\nabla^{\phi}\phi_{ a_0}(a^{-1} b)|}{\|\widetilde \phi(a)^{-1} a^{-1} b\, \widetilde \phi(a)  \| } \, :\, a,b \in B(a_0,r)\cap U, a\neq b\right\}. 
	\]
	Assuming $\widetilde \phi\in {\rm UID}(\widetilde U,\mathbb W;\LL)$, we have by \eqref{eqn:UidInCoordinates}
	\begin{equation}\label{limgoestoz}
	\lim_{ r \to 0}\rho_{a_0}( r)=0,
	\end{equation}
	for every $a_0\in U$.	Fix $U'\Subset U$ with $a_0\in U'$. From \cref{p2.5}, since the intrinsically linear function $\de^{\phi}\!\phi_{a_0}$ depends only on the variables on the first layer of $\W$, and it is homogeneous, we can find a constant $C>0$ depending on $a_0$ for which
	\[
	|\nabla^{\phi}\phi_{a_0}(a^{-1} b)| \leq C |b^1-a^1|, \qquad \forall a,b\in \mathbb R^n,
	\]
	 and, consequently, we have
	\begin{equation}\label{equa1}
	\begin{aligned}
	\frac{|\nabla^{\phi}\phi_{a_0} (a^{-1} b)|}{ |b-a|^{1/s } } \leq C r^{1-1/s },
	\end{aligned}
	\end{equation}
	for all $a, b \in \mathbb R^n$ with $0<|a-b|<r<1$.  By a consequence of \cref{lemma333FS}, see \cite[Corollary 3.13]{FS16}, we have with a little abuse of notation 
	\[
	\widetilde \phi(a)^{-1}a^{-1}b\,\widetilde \phi(a)=  b-a +\mathcal P(\widetilde \phi(a) , a^{-1}b),
	\]
	for every $a,b \in \mathbb R^n$, where  $$\mathcal{P}(\widetilde \phi(a) , a^{-1}b)\coloneqq (\mathcal{P} ^1(\widetilde \phi(a) , a^{-1}b), \dots, \mathcal{P} ^s (\widetilde \phi(a) , a^{-1}b)),$$  with $\mathcal{P}^1 (\widetilde \phi(a) , a^{-1}b) =0$. Moreover, for each $i=2,\dots, s$, there is $C_i>0$ depending only $U'$ and $\widetilde\phi$ such that
	\begin{equation*}
	|\mathcal{P}^i(\widetilde \phi(a) , a^{-1}b) | \leq C_i \left( |b^1-a^1|+\dots + |b^{i-1}-a^{i-1}| \right), 
	\end{equation*}
	for all $a, b \in U'$. Hence there exists $C'>0$ depending only on $C_i$ and on the group $\mathbb G$ such that \[
	\frac{ \|\widetilde \phi(a)^{-1}a^{-1}b\, \widetilde \phi(a)\|}{ |b-a|^{1/s }} \leq C',\qquad \forall a,b\in U' \text{ with } 0<|a-b|<1.
	\]  
	Finally, by the last inequality together with \eqref{equa1}, we get
	\begin{equation}\label{eq:finaleconrho}
	\begin{aligned}
	\frac{| \phi (b)-\phi (a) |}{ |b-a |^{1/s } }&\leq  \frac{| \phi (b)- \phi (a)-\nabla^{\phi}\phi_{a_0} (a^{-1} b) |}{\|\widetilde \phi(a)^{-1}a^{-1}b\, \widetilde\phi(a)\| } \,   \frac{ \|\widetilde \phi(a)^{-1}a^{-1}b\, \widetilde \phi(a)\|}{ |b-a|^{1/s } }  + \frac{| \nabla^{\phi}\phi_{a_0} (a^{-1} b)|}{ |b-a|^{1/s } } \\
	& \leq C'\rho_{a_0}(r) + C r^{1-1/s },
	\end{aligned}
	\end{equation}
	for all $a_0\in U$ and all $a, b \in U'\cap B(a_0,r)$ with $0<|a-b|<r<1$. We stress that, ultimately, $C'$ depends only $U'$ and $\widetilde\phi$, while $C$ depends only on $a_0$. 
	
	We conclude the proof by contradiction. Assume we can find $U'\Subset U$, two sequences $(a_h)$ and $(b_h)$ in $U'$, and an infinitesimal sequence $(r_h)$ of positive numbers such that $0<|a_h-b_h|<r_h$ and 
	\[
	\frac{|\phi(b_h)-\phi(a_h)|}{|b_h-a_h|^{1/s}}>M,
	\] 
	for some $M>0$. Since $\overline{U'}$ is compact, we can assume that, up to passing to subsequences, both $(a_h)$ and $(b_h)$ converge to some $a_0\in \overline{U'}$. By \eqref{eq:finaleconrho} we would find some $M'>0$ such that that
	\[
	\rho_{a_0}(r_h)>M',
	\] 
	for arbitrarily large $h\in\mathbb N$, a contradiction to \eqref{limgoestoz}.
\end{proof}

The previous proposition tells us what is the regularity of $\phi$ in \textbf{all} the exponential coordinates, in case it is {\rm UID}. Actually, we can refine \cref{prop3.6cont} by improving the property \eqref{eqn:IdBehaviourOfkCurves}. We stress that the forthcoming proposition would also follow from the implication (1)$\Rightarrow$(2) of \cite[Theorem~4.3.1]{Koz15} but, up to our knowledge, the proof presented here is new. Indeed, in (1)$\Rightarrow$(2) of \cite[Theorem~4.3.1]{Koz15} it is proved that if the intrinsic graph of $\phi$ is a co-horizontal $C^1_{\rm H}$-surface with complemented tangents, then \eqref{eqn:UidImpliesH1k} holds. Then the following \cref{prop:UidImpliesH1k} would be a consequence of that implication and \cref{prop:UidC1H}. Instead, we here give a direct proof within our context. In conclusion we obtain, in a different way, the implication (1)$\Rightarrow$(2) of \cite[Theorem~4.3.1]{Koz15} by making use of \cref{prop:UidImpliesH1k} and \cref{prop:UidC1H}. 
\begin{prop}\label{prop:UidImpliesH1k}
	Let $\mathbb W$ and $\mathbb L$ be two complementary subgroups of a Carnot group $\mathbb G$, with $\mathbb L$ horizontal and $k$-dimensional. Let $\widetilde U\subseteq \mathbb W$ be open and $\widetilde{\phi} \in {\rm UID}(\widetilde U, \mathbb W;\mathbb L)$. Fix an adapted basis $(X_1,\dots,X_n)$ in which  $\mathbb W=\exp({\rm span}\{X_{k+1},\dots,X_n\})$, $\mathbb L=\exp({\rm span}\{X_{1},\dots,X_k\})$ and let $\widetilde V\Subset \widetilde U$. Then
	\begin{equation}\label{eqn:UidImpliesH1k}
	\lim_{\varrho\to0}\!\left(\!\sup\left\{\frac{|\phi(\gamma(t))-\phi(\gamma(s))|}{|t-s|^{1/{\deg j}}}: j=m+1,\dots,n,\,\gamma'=D^{\phi}_{X_j}\circ\gamma, \gamma\subseteq V, 0<|t-s|\leq\varrho \right\}\right)\!=0.
	\end{equation}
\end{prop}
\begin{proof}
	We use the convention in \cref{def:coordinateconletilde} and \cref{rem:IdAndUidInCoordinates}, taking into account the little abuse of notation as in \cref{rem:abuse}. By item (a) of \cref{prop:ContinuityOfDifferentialUid} we have that $\widetilde\phi$ is intrinsically Lipschitz on $\widetilde V\Subset \widetilde U$. We denote by $C$ the constant for which $\widetilde\phi$ is intrinsically $C$-Lipschitz in $\widetilde V$. 
	
	Fix $w_0\in \widetilde V$. Let us take $m+1\leq j\leq n$, and an integral curve $\widetilde{\gamma}\colon I \to\widetilde{V}\subseteq\mathbb W$ of $D^{\phi}_{X_j}$. Without loss of generality we may assume that the curve is defined on $I=[0,T]$, with $T>0$ possibly depending on the curve. By the particular form of $D^{\phi}_j$ in coordinates, see \eqref{eqn:CoordinateDPhiJ}, and the fact that $j\geq m+1$, we have that the projection of $\widetilde{\gamma}$ on $\mathbb V_1$ is constant. Then, since by \cref{p2.5} $\nabla^{\phi}\phi_{w_0}$ depends only on the projection on $\mathbb V_1$ and it is linear, we have, for all $t,s\in [0,T]$,
	\begin{equation}\label{eqn:Finale2}
	\frac{|\phi(\gamma(t))-\phi(\gamma(s))-\nabla^{\phi}\phi_{w_0}(\gamma(s)^{-1}\cdot\gamma(t) )|}{\|\widetilde{\phi}(\widetilde{\gamma}(s))^{-1}\cdot\widetilde{\gamma}(s)^{-1}\cdot\widetilde{\gamma}(t)\cdot\widetilde{\phi}(\widetilde{\gamma}(s))\|}=\frac{|\phi(\gamma(t))-\phi(\gamma(s))|}{\|\widetilde{\phi}(\widetilde{\gamma}(s))^{-1}\cdot\widetilde{\gamma}(s)^{-1}\cdot\widetilde{\gamma}(t)\cdot\widetilde{\phi}(\widetilde{\gamma}(s))\|}.
	\end{equation}
	 Since $\widetilde\phi$ is intrinsically $C$-Lipschitz in $\widetilde V$, by \eqref{eqn:NormOfGammaGeneral1} there exists a constant $C_j>0$ depending only on $j,C$ and the adapted basis such that
	\begin{equation}\label{eqn:EstimateGanza}
	\|\widetilde{\phi}(\widetilde{\gamma}(s))^{-1}\cdot\widetilde{\gamma}(s)^{-1}\cdot\widetilde{\gamma}(t)\cdot\widetilde{\phi}(\widetilde{\gamma}(s))\|\leq C_j|t-s|^{1/\deg j}, \qquad \forall 0\leq s<t\leq T.
	\end{equation} 
	
	In particular, we can find a constant $C'>0$ depending only on $C$ and the adapted basis such that for every $j=m+1,\dots,n$, for every integral curve $\widetilde\gamma\colon[0,T]\to\widetilde V$ of $D^{\phi}_{X_j}$ and every $0\leq s<t\leq T$ we have 
	\begin{equation}\label{eqn:Finale3}
	\frac{|\phi(\gamma(t))-\phi(\gamma(s))|}{\|\widetilde{\phi}(\widetilde{\gamma}(s))^{-1}\cdot\widetilde{\gamma}(s)^{-1}\cdot\widetilde{\gamma}(t)\cdot\phi(\widetilde{\gamma}(s))\|}\geq \frac{|\phi(\gamma(t))-\phi(\gamma(s))|}{C'|t-s|^{1/\deg j}}.
	\end{equation}
	Combining \eqref{eqn:Finale3} and \eqref{eqn:Finale2} we get
	\begin{equation}\label{eqn:EstimateGanza3}
	 \frac{|\phi(\gamma(t))-\phi(\gamma(s))|}{|t-s|^{1/\deg j}}\leq  C'\frac{|\phi(\gamma(t))-\phi(\gamma(s))-\nabla^{\phi}\phi_{w_0}(\gamma(s)^{-1}\cdot\gamma(t) )|}{\|\widetilde{\phi}(\widetilde{\gamma}(s))^{-1}\cdot\widetilde{\gamma}(s)^{-1}\cdot\widetilde{\gamma}(t)\cdot\widetilde{\phi}(\widetilde{\gamma}(s))\|},
	\end{equation}
	 for every $j=m+1,\dots,n$, every integral curve $\widetilde\gamma\colon[0,T]\to\widetilde V$ of $D^{\phi}_{X_j}$, every $0\leq t<s\leq T$ and every $w_0\in\widetilde V$.
	If $|t-s|\leq \varrho$, by the estimate in \cref{lemma333FS} and \eqref{eqn:EstimateGanza}, we get 
	\begin{equation}\label{eqn:EstimateGanza2}
	\|\widetilde\gamma(s)^{-1}\cdot\widetilde\gamma(t)\|\leq C(\varrho),
	\end{equation}
	for every $j=m+1,\dots,n$ and every integral curve $\widetilde\gamma\colon[0,T]\to\widetilde V$ of $D^{\phi}_{X_j}$, where $C(\varrho)$ is a continuous increasing function such that $\lim_{\varrho\to0}C(\varrho)=0$, depending on $\widetilde\phi$ and independent on the choices of $j$ and $\widetilde\gamma$.
	
	Assume by contradiction that \eqref{eqn:UidImpliesH1k} is false. Then there exist $\eps_0>0$, a sequence of integral curves $\widetilde\gamma_\ell\colon[0,T_\ell]\to\widetilde V$  of $D^{\phi}_{X_{i_\ell}}$, for some $i_\ell\in\{m+1,\dots,n\}$, and sequences of times $0\leq t_\ell <s_\ell\leq T_\ell$ such that $|t_\ell-s_\ell|\to0$ as $\ell\to \infty$ and
	\begin{equation}\label{eqn:Contradiction}
	 \frac{|\phi(\gamma_\ell(t_\ell))-\phi(\gamma_\ell(s_\ell))|}{|t_\ell-s_\ell|^{1/\deg(i_{\ell})}} \geq \eps_0.
	\end{equation}
	By compactness, up to passing to subsequences, there exists $\widetilde v\in \widetilde V$ such that $\widetilde\gamma_\ell(t_\ell)\to \widetilde v$. Then, by the uniform control \eqref{eqn:EstimateGanza2} and the fact that $|t_\ell-s_\ell|\to 0$, we get also $\widetilde\gamma_\ell(s_\ell)\to \widetilde v$. Since $\widetilde\phi$ is uniformly intrinsically differentiable at $\widetilde v$, we get, by \cref{rem:IdAndUidInCoordinates} (see \eqref{eqn:UidInCoordinates}) that the right hand side of \eqref{eqn:EstimateGanza3} evaluated at $\widetilde\gamma=\widetilde\gamma_\ell$, $t=t_\ell$, $s=s_\ell$, and $w_0=\widetilde{v}$ goes to 0 as $\ell\to +\infty$. Passing  \eqref{eqn:EstimateGanza3} to the limit (with $j=i_{\ell}$) gives the sought contradiction with  \eqref{eqn:Contradiction}. 
\end{proof}
\begin{rem}\label{rem:OnlyVerticalCoordinates}
	In case $\G$ is a Carnot group of step 2 and rank $m$, and we choose an adapted basis of the Lie algebra such that $\mathbb L=\exp({\rm span}\{X_1,\dots,X_k\})$ and $$\mathbb W=\exp({\rm span}\{X_{k+1},\dots,X_m,X_{m+1},\dots,X_n\}),$$ we have that $D_{X_j}^{\phi}=\partial_{x_j}$ for $m+1\leq j\leq n$, see \cref{example:Step2}. If $\widetilde{\phi}\in {\rm UID}(\widetilde U,\mathbb W;\mathbb L)$ then \cref{prop:UidImpliesH1k}
	tells us that on every compact subset of $U$, the function $\phi$ is uniformly $1/2$-little H\"older continuous in the vertical coordinates. Moreover, \eqref{eqn:IdBehaviourOf1Curves} and the fact the $\de^{\phi}\!\phi$ is continuous, see \cref{prop:ContinuityOfDifferentialUid}, tell us that along the integral curves of $D_{X_j}^{\phi}$, with $k+1\leq j\leq m$, $\phi$ is $C^1$. Then, by exploiting the triangular form of the vector fields $D^{\phi}$ in \eqref{eqn:CoordinateDPhiJ}, one could use the previous informations to prove that $\phi$ is locally 1/2-little H\"older continuous in \textbf{all} the variables, and thus give an alternative proof of \cref{prop3.6cont} in the step-2 case. This is essentially the idea of the proof in \cite[Proposition 4.4]{ASCV06}. 
\end{rem}

\subsection{Definition of broad*}\label{sub:Broad*}
In this subsection we give the notion of broad* solution to the system $D^{\phi}\phi=\omega$, with $\omega$ continuous. Eventually we show that an intrinsically differentiable function $\phi$ with continuous intrinsic gradient $\nabla^{\phi}\phi$ (see \cref{def:NablaPhiPhi}) is a broad* solution to $D^{\phi}\phi=\nabla^{\phi}\phi$. 

Let $\mathbb W$ and $\mathbb L$ be complementary subgroups of a Carnot group $\mathbb G$, with $\mathbb L$ horizontal and $k$-dimensional. Let $\widetilde U$ be an open subset of $\mathbb W$ and let $\widetilde{\phi}:\widetilde{U}\subseteq \mathbb W\to \mathbb L$ be a continuous function. We fix an adapted basis $(X_1,\dots, X_n)$ of the Lie algebra $\mathfrak g$, such that $\mathbb W=\exp(\mbox{span}\{X_{k+1},\dots,X_n\})$ and $\mathbb L=\exp(\mbox{span}\{X_{1},\dots,X_k\})$. We give the notion of broad* solution of the system \begin{equation}\label{sistemone} \left(
\begin{matrix}
D^\phi _{X_{k+1}}\phi ^{(1)}&  \dots & D^\phi _{X_{m}} \phi ^{(1)} \\
\vdots &\ddots & \vdots  \\
D^\phi _{X_{k+1}}\phi ^{(k)} & \dots & D^\phi _{X_{m}} \phi ^{(k)} \\
\end{matrix}
\, \right) = \left(\,
\begin{matrix}
\omega_{1 \, k+1}&  \dots & \omega_{1\, m} \\
\vdots &\ddots & \vdots  \\
\omega_{k\, k+1} & \dots & \omega_{k\, m} \\
\end{matrix}
\, \right),
\end{equation}
where $\omega \coloneqq\left(\omega _{\ell j}  \right)\colon U \to \mathbb R^{k\times (m-k)}$, with  $\ell\in \{1,\dots,k\}$, $j \in \{k+1,\dots,m\}$, is a continuous matrix valued function, and where we refer to the notation introduced in \cref{def:coordinateconletilde}.

\begin{defi}[Broad* and broad solutions]\label{defbroad*}
Let $\mathbb W$ and $\mathbb L$ be complementary subgroups of a Carnot group $\mathbb G$, with $\mathbb L$ horizontal and $k$-dimensional. Let $\widetilde U\subseteq \mathbb W$ be open and let $\widetilde\phi\colon\widetilde U\to\mathbb L$ be a continuous function. Consider an adapted basis $(X_1,\dots, X_n)$ of the Lie algebra $\mathfrak g$ such that $\mathbb L=\exp({\rm span}\{X_{1},\dots,X_k\})$ and $\mathbb W=\exp({\rm span}\{X_{k+1},\dots,X_n\})$. Let $\omega\coloneqq\left(\omega_{\ell j}\right)\colon U \to \R^{k\times (m-k)}$  be a continuous matrix valued function with  $\ell\in \{1,\dots,k\}$ and $j \in \{k+1,\dots,m\}$. We say that $\phi \eqqcolon(\phi^{(1)},\dots , \phi^{(k)})\in C(U; \R^k)$ is a \emph{broad* solution} of $D^\phi \phi=\omega$ in $U$ if for every $a_0\in U$ there exist $0< \delta_2 < \delta_1$ and $m-k$ maps $E_j^\phi\colon \overline{B(a_0,\delta_2)} \times[-\delta _2,\delta _2]\to \overline{B(a_0,\delta_1)}$ for $j=k+1,\dots, m$, \textbf{where the balls are considered restricted to $U$}, satisfying the following two properties.
\begin{itemize}
	 \item[(a)] For every $a\in \overline{B(a_0, \delta_2)}$ and every $j=k+1,\dots, m$, the map $E_j^\phi(a)\coloneqq E_j^\phi(a,\cdot)$ is $C^1$ regular and it is a solution of the Cauchy problem
		\[
		\begin{cases}
		\dot \gamma= D^\phi_j\circ \gamma&\\
		\gamma(0)=a,&
		\end{cases}
		\]
		in the interval $[-\delta_2,\delta_2]$, where the vector field $D^\phi_j\coloneqq D^{\phi}_{X_j}$ is defined in \eqref{eqn:DefinitionOfDj}.
	\item[(b)] For every $a\in \overline{B(a_0,\delta_2)}$, for every $t\in [-\delta_2,\delta_2]$, every $j=k+1,\dots,m$ and every $\ell=1,\dots,k$ one has
		\[
		\phi^{(\ell)}(E_j^\phi(a,t))-\phi^{(\ell)}(a)=\int_0^t \omega_{\ell j} (E_j^\phi(a,s))\,\de\!s.
		\]
	\end{itemize}

	We say that $D^{\phi}\phi=\omega$ in the {\em broad sense on $U$} if for every $W\in {\rm Lie}(\mathbb W)\cap V_1$ and every $\gamma:I\to U$ integral curve of $D^{\phi}_W$, it holds that 
	$$
	\frac{\de}{\de s}_{|_{s=t}} (\phi\circ\gamma)(s)=\omega(W)(\gamma(t)), \qquad \forall t\in I,
	$$
	where by $\omega(W)$ we mean the matrix $\omega$ applied to the $(m-k)$-vector $W$. 
\end{defi}
\begin{rem}
	We stress that, in the setting of \cref{defbroad*}, if $\widetilde\phi \in C^1(\widetilde U)$ then $D^{\phi}\phi=\nabla^{\phi}\phi$ both pointwise and in the broad sense on $U$. First $\widetilde\phi\in {\rm UID}(\widetilde U,\mathbb W;\mathbb L)$ by \cite[Theorem 4.9]{DiDonato18}, because $\widetilde\phi\in C^1(\widetilde U)$. Then we can consider the intrinsic gradient $\nabla^{\phi}\phi$ as in \cref{def:NablaPhiPhi}, which is continuous, see \cref{prop:ContinuityOfDifferentialUid}. Thus the claim is an outcome of point (i) of \cref{prop:DerivativeOfIntegralCurves} and point (c) of \cref{prop2.22}, that becomes a pointwise equality if $\widetilde\phi$ is $C^1(\widetilde U)$, see the proof of \cref{prop2.22}. 
\end{rem}
\begin{rem}\label{rem:Broad*InCoordinates}
	Let us notice that the definition given in \cref{defbroad*} is a priori susceptible to the choice of an adapted basis. Nevertheless, when it is coupled with the vertically broad* h\"older condition in the same basis, see \cref{def:vertically*Holder}, it is independent of this choice. This is an outcome of \cref{thm:MainTheorem}. Indeed, from (d)$\Rightarrow$(a) of \cref{thm:MainTheorem}, it follows that the broad* condition and the vertically broad* h\"older condition on a fixed basis imply that $\widetilde\phi$ is UID. Thus, from item (i) of \cref{prop:DerivativeOfIntegralCurves}, we get that the broad* condition is satisfied for every other basis. Finally, from \cref{prop:UidImpliesH1k}, we get that also the vertically broad* h\"older condition holds in every other basis. 
	
	With the above reasoning, we remark that we can conclude something stronger: if the broad* condition and the vertically broad* h\"older condition hold on a fixed basis, then they hold uniformly on the choice of $W\in {\rm Lie}(\mathbb W)$ with bounded norm, see \cref{defbroad*Intro}, and \cref{def:vertically*HolderIntro}.
\end{rem}

The following result is already known in the Heisenberg groups $\mathbb H^n$: in case $\mathbb L$ is one-dimensional, it is  proved in \cite[(3)$\Rightarrow$(2) \& Theorem 4.95]{SC16}, while in case $\mathbb L$ is $k$-dimensional it is proved in \cite[Theorem~1.4, (iii)$\Rightarrow$(ii)]{Cor19}. We here generalize it to arbitrary Carnot groups, in the case $\mathbb L$ is horizontal and $k$-dimensional.
\begin{prop}\label{prop:IdGradContImpliesBroad}
	Let $\mathbb W$ and $\mathbb L$ be complementary subgroups of a Carnot group $\mathbb G$, with $\mathbb L$ horizontal and $k$-dimensional, and consider an adapted basis of the Lie algebra $\mathfrak g$ such that $\mathbb W=\exp(\mathrm{span}\{X_{k+1},\dots,X_n\})$ and $\mathbb L=\exp(\mathrm{span}\{X_{1},\dots,X_k\})$. Let $\widetilde U$ be an open subset of $\mathbb W$, and $\widetilde\phi\in {\rm ID}(\widetilde U,\mathbb W;\mathbb L)$ be such that $\de^{\phi}\!\phi$ is continuous on $\widetilde U$. Denote by $\nabla^{\phi}\phi$ the $k\times (m-k)$ matrix that represents $\de^{\phi}\!\phi$ in coordinates, see \cref{def:NablaPhiPhi}.
	
	Then, we have that
	\begin{equation}\label{eqn:Semplice}
	\frac{\de}{\de t}_{|_{t=t_0}}(\phi^{(\ell)}\circ\gamma)(t)=\nabla^{\phi}_{\ell j}\phi(\gamma(t_0)),
	\end{equation}
	for every $j=k+1,\dots,m$, every integral curve $\widetilde\gamma\colon I\to\widetilde U$  of $D^{\phi}_j\coloneqq D^{\phi}_{X_j}$, every $\ell=1,\dots,k$, and every $t_0\in I$. In particular the function $\phi$ is a broad solution, and thus also a broad* solution, of the system $D^{\phi}\phi=\nabla^{\phi}\phi$.
\end{prop}
\begin{proof}
	Equation \eqref{eqn:Semplice} directly follows from \eqref{eqn:IdBehaviourOf1Curves} seen in coordinates. Then, from \eqref{eqn:Semplice} and the fact that $\nabla^{\phi}\phi$ is continuous by hypothesis, the second part of the thesis follows.
\end{proof}

\section{Main Theorems in arbitrary Carnot groups}\label{sec:Theorems}
In this section we prove \cref{thm:MainTheorem}, that is \cref{thm:MainTheoremIntro2} in the introduction. We deal with an arbitrary Carnot group $\mathbb G$ along with a continuous function $\widetilde{\phi}\colon\widetilde U\subseteq\mathbb W\to\mathbb L$, where $\mathbb W$ and $\mathbb L$ are complementary subgroups of $\mathbb G$, with $\mathbb L$ horizontal and $k$-dimensional, and $\widetilde U$ is an open subset of $\mathbb W$. 

In \cref{sub:FromCurvesToFunction}, we study how H\"older properties of $\widetilde\phi$ along integral curves of the vector fields $D^{\phi}$ as defined in \eqref{eqn:DefinitionOfDj} affect the intrinsic regularity of the function $\widetilde \phi$. The main result of this section is a converse of \cref{prop:UidImpliesH1k}: if $D^{\phi}\phi=\omega$ holds in the broad* sense (see \cref{defbroad*}) and there is, locally around every point, a family of curves satisfying the little H\"older regularity condition \eqref{eqn:UidImpliesH1k} (we shall call this property vertically broad* h\"older regularity, see \cref{def:vertically*Holder}), then $\widetilde \phi$ is uniformly intrinsically differentiable. For the full statement, see \cref{prop:DPhiPhi=wBroadImpliesUid}. We notice that, taking \cref{rem:OnlyVerticalCoordinates} into account, the latter proposition generalizes \cite[Theorem~5.7]{ASCV06}, which deals with the case  $\G=\mathbb H^n$ and $\mathbb L$ one-dimensional, \cite[Theorem~5.8, (4)$\Rightarrow$(2)]{DiDonato18}, which is proved in case $\G$ has step 2 and $\mathbb L$ is one-dimensional, and \cite[Theorem~5.3]{Cor19} that solves the problem for $\G=\mathbb H^n$ with $\mathbb L$ horizontal and $k$-dimensional. We remark that, also in these cases, we obtain slightly stronger results, requiring just a locally 1/2-little H\"older regularity in the vertical components.

 \cref{prop:DPhiPhi=wBroadImpliesUid} could also be obtained by a combination of (2)$\Rightarrow$(1) of \cite[Theorem~4.3.1]{Koz15} and \cref{prop:UidC1H}, which is proved in \cite{DiDonato18}. The idea of (2)$\Rightarrow$(1) in \cite[Theorem~4.3.1]{Koz15} is to show that the H\"older conditions on $\widetilde\phi$ along a family of integral curves, that is the assumptions of \cref{prop:DPhiPhi=wBroadImpliesUid}, imply that the intrinsic graph of $\widetilde \phi$ is a co-horizontal $C^1_{\rm H}$-surface with complemented tangents. To prove this latter fact the author uses a characterization of co-horizontal $C^1_{\rm H}$-surfaces by means of uniform Hausdorff convergence to tangents - see \cite[Theorem~3.1.12]{Koz15} - that is in turn based on the so-called four cones Theorem, see \cite[Theorem~1.2]{BK14}. With the independent proof we give in \cref{prop:DPhiPhi=wBroadImpliesUid}, with more analytic flavor, we stress we can indirectly obtain (2)$\Rightarrow$(1) of \cite[Theorem~4.3.1]{Koz15} by making use of \cref{prop:DPhiPhi=wBroadImpliesUid} and \cref{prop:UidC1H}. We also obtained (1)$\Rightarrow$(2) of \cite[Theorem~4.3.1]{Koz15}, see the discussion before \cref{prop:UidImpliesH1k}.

Our proof of \cref{prop:DPhiPhi=wBroadImpliesUid} requires \cref{prop:InverseLpischitz}, that is stated  only for $\mathbb L$ horizontal and $k$-dimensional. The \cref{prop:InverseLpischitz} is a converse of \cref{lem:EstimateOnTheNormOfGamma}, i.e., it can be roughly read in the following way: the uniform H\"older regularity of the curves $\widetilde\phi\circ \widetilde\gamma$, where $\widetilde\gamma$ is an integral curve of the vector field $D^\phi$, implies the intrinsically Lipschitz regularity of $\widetilde\phi$. We give a proof of \cref{prop:InverseLpischitz} as we crucially need it for the proof of \cref{prop:DPhiPhi=wBroadImpliesUid}, but we remark that a more general statement can be given, in case $\mathbb W$ is normal, with a proof that is very similar to the one of \cref{prop:InverseLpischitz}. Notice that the general statement with $\mathbb W$ normal can be found in (3)$\Rightarrow$(1) of \cite[Theorem~4.2.16]{Koz15}. For more details, see \cref{rem:LipschitzOnCurvesImpliesIntrinsicLipschitz}.

As a by-product, we obtain some analytical results, that can have their own independent interests. The first one is given by \cref{coroll:StarImpliesWithoutStar} and it states that broad* regularity implies broad regularity. Roughly speaking, having a function that is H\"older regular on a \textbf{precise family} of integral curves implies the H\"older regularity on \textbf{every} integral curve. Then, in \cref{coroll:IDContinuousImpliesUid}, we prove that every intrinsically differentiable function that is vertically broad* h\"older (see \cref{def:vertically*Holder}), and that has a continuous intrinsic gradient, is uniformly intrinsically differentiable. We do not know at present whether the assumption on the vertically broad* h\"older regularity can be dropped in \cref{coroll:IDContinuousImpliesUid}, see also \cref{rem:StrategyIdGradContinuousUid}. We expect that the hypothesis on the vertically broad* h\"older regularity in \cref{coroll:IDContinuousImpliesUid} can be dropped in general, see also the paragraph \textbf{Geometric characterizations of intrinsic differentiability} in the introduction. From the results proved in \cite{BSC10b} and \cite{Cor19},  we know that the assumption on the vertically broad* h\"older regularity in \cref{coroll:IDContinuousImpliesUid} is not necessary in the case of the Heisenberg groups $\mathbb H^n$, with $\mathbb L$ horizontal $k$-dimensional, see also the introduction to \cref{sec:Applications}. We stress we obtain that we can remove the assumption on the vertically broad* h\"older regularity in \cref{coroll:IDContinuousImpliesUid} also in the case of step-2 Carnot groups with $\mathbb L$ one-dimensional, see \cref{sec:Step2}. 

In \cref{sub:Area} we focus on the case in which $\mathbb L$ is one-dimensional and we prove \cref{prop2.22}. We present an area formula that represents the perimeter of the subgraph of a uniformly intrinsically differentiable function $\phi$ in terms of the density $\sqrt{1+|\nabla^{\phi}\phi|^2}$. For more details about the area formula and a representation involving the Hausdorff measures, we refer the reader to \cref{rem:AreaFormula}. In \cref{prop2.22} we also prove that, whenever the target $\mathbb L$ is one-dimensional, every uniformly intrinsically differentiable function $\phi$ is \textbf{a distributional solution} of the system $D^{\phi}\phi=\nabla^{\phi}\phi$ and, in \cref{coroll:Broad*ImpliesDistributionally}, we deduce that if $D^{\phi}\phi=\omega$ holds in the broad* sense with a continuous datum $\omega$ and $\phi$ is vertically broad* h\"older, then $D^{\phi}\phi=\omega$ in the sense of distributions. We do not know, in general, if in \cref{coroll:Broad*ImpliesDistributionally} we can remove the assumption on the vertically broad* h\"older regularity. In fact, one can remove the hypothesis on the vertically broad* h\"older regularity in \cref{coroll:Broad*ImpliesDistributionally} in Heisenberg groups, and it is a consequence of the results in \cite{BSC10b}. We stress that thanks to the results obtained in \cref{sec:Step2}, we drop the assumption on the vertically broad* h\"older regularity in \cref{coroll:Broad*ImpliesDistributionally} also in the case of step-2 Carnot groups. 

It is interesting to investigate the converse implication: if one has $D^{\phi}\phi=\omega$ in the sense of distributions with a continuous function $\omega$, is it true that $D^{\phi}\phi=\omega$ in the broad* sense? This is actually the case in the Heisenberg groups, see \cite{BSC10a}, and the techniques used in \cref{sec:Step2} seem a good tool to address this implication in arbitrary step-2 Carnot groups. We will not address this issue in this paper and it will be the target of further investigations. It is however interesting to notice that, in some examples besides the step-2 case and for particular $\phi$, one can obtain that if $D^{\phi}\phi=\omega$ holds in the sense of distributions with a continuous function $\omega$, then $D^{\phi}\phi=\omega$ in the broad* sense. We will not discuss this issue in the paper, but we refer the reader to \cite{ABC16}.

In \cref{sub:Relations} we come back to the general case in which the target $\mathbb L$ is horizontal and not necessarily one-dimensional. We prove that if $\phi$ is locally approximable with a sequence of smooth functions whose intrinsic derivatives converge to a continuous function $\omega$, then $D^{\phi}\phi=\omega$ in the broad*, see \cref{prop4.12}. This notion of local approximability has been first introduced and studied in \cite{ASCV06}, see also \cref{rem:PointwiseAndLocalApproximability}. We exploit this result to prove that every uniformly intrinsically differentiable function $\phi$ always solves $D^{\phi}\phi=\nabla^{\phi}\phi$ in the broad* sense. 

In \cref{sub:Main}, we combine some of the previous results together to prove our main theorem \cref{thm:MainTheorem}, which is \cref{thm:MainTheoremIntro2} in the introduction. Notice that our result provides the generalization to all Carnot groups, and to any possible horizontal and $k$-dimensional target $\mathbb L$, of \cite[Theorem~5.8]{DiDonato18}. We stress that \cref{thm:MainTheorem} will be strengthened in \cref{sec:Step2} dropping the hypothesis on the vertical broad* h\"older regularity in the setting of Carnot groups of step 2. 
We stress that, in general, the assumption on the vertical broad* h\"older regularity cannot be dropped in \cref{thm:MainTheorem}, see \cref{rem:VerticallyBroadHolderNonSiToglie} for a counterexample in the easiest step-3 group, namely the Engel group. 
 
\subsection{From regularity of $\phi$ along integral curves of $D^{\phi}$ to regularity of $\phi$}\label{sub:FromCurvesToFunction}
In this subsection we show how the H\"older regularity of $\phi$ along integral curves of $D^{\phi}$ affects the intrinsic regularity of $\widetilde\phi$. 
 
\begin{prop}\label{prop:InverseLpischitz}
	Let $\mathbb W$ and $\mathbb L$ be complementary subgroups of a Carnot group $\G$, with $\mathbb L$ horizontal and $k$-dimensional, let $\widetilde U$ be open, and let $\widetilde{\phi}\colon\widetilde{U}\subseteq \mathbb W\to\mathbb L$ be a continuous function with $e\in\widetilde{U}$ and $\widetilde{\phi}(e)=e$.
	
	Let $(X_1,\dots,X_n)$ be an adapted basis of the Lie algebra $\mathfrak g$ such that $\mathbb L=\exp({\rm span}\{X_{1},\dots,X_k\})$ and $\mathbb W=\exp({\rm span}\{X_{k+1},\dots,X_n\})$. Denote by
	$D^{\phi}_j\coloneqq D^{\phi}_{X_j}$, for every $j=k+1,\dots,n$. Let $L>0$.
	
	Fix $v\in U$ and consider a concatenation of curves $\gamma_{k+1},\dots,\gamma_n$ in $U$ connecting $0$ to $v$ such that $\gamma_j\colon I_j\to U$ is an integral curve of $D^\phi_j$ for $j=k+1,\dots, n$. Assume that the function $\phi\circ \gamma_j$ is $\frac{1}{\mathrm{deg}j}$-H\"older continuous on $I_j$, for $j=k+1,\dots,n$, with H\"older constant $L$.
	Then, there exists $C>0$ only depending on $L, \mathbb W$, $\mathbb L$ and the adapted basis $(X_1,\dots, X_n)$ such that
	\begin{equation}\label{eqn:ThesisOfConverseLipschitz1}
	|\phi\circ \gamma_j(t)|\leq C\|v\|_{\mathbb G}, \qquad \forall t\in I_j, \quad  \forall j=k+1,\dots, n,
	\end{equation}
	\begin{equation}\label{eqn:ThesisOfConverseLipschitz2}
	\;\;\;|\gamma_j^{(i)}(t)| \leq C\|v\|_{\mathbb G}^{\deg i}, \qquad\; \forall t\in I_j, \quad \forall i,j=k+1,\dots,n,
	\end{equation}
	where $\gamma^{(i)}_j$ denotes the $i$-th component of $\gamma_j$ in exponential coordinates.
\end{prop}

\begin{proof}
	Up to bi-Lipschitz equivalence, we can prove the result choosing the anisotropic norm $\|x\|_{\mathbb G}\coloneqq \sum_{\ell=1}^n |x_{\ell}|^{1/\deg \ell}$. For the sake of readability, we fix $k=1$ and assume that all the layers $V_i$ of the algebra $\mathfrak g$, with $i\geq 2$, are 1-dimensional, so that $\deg i=i-1$ for every $i=2,\dots, n$. The proof in the general case only requires a typographical effort to deal with multiple components in each layer and the fact that we have more zeros in the components of the first layer.
	
	We work in exponential coordinates so that $v=( 0,v_{2},\dots,v_n)$, and denote the extremal points of the concatenation of $\gamma_2,\dots, \gamma_n$ by the following chain
	\begin{equation}\label{eqn:Concatenation}
	\begin{aligned}
	e=(0,0,\dots,0)&\rightarrow_{\gamma_{2}}( 0,v_2,r_{3,2},\dots,r_{n,2})\rightarrow_{\gamma_{3}}\dots  \\\dots 
	&\rightarrow_{\gamma_j}( 0,v_{2},\dots,v_j,r_{j+1,j},\dots,r_{n,j})\rightarrow_{\gamma_{j+1}}\dots \\
	\dots &\rightarrow_{\gamma_{n}} v=( 0,v_{2},\dots,v_n).
	\end{aligned}
	\end{equation}
	In the previous chain, we used the triangular form of $D^\phi_j$ given in \eqref{eqn:CoordinateDPhiJ}, so that the flow along $D^\phi_j$ does not affect the coordinates with index less than $j$. Notice also that, without loss of generality we can assume $I_j\subseteq[-|v_j-r_{j,j-1}|,|v_j-r_{j,j-1}|]$, with the convention $r_{2,1}\coloneqq 0$. Indeed, again by using \eqref{eqn:CoordinateDPhiJ}, in order to correct the error $r_{j,j-1}$ in the $j$-th component, we have to flow for a time $v_j-r_{j,j-1}$ along $D^{\phi}_j$.

	We prove \eqref{eqn:ThesisOfConverseLipschitz1} and \eqref{eqn:ThesisOfConverseLipschitz2} by induction on $j$. When we say that a constant depends on $\mathbb W$ and  $\mathbb L$ we mean that also depends on the chosen adapted basis $(X_1,\dots, X_n)$.
	
	 \noindent By assumption, the curve $\phi\circ\gamma_2$ is $L$-Lipschitz on $I_2$  and $\phi\circ\gamma_2(0)=0$. Consequently, we have that 
	\begin{equation}\label{eqn:FirstInduction}
	|\phi\circ\gamma_2(t)|=|\phi\circ\gamma_2(t)-\phi\circ\gamma_2(0)|\leq L|t|\leq L|v_2| \leq L\|v\|_{\mathbb G}, \qquad \forall t\in I_2.
	\end{equation}
	Therefore \eqref{eqn:ThesisOfConverseLipschitz1} is proved for $j=2$. 
	
	Next we shall prove \eqref{eqn:ThesisOfConverseLipschitz2} for $j=2$ by means of induction on $i$. For $i=2$, we have $\gamma_2^{(2)}(t)=t$ and then  
	\[
	|\gamma_2^{(2)}(t)|=|t|\leq |v_2|\leq \|v\|_{\mathbb G}, \qquad \forall t\in I_2.
	\]
	Assume now that for some $i_0\geq 2$, there is a constant $C_{2,i_0}>0$ such that
	\begin{equation}\label{eqn:ipotesiinduttiva1}
	|\gamma_2^{(i)}(t)|\leq C_{2,i_0}\|v\|_{\mathbb G}^{\deg i}= C_{2,i_0}\|v\|_{\mathbb G}^{i-1},\quad \forall t\in I_2,\quad \forall i=2,\dots,i_0.
	\end{equation}
	 We want to prove that there exists $C_{2,i_0+1}>0$ such that 
	 \[
	  |\gamma_2^{(i_0+1)}(t)|\leq C_{2,i_0+1}\|v\|_{\mathbb G}^{\deg (i_0+1)} =C_{2,i_0+1}\|v\|_{\mathbb G}^{i_0}, \quad \forall t\in I_2,
	 \]
	 where $C_{2,i_0+1}$ only depends on  $ i_0, L,C_{2,i_0}, \mathbb W$ and $\mathbb L$. By using the particular triangular form of $D_2^{\phi}$ in exponential coordinates, see \eqref{eqn:CoordinateDPhiJ}, we get that
	\begin{equation}\label{eqn:SecondInduction}
	|\gamma_2^{(i_0+1)}(t)|=\left|\int_0^t P_{i_0+1}^2(\phi\circ\gamma_2(s),\gamma_2^{(2)}(s),\dots,\gamma_2^{(i_0)}(s))\de\!s\right|,\quad \forall t\in I_2,
	\end{equation}
	for some polynomial $P^2_{i_0+1}$ of homogeneous degree $\deg(i_0+1)-\deg 2=i_0-1$. Then, from \eqref{eqn:FirstInduction} and \eqref{eqn:ipotesiinduttiva1}, we deduce there exists $ C_{2,i_0+1}$ depending on $i_0$, $L$, $C_{2,i_0}$ ,$\mathbb W$ and $\mathbb L$ such that
	\[
	\left|P_{i_0+1}^2(\phi\circ\gamma_2(t),\gamma_2^{(2)}(t),\dots,\gamma_2^{(i_0)}(t))\right| \leq C_{2,i_0+1}\|v\|_{\mathbb G}^{i_0-1}, \qquad \forall t\in I_2,
	\]
	and thus, by using this inequality in \eqref{eqn:SecondInduction}, and since $|t|\leq |v_2|\leq \|v\|_{\mathbb G}$, we get 
	\[
	|\gamma_2^{(i_0+1)}(t)|\leq C_{2,i_0+1}|t|\|v\|_{\mathbb G}^{i_0-1}\leq  C_{2,i_0+1}\|v\|_{\mathbb G}^{i_0}, \qquad \forall t\in I_2.
	\]
	To conclude the proof  of \eqref{eqn:ThesisOfConverseLipschitz1} and \eqref{eqn:ThesisOfConverseLipschitz2} in the case $j=2$, it is enough to set $C_2\coloneqq \max_{i=2,\dots,n} C_{2,i}$, where $C_{2,2}=\max\{L,1\}$.
	
	Assume now that for some $2\leq j_0\leq n$  and some constant $C_{j_0}>0$ one has
	\begin{equation}\label{eqn:ipotesiinduttivaj1}
	|\phi\circ \gamma_{j}(t)|\leq C_{j_0}\|v\|_{\mathbb G}, \qquad \forall t\in I_{j},\quad \forall j=2,\dots,j_0, \\
	\end{equation}
	\begin{equation}\label{eqn:ipotesiinduttivaj2}
	|\gamma_{j}^{(i)}(t)| \leq C_{j_0}\|v\|_{\mathbb G}^{\deg i}, \qquad \forall t\in I_{j}, \quad \forall i\geq 2,\quad \forall j=2,\dots, j_0.
	\end{equation}
	We want to find $C_{j_0+1}>0$ only depending on $L, j_0, C_{j_0}, \mathbb W$ and $\mathbb L$ such that
	\begin{equation}\label{eqn:ThirdInduction1}
	|\phi\circ \gamma_{j_0+1}(t)|\leq C_{j_0+1}\|v\|_{\mathbb G}, \qquad \forall t\in I_{j_0+1}, 
	\end{equation}
	\begin{equation}\label{eqn:ThirdInduction2}
	|\gamma_{j_0+1}^{(i)}(t)| \leq C_{j_0+1}\|v\|_{\mathbb G}^{\deg i}, \qquad \forall t\in I_{j_0+1}, \quad \forall i= 2,\dots, n.
	\end{equation}
	
	To prove \eqref{eqn:ThirdInduction1}, we develop the following inequalities:
	\begin{equation}\label{eqn:CentralEstimate}
	\begin{aligned}
	|\phi\circ \gamma_{j_0+1}(t)|&\leq |\phi\circ \gamma_{j_0+1}(t)-\phi\circ \gamma_{j_0+1}(0)|+|\phi\circ \gamma_{j_0+1}(0)|\leq L|t|^{1/\deg(j_0+1)}+|\phi\circ \gamma_{j_0+1}(0)| \\
	&\leq L|v_{j_0+1}-r_{j_0+1,j_0}|^{1/j_0}+C_{j_0}\|v\|_{\mathbb G}\leq L|v_{j_0+1}|^{1/j_0}+L|r_{j_0+1,j_0}|^{1/j_0}+C_{j_0}\|v\|_{\mathbb G} \\
	&\leq L\|v\|_\G+LC_{j_0}^{\frac 1{j_0}}\|v\|_\G+C_{j_0}\|v\|_\G=\widetilde C_{j_0+1}\|v\|_\G, \qquad \forall t\in I_{j_0+1},
	\end{aligned}
	\end{equation}
	where in the second inequality we used the fact that $\phi\circ\gamma_{j_0+1}$ is H\"older of constant $L$; in the third inequality we used the definition of $I_{j_0+1}$ for the first term, while the estimate on the second term  comes from \eqref{eqn:ipotesiinduttivaj1} and the fact that $\gamma_{j_0+1}(0)$ is the endpoint of $\gamma_{j_0}$; the fifth inequality is a consequence of \eqref{eqn:ipotesiinduttivaj2}, inequality $|v_{j_0+1}|^{1/j_0}\leq \|v\|_{\mathbb G}$, and the fact that $r_{j_0+1,j_0}$ is the endpoint of $\gamma_{j_0}^{(j_0+1)}$.
	
	We are left to prove \eqref{eqn:ThirdInduction2}. First we notice that, by \eqref{eqn:CoordinateDPhiJ}, for $2\leq i\leq j_0$ we have $\gamma_{j_0+1}^{(i)}(t)\equiv v_i\leq \|v\|^{\deg i}_\G$ for all $t\in I_{j_0+1}$. Using again \eqref{eqn:CoordinateDPhiJ}, for $i=j_0+1$ we have that $\gamma_{j_0+1}^{(j_0+1)}(t)=r_{j_0+1,j_0}+t$. Then, since  $|t|\leq |v_{j_0+1}-r_{j_0+1,j_0}|$, for every $t\in I_{j_0+1}$, arguing similarly to \eqref{eqn:CentralEstimate}, one can obtain
	\[
	|\gamma_{j_0+1}^{(j_0+1)}(t)|\leq 2|r_{j_0+1,j_0}|+|v_{j_0+1}|\leq (2C_{j_0}+1)\|v\|_\G^{j_0}.
	\]
	Setting $C_{j_0+1,j_0+1}=\max\{2C_{j_0}+1,1\}$, we get  
	\[
	|\gamma_{j_0+1}^{(i)}(t)| \leq C_{j_0+1,j_0+1}\|v\|_{\mathbb G}^{\deg i},\qquad \forall t\in I_{j_0+1},\quad \forall i=2,\dots,j_0+1.
	\]
	
	Let us now proceed by induction on $i$ and assume there exists $i_0\in\{j_0+1,\dots,n\}$ and a constant $C_{j_0+1,i_0}>0$ such that
	\begin{equation}\label{eqn:induzionenellinduzione}
	|\gamma_{j_0+1}^{(i)}(t)|\leq  C_{j_0+1,i_0}\|v\|_{\mathbb G}^{\deg i}=C_{j_0+1,i_0}\|v\|_{\mathbb G}^{i-1}, \quad \forall t\in I_{j_0+1}, \quad \forall i=2,\dots,i_0.
	\end{equation}
	 We want to find $C_{j_0+1,i_0+1}>0$ only depending on $L, C_{j_0}, C_{j_0+1,i_0}, \mathbb W$ and $\mathbb L$ such that
	 \[
	 |\gamma_{j_0+1}^{(i_0+1)}(t)|\leq  C_{j_0+1,i_0+1}\|v\|_{\mathbb G}^{\deg (i_0+1)} = C_{j_0+1,i_0+1}\|v\|_{\mathbb G}^{i_0},\quad \forall t\in I_{j_0+1}.
	 \]
	 By using again \eqref{eqn:CoordinateDPhiJ}, we get that, for every $t\in I_{j_0+1}$, one has
	\begin{equation}\label{eqn:SecondInduction2}
	\begin{aligned}
	|\gamma_{j_0+1}^{(i_0+1)}(t)|&=\left|r_{i_0+1,j_0}+\int_0^t P_{i_0+1}^{j_0+1}(\phi\circ\gamma_{j_0+1}(s),\gamma_{j_0+1}^{(2)}(s),\dots,\gamma_{j_0+1}^{(i_0)}(s))\de\!s\right| \\
	&\leq |r_{i_0+1,j_0}|+\int_0^t\left|P_{i_0+1}^{j_0+1}(\phi\circ\gamma_{j_0+1}(s),\gamma_{j_0+1}^{(2)}(s),\dots,\gamma_{j_0+1}^{(i_0)}(s))\right|\de\!s,
	\end{aligned}
	\end{equation}
	for a polynomial $P^{j_0+1}_{i_0+1}$ of homogeneous degree $\deg(i_0+1)-\deg(j_0+1)=i_0-j_0$. From \eqref{eqn:CentralEstimate} and \eqref{eqn:induzionenellinduzione}, we deduce that there exists a constant $M>0$ depending only on $j_0, i_0, L, \mathbb W$ and $\mathbb L$ (it indeed depends on the coefficients of the polynomial, the constant $\widetilde C_{j_0+1}$, and the induction constant $C_{j_0+1,i_0}$) such that
	\[
	\left|P_{i_0+1}^{j_0+1}(\phi\circ\gamma_{j_0+1}(t),\gamma_{j_0+1}^{(2)}(t),\dots,\gamma_{j_0+1}^{(i_0)}(t))\right| \leq M\|v\|_{\mathbb G}^{i_0-j_0}, \qquad \forall t\in I_{j_0+1},
	\]
	and thus from \eqref{eqn:SecondInduction2} we get
	\begin{equation}\label{eqn:FourthInduction}
	|\gamma_{j_0+1}^{(i_0+1)}(t)|\leq |r_{i_0+1,j_0}|+M|t|\|v\|_{\mathbb G}^{i_0-j_0}\leq |r_{i_0+1,j_0}|+M|v_{j_0+1}-r_{j_0+1,j_0}|\|v\|_{\mathbb G}^{i_0-j_0},
	\end{equation}
	for every $t\in I_{j_0+1}$. Notice that $r_{i_0+1,j_0}$ is the endpoint of $\gamma_{j_0}^{(i_0+1)}$, as well as $r_{j_0+1,j_0}$ is the endpoint of $\gamma_{j_0}^{(j_0+1)}$. Thus, by \eqref{eqn:ipotesiinduttivaj2}, we get that 
	\[
	|r_{i_0+1,j_0}|\leq C_{j_0}\|v\|^{\deg(i_0+1)}_{\mathbb G}=C_{j_0}\|v\|^{i_0}_{\mathbb G},
	\]
	and 
	\[
	|v_{j_0+1}-r_{j_0+1,j_0}|\leq (1+C_{j_0})\|v\|_{\mathbb G}^{\deg(j_0+1)}=(1+C_{j_0})\|v\|_{\mathbb G}^{j_0}.
	\]
	Replacing these last two equalities in \eqref{eqn:FourthInduction}, we get 
	\[
	|\gamma_{j_0+1}^{(i_0+1)}(t)|\leq  C_{j_0+1,i_0+1}\|v\|^{i_0}_{\mathbb G}, \qquad \forall t\in I_{j_0+1},
	\]
	for the constant $ C_{j_0+1,i_0+1}\coloneqq C_{j_0}+M(1+C_{j_0})$, which only depends on $j_0, i_0, L, \mathbb W$, and $\mathbb L$. Inequalities \eqref{eqn:ThirdInduction1} and \eqref{eqn:ThirdInduction2} are completed by choosing $C_{j_0+1}\coloneqq\max\{ \max_{i=j_0+1,\dots,n} C_{j_0+1,i}, \widetilde C_{j_0+1} \}$. To conclude the proof, it is enough to set $C\coloneqq \max _{j=2,\dots, n} C_j$.
\end{proof}
\begin{rem}[An improvement of \cref{prop:InverseLpischitz}]\label{rem:LipschitzOnCurvesImpliesIntrinsicLipschitz}
	The careful reader could have noticed that the scheme of the proof  \cref{prop:InverseLpischitz} above can be adapted to prove the very same statement, in the more general case in which $\mathbb W$ is normal. We will not use the conclusions of this remark in what follows, but we nonetheless give a sketch of the proof of this fact.
	
	Indeed, also in the case in which $\mathbb W$ is normal, taking \cref{rem:SameProofWNormal} into account, the vector fields $D^{\phi}_{j}$, for $k+1\leq j\leq n$, in exponential coordinates, have a triangular form analogous to \cref{prop:CompatibleCoordinates}. Then, if we adopt the same notation and setting as the statement of \cref{prop:InverseLpischitz} and if we assume that the curves $\widetilde\phi\circ\widetilde\gamma_j$ are $\frac{1}{\deg j}$-H\"older with respect to the norm $\|\cdot\|_{\mathbb G}$, the same double-induction argument of the proof of \cref{prop:InverseLpischitz} implies that, for any $k+1\leq j\leq n$, and any $t\in I_j$, one has $\|\widetilde\phi\circ\widetilde\gamma_j(t)\|_{\mathbb G} \leq\! C\| v\|_{\mathbb G}$, and $\|\widetilde\gamma_j(t)\|_{\mathbb G}\leq\! C\|v\|_{\mathbb G}$, instead of \eqref{eqn:ThesisOfConverseLipschitz1}, and \eqref{eqn:ThesisOfConverseLipschitz2},respectively.
	
	Thus, by evaluating the general form of \eqref{eqn:ThesisOfConverseLipschitz1} at $j=n$ and  time $t$ corresponding to the endpoint of $I_n$, we get $\|\widetilde\phi(v)\|_{\mathbb G}\leq C\|v\|_{\mathbb G}$, with a constant $C$ only depending on $L$, $\mathbb W$, $\mathbb L$ and the basis adapted to the splitting. Then, if we assume that the bound $L$ on the H\"older constant of $\widetilde\phi\circ\widetilde\gamma_j$ is uniform with respect to the choice of the integral curve $\gamma_j$ of $D^{\phi}_j$, with $k+1\leq j\leq n$, we get, by exploiting the fact that $\sW$ is locally $D^\phi$-connectible according to \cref{rem:ConnectingCurves}, that 
	\begin{equation}\label{eqn:ConverseLipschitzWNormal}
	\|\widetilde{\phi}(v)\|_{\mathbb G}\leq C\|v\|_{\mathbb G},
	\end{equation}
	for every $v\in\widetilde U'\Subset\widetilde U$, where the constant $C$ only depends on $L$, $\mathbb W$, $\mathbb L$ and the chosen basis adapted to the splitting.
	
	Finally, if we do not necessarily assume $\widetilde{\phi}(e)=e$ as in the statement of \cref{prop:InverseLpischitz}, but we still assume that the $1/\deg j$-H\"older constant of $\widetilde\phi\circ\widetilde\gamma_j$ is uniformly bounded with respect to the choice of the integral curves $\widetilde\gamma_j$, the same translation argument as in the beginning of the proof of \cref{prop:IntrinsicLipschitzIntegralCurve}, joined with the conclusion \eqref{eqn:ConverseLipschitzWNormal} and the third point of \cref{prop:IntrinsicLipschitz}, implies that $\widetilde\phi$ is intrinsically Lipschitz on $\widetilde U'\Subset\widetilde U$. This last statement is the local converse of \cref{prop:IntrinsicLipschitzIntegralCurve}.
	
	We finally remark that the improved result we described here is the implication (3)$\Rightarrow$(1) of \cite[Theorem~4.2.16]{Koz15}.
\end{rem}

We now exploit \cref{prop:InverseLpischitz} to show a criterion to prove that a function is uniformly intrinsically differentiable in an arbitrary Carnot group. To do so we introduce the definition of the vertically broad* h\"older property in a fixed adapted basis. 

\begin{defi}[Vertically broad* h\"older and vertically broad h\"older condition]\label{def:vertically*Holder}
	Let $\mathbb W$ and $\mathbb L$ be two complementary subgroups of a Carnot group $\G$ with $\mathbb L$ horizontal and $k$-dimensional, and fix an adapted basis $(X_1,\dots,X_n)$ of the Lie algebra $\mathfrak g$ such that $\mathbb W=\exp({\rm span}\{X_{k+1},\dots,X_n\})$ and $\mathbb L=\exp({\rm span}\{X_{1},\dots,X_k\})$. Let us fix $\widetilde U\subseteq \mathbb W$ an open subset and let us denote $D^{\phi}_j\coloneqq D^{\phi}_{X_j}$ as defined in \eqref{eqn:DefinitionOfDj}. We say that a continuous function $\widetilde \phi\colon\widetilde U\to \mathbb L$ is \emph{vertically broad* h\"older} if for every $a_0$ in $U$ there exist $\delta_{a_0}>0$ and neighborhoods $U'_{a_0}\Subset U_{a_0}\Subset U$ of $a_0$ such that for every $a\in U'_{a_0}$ and every $j=m+1,\dots, n$ one can find a $C^1$ regular solution $E_j^\phi(a)\colon[-\delta_{a_0},\delta_{a_0}]\to U_{a_0}$ of the Cauchy problem
	\[
	\begin{cases}
	\dot\gamma=D^\phi_j\circ\gamma& \\
	\gamma(0)=a& 
	\end{cases}
	\]
	such that
	\begin{equation}\label{eqn:vertcicallybroad*holder}
	\lim_{\varrho\to0}\left(\sup\left\{\frac{|\phi(E_j^\phi(a,t))-\phi(E_j^\phi(a,s))|}{|t-s|^{1/{\mathrm{deg}j}}}: j=m+1,\dots,n, a\in U'_{a_0}, 0<|t-s|\leq\varrho \right\}\right) =0.
	\end{equation}
	We moreover say that $\widetilde\phi$ is \emph{vertically broad h\"older} if for every $V\Subset U$ one has
	\begin{equation*}
	\lim_{\varrho\to0}\left(\sup\left\{\frac{|\phi(\gamma(t))-\phi(\gamma(s))|}{|t-s|^{1/{\deg j}}}: j=m+1,\dots,n,\,\dot\gamma=D^{\phi}_{j}\circ\gamma, \gamma\subseteq V, 0<|t-s|\leq\varrho \right\}\right) =0.
	\end{equation*}
\end{defi}
\begin{rem}\label{rem:VerticallyBroad*InCoordinates}
	Notice that \cref{def:vertically*Holder} is a priori susceptible to the choice of a basis adapted to the splitting. However, when it is coupled with the broad* condition in the same basis, see \cref{defbroad*}, it is independent on this choice. See \cref{rem:Broad*InCoordinates} for details.
\end{rem}

\begin{prop}\label{prop:DPhiPhi=wBroadImpliesUid}
	Let $\mathbb W$ and $\mathbb L$ be complementary subgroups of a Carnot group $\mathbb G$, with $\mathbb L$ horizontal and $k$-dimensional, and fix an adapted basis $(X_1,\dots, X_n)$ of the Lie algebra $\mathfrak g$ such that $\mathbb W=\exp({\rm span}\{X_{k+1},\dots,X_n\})$ and $\mathbb L=\exp({\rm span}\{X_{1},\dots,X_k\})$. Let $\widetilde U\subseteq \mathbb W$ be open, let $\widetilde{\phi}\colon\widetilde{U}\to\mathbb L$ be a vertically broad* h\"older map and assume $\omega\colon U\subseteq \mathbb R^{n-k}\to\mathbb R^{k\times(m-k)}$ is a continuous function such that $D^{\phi}\phi=\omega$ in the broad* sense on $U$.
	 Then $\widetilde{\phi}\in \mathrm{UID}(\widetilde{U}, \mathbb W; \mathbb L)$. Moreover $\nabla^{\phi}\phi=\omega$, where $\nabla^{\phi}\phi$ is the intrinsic gradient defined in \cref{def:NablaPhiPhi}.
\end{prop}
\begin{proof}
	Up to bi-Lipschitz equivalence, we can prove the result choosing the anisotropic norm $\|x\|_{\mathbb G}\coloneqq \sum_{\ell=1}^n |x_{\ell}|^{1/\deg \ell}$. For the sake of readability, we give the proof in the case $k=1$. The proof for a larger $k$ only requires a typographical effort due to the fact that $\phi$ has more than one component. 

	Fix $a_0\in U$. According to the definition of vertically broad* h\"older, we find $\delta_{a_0}>0$, neighborhoods $U'_{a_0}\Subset U_{a_0}\Subset U$ of $a_0$ and $C^1$ maps $E^\phi_j(a)\colon[-\delta_{a_0},\delta_{a_0}]\to U_{a_0}$ satisfying the conditions of \cref{def:vertically*Holder}. Define for every $\varrho>0$ sufficiently small the quantity
	\begin{equation}\label{eqn:DefinitionOff}
	f(\varrho)\coloneqq\sup\left\{\frac{|\phi(E_j^\phi(a,t))-\phi(E_j^\phi(a,s))|}{|t-s|^{1/{\mathrm{deg}j}}}: j=m+1,\dots,n, a\in U'_{a_0}, 0<|t-s|\leq\varrho \right\},
	\end{equation}
	which by assumption converges to $0$ as $\varrho\to0$. Throughout the proof, we will often write $E_j^\phi$ instead of $E_j^\phi(a)$, where the dependence on the starting point has to be understood for a suitable $a\in U'_{a_0}$.
	
	We claim that $\phi$ is UID\ at $a_0$ and $\nabla^{\phi}\phi _{a_0}(b)=\omega(a_0)\cdot b^1$, for every $b\in \mathbb R^{n-1}$, where $b^1\in \mathbb R^{m-1}$ denotes the projection of $b$ onto the first $(m-1)$ components, and where $\omega(a_0)\cdot b^1\coloneqq\sum_{i=2}^{m} \omega_i(a_0)\cdot b^1_i$ denotes the usual scalar product on $\mathbb R^{m-1}$. 
	
	By \eqref{eqn:UidInCoordinates}, we just need to prove that
	\begin{equation}\label{eqn:EquationToShow}
	\lim_{\varrho\to 0}\left(\sup\left\{\frac{|\phi(b)-\phi(a)-\omega(a_0)\cdot(b^1-a^1)|}{\|\widetilde{\phi}(a)^{-1}a^{-1}b\,\widetilde{\phi}(a)\|}: a\in B(a_0,\varrho), \|a^{-1}b\|<\varrho\right\}\right)= 0.
	\end{equation}
	
	Since $D^{\phi}\phi=\omega$ in the broad* sense, by \cref{defbroad*}, we can find neighborhoods $ U''\Subset U' \Subset U'_{a_0}\Subset U$ of $a_0$ and $\delta>0$ such that, for every $j=2,\dots,m$, one has $E_j^\phi(U''\times[-\delta,\delta])\subseteq U'$. We can improve this observation using the triangular form of $D^{\phi}_j$, see \eqref{eqn:CoordinateDPhiJ}, and arguing as in \cref{rem:ConnectingCurves}. Indeed, the 
	sets $ U''$ and $U'$ can be chosen small enough such that, for every $a,b\in U''$, there exists a path connecting $a$ to $b$, entirely contained in $U'$, made first by a concatenation of the maps $E_2^\phi,\dots,E_m^\phi$ (defined accordingly to \cref{defbroad*}) and then by a concatenation $E^\phi_{m+1},\dots,E^\phi_n$ of integral curves of $D^{\phi}_{m+1},\dots,D^{\phi}_n$ provided by the vertically broad* h\"older condition. 
	
	Let use improve this conclusion. We know from \cref{lem:Invariance2IntegralCurve} that if $\widetilde{\gamma}$ is an integral curve of $D^{\phi}_j$, then $\widetilde{\gamma}_q\coloneqq q\cdot \widetilde{\gamma}\cdot (q_{\mathbb L})^{-1}$ is an integral curve of $D^{\phi_q}_j$, see \eqref{eqn:IntegralCurveOfDPhiQ}. Then, by possibly taking a smaller $U''$, we can suppose without loss of generality that there exist neighborhoods $ V''\Subset V'$  of $0$, such that, for every $a\in U''$ and every $b'\in  V''$, there exists a path connecting $0$ to $b'$, entirely contained in $V'$ made of $q$-translations of exponential maps. Indeed, if $a, b\in U''$, it is enough to set $q\coloneqq\widetilde{\phi}(a)^{-1}a^{-1}$ and build the concatenation of $q$-translated of the maps $(E_2^\phi)_q,\dots,(E_m^\phi)_q$, that are integral curves of $D^{\phi_q}_2,\dots,D^{\phi_q}_m$, respectively, by \cref{lem:Invariance2IntegralCurve}, and then chain it with the $q$-translated curves $(E^\phi_{m+1})_q,\dots,(E^\phi_n)_q$, that are integral curves of $D^{\phi_q}_{m+1},\dots,D^{\phi_q}_n$, respectively. By construction,  this concatenation connects $0$ to $b'=\widetilde\phi(a)^{-1}a^{-1}b\widetilde \phi(a)$. It suffices then to take a small enough \[V''\subseteq\bigcup\{\widetilde\phi(a)^{-1}a^{-1}b\,\widetilde \phi(a): a, b\in U''\}.\] 
	Moreover, by (d) of \cref{prop:PropertiesOfIntrinsicTranslation}, we get $\phi_q(0)=0$ and by (c) of \cref{prop:PropertiesOfIntrinsicTranslation}, one also has $\phi_q(b')=\phi_q(b')-\phi_q(0)=\phi(b)-\phi(a)$.
	Notice also that $\left(\widetilde{\phi}(a)^{-1}a^{-1}b\,\widetilde{\phi}(a)\right)^1 = b^1-a^1$, so that we have
	\begin{equation}\label{eqn:KeyEstimate2}
	\frac{|\phi(b)-\phi(a)-\omega(a_0)\cdot(b^1-a^1)|}{\|\widetilde{\phi}(a)^{-1}a^{-1}b\,\widetilde{\phi}(a)\|} = \frac{|\phi_q(b')-\phi_q(0)-\omega(a_0)\cdot (b')^1|}{\|b'\|}, \quad \forall a,b\in U''.
	\end{equation}
	
	For any $a, b\in U''$, we hence consider the concatenation $(E^\phi_2)_q,\dots,(E^\phi_n)_q$ of integral curves of $D^{\phi_q}_2,\dots,D^{\phi_q}_n$ entirely lying in $V'$, constructed as above, that connects $0$ to $b'$. Similarly to \eqref{eqn:Concatenation}, we denote the concatenation by the following chain 
	\begin{equation}\label{eqn:ConcatenationInTheorem}
	\begin{aligned}
	\bar b'_1\coloneqq( 0,0,\dots,0)&\rightarrow_{(E^\phi_{2})_q}\bar b'_{2}\coloneqq( 0,b'_2,r_{3,2},\dots,r_{n,2})\rightarrow_{(E^\phi_{3})_q}\dots  \\\dots 
	&\rightarrow_{(E^\phi_j)_q} \bar b'_j\coloneqq( 0,b'_2,\dots,b'_j,r_{j+1,j},\dots,r_{n,j})\rightarrow_{(E^\phi_{j+1})_q}\dots \\
	\dots &\rightarrow_{(E^\phi_{n})_q}\bar b'_n\coloneqq b'=(0,b'_2,\dots,b'_n),
	\end{aligned}
	\end{equation}
	where each $(E^\phi_j)_q$ is defined on  $I_j\subseteq[-|b_j'-r_{j,j-1}|,|b_j'-r_{j,j-1}|]$, with the convention  $r_{2,1}\coloneqq 0$. Since $D^{\phi}\phi=\omega$ in the broad* sense, by \cref{lem:Invariance2IntegralCurve} and in particular \eqref{eqn:HowOmegaChanges} and \eqref{eqn:FundamentalTheoremOfCalculusOnCurves}, we get that 
	\begin{equation}\label{eqn:FundamentalOfCalculus}
	(\phi_q\circ(E_j^\phi(a))_q)'(t)=(\overline \omega_j)_q((E^\phi_j(a))_q(t))=\omega_j(E^\phi_j((a,t))),
	\end{equation}
	 for all $j=2,\dots,m$, all $t\in I_j$ and all $a\in U''$; where the first equality follows from \eqref{eqn:FundamentalTheoremOfCalculusOnCurves} and the second one by the definition of $(\overline \omega_j)_q$, see \eqref{eqn:Translation} and \eqref{eqn:HowGammaChanges}. For every $a, b \in U''$, we can now perform the following estimates, which we subsequently explain:
	\begin{equation}\label{eqn:KeyEstimate}
	\begin{aligned}
	|\phi_q(b') -&\phi_q(0)-\omega(a_0)\cdot (b')^1|\\ &=\left|\sum_{j=2}^{m}\left(\phi_q(\bar b'_j)-\phi_q(\bar b'_{j-1})-\omega_j(a_0)b'_j\right)+\sum_{j=m+1}^n \left(\phi_q(\bar b'_j)-\phi_q(\bar b'_{j-1})\right)\right| \\
	&=\left|\sum_{j=2}^{m}\left(\omega_j(b_j^*)-\omega_j(a_0)\right)b'_j+\sum_{j=m+1}^n \left(\phi_q(\bar b'_j)-\phi_q(\bar b'_{j-1})\right)\right|  \\
	\leq \sup_{j=2,\dots,m}|&\omega_j(b_j^*)-\omega_j(a_0)|\|b'\|+ f\left(\sup_{j=m+1,\dots,n}| b'_j-r_{j,j-1}|\right)\sum_{j=m+1}^n|b'_j-r_{j,j-1}|^{1/\deg j}.
	\end{aligned}
	\end{equation}
	In the second equality we used \eqref{eqn:FundamentalOfCalculus} and, for every $j=2,\dots,m$, the point $b_j^*$ is on $E_j^\phi(I_j)$ and satisfies the conditions of Lagrange's Theorem. In the third inequality we estimated the first term with the supremum norm and the second term by exploiting the definition of $f$ in \eqref{eqn:DefinitionOff}, but replacing in \eqref{eqn:DefinitionOff} $\phi\circ E^{\phi}_j$ with $\phi_q\circ(E^\phi_j)_q$. Indeed, by \eqref{eqn:Translation}, the map $\phi_q\circ (E^\phi_j)_q$ is a Euclidean translation (in the horizontal coordinates) of $\phi\circ E^\phi_j$.
	
	Notice that, by continuity and with simple estimates coming from \cref{lemma333FS}, by definition of $b'$ and $b^*_j$ one has
	\[
	\begin{aligned}
	\lim_{\varrho\to0}& \left(\sup\{\|b'\|: a,b\in U'', \|a^{-1}b\|\leq \varrho\}\right)=0, \qquad\text{and}\\
	\lim_{\varrho\to0}&\left(\sup\{\|a_0^{-1}b^*_j\|: a,b\in U'', \|a^{-1}b\|\leq \varrho\}\right)=0, \quad\forall j=2,\dots, n,
	\end{aligned}
	\]
	where we implicitly mean that $b^*_j$ is also a function of the concatenation and the limit is uniform also with respect to that choice.
	Moreover, the concatenation in \eqref{eqn:ConcatenationInTheorem} uniformly collapses to $0$ as $\varrho\to 0$, since by continuity one has 
	\[
	\lim_{\varrho\to0}\left(\sup\{|b'_j-r_{j,j-1}|: a,b\in U'', \|a^{-1}b\|\leq \varrho\}\right)=0, \qquad \forall j=2,\dots,n,
	\]
	where again, the uniformity has to be understood also in the choice of the concatenation.
	We can then find two continuous functions $\alpha_1,\alpha_2\colon (0,+\infty)\to (0,+\infty)$ such that $\lim_{\varrho\to 0} \alpha_1(\varrho)=\lim_{\varrho\to0}\alpha_2(\varrho)=0$ and, combining \eqref{eqn:KeyEstimate} and \eqref{eqn:KeyEstimate2}, one has  
	\begin{equation}\label{eqn:Final}
	\begin{aligned}
	\frac{|\phi(b)-\phi(a)-\omega(a_0)\cdot(b^1-a^1)|}{\|\widetilde{\phi}(a)^{-1}a^{-1}b\,\widetilde{\phi}(a)\|}&\leq \sup_{j=2,\dots,m}\|\omega_j-\omega_j(a_0)\|_{L^\infty(B(a_0,\alpha_1(\varrho)))}\\&\hphantom{=}+f(\alpha_2(\varrho))\sum_{j=m+1}^n\frac{|b'_j-r_{j,j-1}|^{1/\deg j}}{\|b'\|},
	\end{aligned}
	\end{equation}
	for every sufficiently small $\varrho>0$ and every $a\in B(a_0,\varrho)$ and $b\in U''$ such that $\|a^{-1}b\|<\varrho$.
	We claim that we are in a position to apply \cref{prop:InverseLpischitz} to the function $\phi_q$ and to the curves $E^{\phi_q}_j$ that connect 0 to $b'$. Indeed, for $j=2,\dots,m$, the uniform bound on the Lipschitz constant of the map $\phi_q\circ(E^\phi_j)_q$ follows from \eqref{eqn:FundamentalOfCalculus} and the fact that $\omega$ is continuous, while for $j\geq m+1$ we use the fact that, since $f(\varrho)\to0$ as $\varrho\to0$, the maps $\phi_q\circ(E^\phi_j)_q$ are in $h^{1/\deg j}(I_j)$ with a uniform bound on the H\"older constant. In particular, by \eqref{eqn:ThesisOfConverseLipschitz2}, we get that for every $\varrho$ sufficiently small, there exists $C'>0$ depending on the uniform bound of $\omega$ on $U'$ such that for every $j=2,\dots,n$  
	\[
	|r_{j,j-1}|^{1/{\rm deg}j}\leq C'\|b'\|,
	\]
	and thus from \eqref{eqn:Final} we get
	\begin{equation*}
	\frac{|\phi(b)-\phi(a)-\omega(a_0)\cdot(b^1-a^1)|}{\|\widetilde{\phi}(a)^{-1}a^{-1}b\,\widetilde{\phi}(a)\|} \leq \sup_{j=2,\dots,m}\|\omega_j-\omega_j(a_0)\|_{L^\infty(B(a_0,\alpha_1(\varrho)))}+f(\alpha_2(\varrho))(n-m)(1+C').
	\end{equation*}
	By the continuity of each $\omega_j$, the proof follows by letting $\varrho\to0$.
\end{proof}
For the forthcoming corollaries we recall for the reader's benefit that we are going to use the notation in \cref{def:coordinateconletilde}.
\begin{coroll}\label{coroll:Koz}
	Let $\mathbb W$ and $\mathbb L$ be two complementary subgroups of a Carnot group $\G$, with $\mathbb L$ horizontal and $k$-dimensional and let $(X_1,\dots,X_n)$ be an adapted basis of the Lie algebra $\mathfrak g$ such that $\mathbb W=\exp({\rm span}\{X_{k+1},\dots,X_n\})$ and $\mathbb L=\exp({\rm span}\{X_{1},\dots,X_k\})$. Let $\widetilde U\subseteq \mathbb W$ be open and let $\widetilde \phi\colon\widetilde U\to \mathbb L$ and $\omega\colon U\subseteq \R^{n-k}\to\R^{k\times(m-k)}$ be two continuous functions. Assume that $\widetilde\phi$ is vertically broad* h\"older and assume $D^\phi\phi=\omega$ holds in the broad* sense on $U$. Then the graph of $\widetilde\phi$ is a co-horizontal $C^1_{\rm H}$-surface with tangents complemented by $\mathbb L$.
\end{coroll}
\begin{proof}
	It is enough to combine \cref{prop:DPhiPhi=wBroadImpliesUid} and \cref{prop:UidC1H}.
\end{proof}
\begin{coroll}\label{coroll:IDContinuousImpliesUid}
	Let $\mathbb W$ and $\mathbb L$ be two complementary subgroups of a Carnot group $\G$, with $\mathbb L$ horizontal and $k$-dimensional and let $(X_1,\dots,X_n)$ be an adapted basis of the Lie algebra $\mathfrak g$ such that $\mathbb W=\exp({\rm span}\{X_{k+1},\dots,X_n\})$ and $\mathbb L=\exp({\rm span}\{X_{1},\dots,X_k\})$. Let $\widetilde U\subseteq \mathbb W$ be open and let $\widetilde \phi\in \mathrm{ID}(\widetilde U,\mathbb W;\mathbb L)$ be a vertically broad* h\"older map and assume $\de ^\phi\!\phi$ is continuous on $\widetilde U$. Then $\widetilde\phi\in \mathrm{UID}(\widetilde U, \mathbb W;\mathbb L)$.
\end{coroll}
\begin{proof}
	It is enough to combine \cref{prop:IdGradContImpliesBroad} and \cref{prop:DPhiPhi=wBroadImpliesUid}.
\end{proof}
\begin{rem}\label{rem:StrategyIdGradContinuousUid}
	We do not know whether, in arbitrary Carnot groups, one can drop the condition of vertically broad* h\"older regularity in \cref{coroll:IDContinuousImpliesUid}. This is the case for step-2 Carnot groups, with $\mathbb L$ horizontal and one-dimensional, see \cref{coroll:step2.2}. 
	
	
\end{rem}
\begin{coroll}\label{coroll:StarImpliesWithoutStar}
	Let $\mathbb W$ and $\mathbb L$ be two complementary subgroups of a Carnot group $\G$, with $\mathbb L$ horizontal and $k$-dimensional and let $(X_1,\dots,X_n)$ be an adapted basis of the Lie algebra $\mathfrak g$ such that $\mathbb W=\exp({\rm span}\{X_{k+1},\dots,X_n\})$ and $\mathbb L=\exp({\rm span}\{X_{1},\dots,X_k\})$. Let $\widetilde U\subseteq \mathbb W$ be open and let $\widetilde \phi\colon\widetilde U\to \mathbb L$ and $\omega\colon U\subseteq \R^{n-k}\to\R^{k\times(m-k)}$ be two continuous functions. Assume $\widetilde\phi$ is vertically broad* h\"older and assume $D^\phi \phi=\omega$ holds in the broad* sense on $U$. Then $\widetilde\phi$ is vertically broad h\"older and $D^\phi\phi=\omega$ holds in the broad sense on $U$.
\end{coroll}
\begin{proof}
	It is enough to combine \cref{prop:DPhiPhi=wBroadImpliesUid}, \cref{prop:UidImpliesH1k} and \cref{prop:IdGradContImpliesBroad}.
\end{proof}

\subsection{Area formula for codimension-one graphs in terms of intrinsic derivatives}\label{sub:Area} 
We prove here that, if $\mathbb L$ is horizontal and one-dimensional, and $\widetilde\phi\in {\rm UID}(\widetilde U,\mathbb W;\mathbb L)$, with $\widetilde U$ open, then $D^{\phi}\phi=\nabla^{\phi}\phi$ holds {\bf in the sense of distributions}. For a precise definition of the distribution $D^{\phi}\phi$ the interested reader may soon read the first lines of the proof of \cref{prop2.22}. We also provide an area formula for ${\rm graph}(\widetilde\phi)$ in this case. For the case $\mathbb G= \mathbb H^n$, this formula was already obtained in \cite[Proposition 2.22 \& Remark 2.23]{ASCV06}. Recall that we are going to use the notation given in \cref{def:coordinateconletilde}.
\begin{prop}\label{prop2.22}
Let $\W$ and $\mathbb L$  be two complementary subgroups of a Carnot group $\G$ with $\mathbb L$ horizontal and one-dimensional, and let $(X_1,\dots, X_n)$ be an adapted basis of the Lie algebra $\mathfrak g$ such that $\mathbb L=\exp({\rm span}\{X_1\})$ and $\mathbb W=\exp({\rm span}\{X_2,\dots,X_n\})$. Let $\widetilde U \subseteq \mathbb W$ be open and consider $\widetilde \phi\in \mathrm{UID}(\widetilde U, \mathbb W; \mathbb L)$. Then the following facts hold.
\begin{itemize}
	\item[(a)] For every $j=2,\dots,m$, the distribution $D_j^{\phi}\phi\coloneqq D_{X_j}^{\phi}\phi$ is well-defined on $U$.
	\item[(b)] For every $a\in U$, there exist $\delta>0$ and a family of functions $\{\phi_\eps\in C^1(B(a,\delta)):\eps\in (0,1)\}$ such that
	\[
	\lim_{\eps\to0}\phi_\eps=\phi \quad\text{and}\quad\lim_{\eps\to0}D_j^{\phi_\eps}\phi_\eps=\nabla^\phi\phi(X_j)  \quad\text{in $L^\infty(B(a,\delta))$},
	\]
	for every  $j=2,\dots,m$.
	\item[(c)] The equation
\begin{equation}\label{eqn:DistributionalEquation}
D_j^{\phi}\phi = \nabla^{\phi}\phi(X_j)\eqqcolon\nabla^\phi_j\phi,
\end{equation}
holds in the distributional sense on $U$, for every $j=2,\dots,m$. Here $\nabla^{\phi}\phi$ is the intrinsic gradient of $\phi$, see \cref{def:NablaPhiPhi}.
\item[(d)] The subgraph of $\phi$ defined by $ E_\phi\coloneqq\{ w\cdot\exp(tX_1): w\in U, t<\phi(w)\}$ has locally finite $\G$-perimeter\footnote{Here we take the usual definition of the horizontal perimeter with respect to the orthonormal basis $(X_1,\dots,X_m)$, see \cite[Definition 2.18]{FSSC03a}.} in $U\cdot \exp(\R X_1)$ and its $\G$-perimeter measure $|\partial E_\phi|_{\G}$ is given by
\begin{equation}\label{eqn:PerimeterSubgraph}
|\partial E_\phi|_{\G}(V) =\int _{\Phi^{-1}( V)} \sqrt{1+  |\nabla^{ \phi} \phi|_{m-1}^2} \,  \de\mathscr{L}^{n-1},
\end{equation}
for every Borel set $ V\subseteq U\cdot \exp{(\R X_1)} $, where $\Phi\colon U\to \R^n$ is the graph function of $\widetilde \phi$ composed with the exponential coordinates, and with a little abuse of notation we wrote $U\cdot \exp(\R X_1)$ meaning the set $\widetilde U\cdot \exp(\R X_1)$ embedded in $\R^n$ through exponential coordinates.

Moreover, the set $\mbox{graph}(\widetilde\phi)$ has a unit normal given, up to a sign, by 
\begin{equation}\label{eqn:Normal}
\nu_{E_{\phi}} = \left(-\frac{1}{\sqrt{1+  |\nabla^{ \phi} \phi|_{m-1}^2}},\frac{\nabla_2^{\phi}\phi}{\sqrt{1+  |\nabla^{ \phi} \phi|_{m-1}^2}},\dots,\frac{\nabla_m^{\phi}\phi}{\sqrt{1+  |\nabla^{ \phi} \phi|_{m-1}^2}}\right)\in \R^m.
\end{equation}
\end{itemize}
\end{prop} 
\begin{proof} 
Notice that the fact that $D^{\phi}_j\phi$ is a well-defined distribution on $U$ is a consequence of \cref{prop:CompatibleCoordinates} and the fact that $\mathbb L$ is one-dimensional. Indeed, in coordinates, we get that the vector field $D^{\phi}_j$ is the sum of terms $g(x)\phi^h\partial_{x_i}$, for some polynomial function $g$ of the coordinates of $\mathbb W$, some integer $h\geq 0$, and some $i=2,\dots,n$. Thus in order to define the distribution $D^{\phi}_j\phi$, we only need to define $g(x)\phi^h\partial_{x_i}\phi\coloneqq g(x)\frac{1}{h+1}\partial_{x_i}\phi^{h+1}$. Since $g(x)\frac{1}{h+1}\partial_{x_i}\phi^{h+1}$ is well-defined in the sense of distributions, because $\phi$ is continuous, (a) is proved.

We now show that equality \eqref{eqn:DistributionalEquation} holds pointwise if $\widetilde{\phi}\colon\widetilde{U}\subseteq \mathbb W\to\mathbb L$ is of class $C^1$ in $\widetilde{U}$. In case $\widetilde \phi\in C^1(\widetilde U)$, then, by \cite[Theorem 4.9]{DiDonato18}, one has $\widetilde{\phi}\in \mathrm{UID}(\widetilde U, \mathbb W; \mathbb L)$. \\ We can assume without loss of generality that $e\in\widetilde{U}$ and $\widetilde{\phi}(e)=e$. Indeed, if we want to prove identity \eqref{eqn:DistributionalEquation} in $a\in U$, we may consider $\widetilde{\phi}_p$ with $p\coloneqq\widetilde{\phi}(a)^{-1}\cdot  a^{-1}$, and use the invariance properties given by \cref{lem:Invariance1DPhi} and \cref{rem:InvarianceOfIdByTranslations}, and then notice that $\widetilde{\phi}_p(e)=e$. 
 By \cref{prop:UidC1H}, since $\widetilde \phi$ is ${\rm UID}$, the set ${\rm graph}(\widetilde{\phi})$ is a $C^1_{\rm H}$-hypersurface and therefore there exist a neighborhood $\widetilde V$ of $e$ in $\G$ and a function $f\in C^1_{\rm H}(\widetilde V)$ such that 
 \begin{equation}\label{eqn:proprietaC1H}
 \mathrm{graph}(\widetilde\phi)\cap \widetilde V=\{p\in \G: f(p)=0\}\cap \widetilde V \quad \mbox{and} \quad \nabla_\G f\neq 0 \text{ on $\widetilde V$}.
 \end{equation}
  Then, for every sufficiently small $\eps>0$ and for every $j=2,\dots,n$, one has
  \[
  f\Big(\exp (\eps X_j) \widetilde{\phi} (\exp (\eps X_j)) \Big)=0.
  \]
  Therefore, with a little abuse of notation, one can differentiate with respect to $\eps$ to get
\begin{equation}\label{eqn:DeriveF}
\begin{split}
0& = \frac \de {\de\eps }_{|_{\eps=0}} f\Big(\exp (\eps X_j) \widetilde{\phi} (\exp (\eps X_j)) \Big)\\
&= \frac \de {\de\eps }_{|_{\eps=0}} f\Big(\widetilde{\phi} (\exp (\eps X_j)) \Big) + \frac \de {\de\eps }_{|_{\eps =0}} f\Big(\exp (\eps X_j)  \Big) \\
&= X_1 f_{|_e}\left(\frac \de {\de\eps }_{|_{\eps =0}} \phi  (\eps X_j) \right) + X_jf_{|_e}\\
&= X_1 f_{|_e}\left(\frac \de {\de\eps }_{|_{\eps =0}} ( \phi \circ \pi_{\mathbb W})(\eps X_j) \right) + X_jf|_e\\
&= (X_1 f)_{|_e} D_j^\phi\phi|_e  + (X_jf)_{|_e},
\end{split}
\end{equation}
where we used the fact that $\G=\mathbb W\cdot\mathbb L$ and exploited the fact that $\widetilde{\phi}$ takes values in $\mathbb L=\exp({\rm span}\{X_1\})$. The last equality follows by using the definition of $D^{\phi}_j$ acting on $\phi$, see \eqref{eqn:DefinitionOfDj}. The claim is then obtained by \eqref{eqn:DeriveF} and \eqref{DPHI2}.

To prove (b), we use some ideas of \cite[Theorem~2.1]{FSSC03b} to show that, for any $a\in U$, there exist $\delta>0$ and a family of functions $\{\phi_{\eps}\in C^1(U): 0<\eps<1\}$, such that 
\[
\phi_{\eps}\to\phi, \qquad {\rm and} \quad  D_j^{\phi_{\eps}}\phi_{\eps} \to \nabla^{\phi}\phi(X_j), \quad \forall j=2,\dots,m,
\]
uniformly in the Euclidean ball $B_{\rm e}(a,\delta)\Subset U$, as $\eps\to 0$.  \\
 By \cref{prop:UidC1H}, we can find a neighborhood $\widetilde V$ of $a\cdot \widetilde \phi(a)$ and $f\in C^1_H(\widetilde V)$ satisfying \eqref{eqn:proprietaC1H}. Let $\delta>0$  be such that, setting  $B\coloneqq B_{\rm e}(a,\delta)$, one has $\widetilde{B}\cdot\widetilde{\phi}(\widetilde{B})\subseteq \widetilde V$. Then, up to reducing $\delta$ and $\widetilde V$ and a regularization argument analogous to \cite[Step 1 of Theorem~2.1]{FSSC03b}, we can construct a family  $\{f_{\eps}\in C^1(\widetilde V): 0<\varepsilon<1\}$ such that
\begin{equation}\label{eqn:UniformConvergenceXj}
\lim_{\eps\to0}\left(\sup_{\widetilde V}|X_jf_{\eps}- X_jf|\right)=0, \qquad \forall j=1,\dots,m.
\end{equation}
 Since $\widetilde{\phi}$ takes values in $\exp({\rm span}\{X_1\})$, by \cref{prop:UidC1H}, we can assume without loss of generality that $X_1f>0$ on $\widetilde V$, since $X_1f\neq 0$ and we are free to possibly exchange $f$ with $-f$. By \eqref{eqn:UniformConvergenceXj}, we can find $\varepsilon_0>0$ such that $X_1f_{\eps}>0$ on $\widetilde V$, for every $\eps\in (0,\eps_0)$. For any such $\eps>0$, by the Euclidean implicit function theorem, we can find $\widetilde \phi_{\eps}$, defined on $\widetilde B$, such that
 \[
 \mathrm{graph}(\widetilde \phi_\eps)\cap \widetilde{B}\cdot\widetilde{\phi}_\eps(\widetilde B)=\{p\in \G: f_\eps(p)=0\}\cap\widetilde{B}\cdot\widetilde{\phi}_\eps(\widetilde B).
 \]
 Moreover, since $f_\eps$ is smooth, then also $\widetilde \phi_\eps$ is smooth. In particular, for every $b\in \widetilde{B}$, one has
\[
f_{\eps}(b\cdot\widetilde{\phi}_{\eps}(b))=0, \qquad \forall \eps\in (0,\varepsilon_0).
\]
From \cite[Step 3 of Theorem 1.2]{FSSC03b} we deduce that
\begin{equation}\label{eqn:UniformConvergencePhi}
\lim_{\eps\to 0}\left(\sup_B|\phi_{\eps}-\phi|\right)=0.
\end{equation}
Denote by $\Phi_{\eps}\colon B\to \R^n$ the graph function of $\widetilde\phi_\eps$ composed with the exponential coordinates. Since $\widetilde{\phi}_{\eps}$ is of class $C^1$, by using the pointwise version of \eqref{eqn:DistributionalEquation} for $C^1$ functions and \eqref{DPHI2}, we deduce that
\begin{equation}\label{eqn:nonnevica}
D^{\phi_{\eps}}_j\phi_{\eps}(x)=-\frac{X_jf_{\eps}}{X_1f_{\eps}}\circ \Phi_{\eps}(x), \qquad \forall j=2,\dots,m, \quad \forall x\in B.
\end{equation}
Then, by \eqref{eqn:nonnevica}, \eqref{eqn:UniformConvergenceXj}, \eqref{DPHI2} and \eqref{eqn:UniformConvergencePhi}, we conclude that, for any $j=2,\dots, m$, the family $D^{\phi_{\eps}}_j\phi_{\eps}$ converges uniformly on $B$ to $-\frac{X_jf}{X_1f}\circ\Phi=\nabla^{\phi}\phi(X_j)$, as $\varepsilon \to 0$.\\

The proof of (c) follows directly from (b) and the particular form of the distribution $D^{\phi}_j\phi$ we discussed at the beginning of this proof. In fact from the convergence proved in (b), we know that, for any $a\in U$, there exists $\delta>0$ such that $D_j^\phi \phi=\nabla^\phi\phi(X_j)$ on $B(a, \delta)$ in the sense of distributions, for every $j=2,\dots, m$. It is then enough to consider a locally finite countable open sub-covering $\{B(a_h,\delta_h):h\in \mathbb N\}$ of $U$ and build a partition of unity subordinate to it. The fact that $D^\phi_j\phi=\nabla^\phi \phi(X_j)$ holds in the sense of distributions on $U$ is a consequence of the local identity, the linearity of distributions, and a standard argument using the partition of unity. 

To prove (d), we first notice that, by \cite[Theorem~2.1]{FSSC03b} and \cref{prop:UidC1H}, we know that, for every $p\in \mathrm{graph}(\widetilde \phi)$ there exists a neighborhood $\widetilde V'$ of $p$ and a function $ f\in C_H^1(\widetilde V')$, with $X_1 f>0$ on $\widetilde V'$, representing $\mathrm{graph}(\widetilde \phi)$ as non critical level set and such that
\[
|\partial  E_\phi|_\G ( V) = \int _{\Phi^{-1}(V)} \frac{ |\nabla_\G  f|}{X_1 f}\circ \Phi \, \de \mathscr{L}^{n-1} 
\]
holds for every Borel set $ V\subseteq  V'$. Now by \eqref{DPHI2}, we can write
\begin{equation}\label{eq:arealocale}
|\partial  E|_\G ( V) = \int _{\Phi^{-1}(V)} \sqrt{1+|\nabla^\phi \phi|^2_{m-1}} \de \mathscr{L}^{n-1}, \quad\text{for every Borel set }\quad V \subseteq  V'.
\end{equation}
 Since the right hand side of \eqref{eq:arealocale} does not depend on the choice of $ f$, by a covering argument we can extend it to every Borel subset $V$ of $ U\cdot \exp(\R X_1)$.
 
 The explicit expression of the unit normal in \eqref{eqn:Normal} comes from the fact that the graph of $\widetilde\phi$ is locally the zero-level set of $ f$. Thus, the unit normal of ${\rm graph}(\widetilde\phi)$ is in the direction of $\nabla_{\mathbb H} f$, and taking  \eqref{DPHI2} into account, one has $\nabla_j^{\phi}\phi=-\frac{X_j f}{X_1 f}\circ\Phi$ for every $j=2,\dots,m$, and then we conclude by normalization. 
\end{proof} 

\begin{rem}\label{rem:AreaFormula}
	For what concerns the area formula in arbitrary Carnot groups, in \cite[Theorem~1.2]{Mag17}, the author proves that 
	\begin{equation}\label{eqn:AreaFormulaMagnani}
	P_\G(E)=\beta(d,\nu_E)\mathscr S^{Q-1}\res \mathcal FE,
	\end{equation}
	for any set $E$ of finite perimeter with $C^1_{\rm H}$-rectifiable reduced boundary $\mathcal FE$. The density $\beta$ is explicitly computed in \cite[Theorem~3.2]{Mag17} and depends on the metric $d$ and on the normal $\nu_E$ of $E$ that is defined in the sense of Geometric Measure Theory. Moreover, in case the distance $d$ is vertically symmetric, $\beta$ is a constant that only depends on the group $\G$ and on the metric $d$, see \cite[Theorem~6.3]{Mag17}. We finally notice that every Carnot group admits a metric $d$ that is vertically symmetric, see \cite[Theorem~5.1]{FSSC03a}.
	For a survey on the area formula in Carnot groups, we refer the reader to \cite[Section~4]{SC16}.
	
	Notice that if $\widetilde{\phi}$ is in ${\rm UID}(\widetilde U,\mathbb W;\mathbb L)$, the set ${\rm graph}(\widetilde\phi)$ is a $C^1_{\rm H}$-hypersurface by \cref{prop:UidC1H}. Then, by definition of $C^1_{\rm H}$-rectifiability and by \cite[Theorem~2.1]{FSSC03b}, the subgraph $E_{\phi}$ of $\widetilde\phi$ has a $C^1_{\rm H}$-rectifiable reduced boundary $\mathcal{F}E_{\phi}={\rm graph}(\widetilde\phi)$. Thus we are in a position to apply \cref{prop2.22} and \cite[Theorem~1.2]{Mag17}, and in particular to compare \eqref{eqn:PerimeterSubgraph} with \eqref{eqn:AreaFormulaMagnani} in order to obtain the explicit representation 
	\begin{equation}\label{eqn:AreaFormula2}
	\int_{\widetilde V\cap {\rm graph}(\widetilde \phi)} \beta(d,\nu_{E_\phi})\de \mathscr S^{Q-1} = \int_{\Phi^{-1}(V)} \sqrt{1+|\nabla^{\phi}\phi|_{m-1}^2}\de\mathscr L^{n-1},
	\end{equation}
	for every Borel set $\widetilde V\subseteq \widetilde U\cdot \exp(\mathbb R X_1)$.
	
	A general area formula for a $C^1_{\rm H}$-surface $\Sigma$, valid in an arbitrary Carnot groups, has been very recently obtained in \cite[Theorem~1.1]{JNGV20}. If we call $\alpha$ the Hausdorff dimension of $\Sigma$, and we suppose $\Sigma={\rm graph}(\widetilde\phi)$, this formula allows for a representation of $\mathscr S^\alpha\res \Sigma$ as an integral on $\widetilde U$ of a properly defined area element with respect to $\mathscr S^{\alpha}\res \mathbb W$, see \cite[Lemma~3.2]{JNGV20}. According to the previous equation \eqref{eqn:AreaFormula2}, we thus get that for an arbitrary Carnot group $\mathbb G$, and in case $\mathbb L$ is one-dimensional, the area element of \cite[Theorem~1.1]{JNGV20} is, up to the function $\beta$, explicitly written in terms of the intrinsic gradient of $\widetilde \phi$. 
	
	Eventually, by the recent work \cite{CM20}, in particular \cite[Eq.\ (43) after Theorem~4.2]{CM20}, we get that on $\mathbb H^n$ equipped with a vertically symmetric distance, in case $\mathbb L$ is horizontal and $k$-dimensional, one can explicitly write the area element of \cite[Theorem~1.1]{JNGV20} in terms of the intrinsic gradient $\nabla^{\phi}\phi$. 
\end{rem}

\begin{coroll}\label{coroll:Broad*ImpliesDistributionally}
	Let $\mathbb W$ and $\mathbb L$ be two complementary subgroups of a Carnot group $\G$, with $\mathbb L$ horizontal and one-dimensional, and let $(X_1,\dots, X_n)$ be an adapted basis of the Lie algebra $\mathfrak g$ such that $\mathbb W=\exp({\rm span}\{X_{2},\dots,X_n\})$ and $\mathbb L=\exp({\rm span}\{X_{1}\})$. Let $\widetilde U\subseteq \mathbb W$ be open, and let $\widetilde \phi\colon\widetilde U\to \mathbb L$ and $\omega\colon U\subseteq \R^{n-1}\to\R^{m-1}$ be two continuous functions. Assume $\widetilde\phi$ is vertically broad* h\"older and assume that $D^\phi \phi=\omega$ in the broad* sense on $U$. Then $D^\phi\phi=\omega$ in the sense of distributions and $\omega=\nabla^\phi\phi$.
\end{coroll}
\begin{proof}
	It is enough to combine \cref{prop:DPhiPhi=wBroadImpliesUid} and item (c) of \cref{prop2.22}.
\end{proof}

\subsection{Relations between intrinsic differentiability and local approximability}\label{sub:Relations}

\begin{rem}\label{remprop2.22}
 Point (b) of \cref{prop2.22} can actually be generalized to the case $\mathbb L$ horizontal and $k$-dimensional. The computation for the $k$-dimensional case is similar but one should pay attention to the fact that  if $f\eqqcolon(f^{(1)},\dots, f^{(k)})\in C^1_{\rm H}(\widetilde V;\R^k)$ is a vector valued map, its horizontal gradient is a $(k \times m)$-dimensional matrix (see \cref{remarkDPHI2Vectorial}). Using a regularization argument that is similar to \cite[Step 1 of Theorem 2.1]{FSSC03b} as in the proof of \cref{prop2.22}, we can find a family of functions $\{f_{\eps}\in C^1(\widetilde V;\R^k):0<\eps<1\}$, such that each component $X_jf^{(i)}_{\eps}$ converges uniformly to $X_jf^{(i)}$ for every $i=1,\dots, k$ and $j=1,\dots, m$ as $\eps\to 0$, and such that, for every $\eps\in (0,1)$, the associated matrix $\nabla_{\mathbb L}f_\eps$ defined in \cref{remarkDPHI2Vectorial} has det$\nabla_{\mathbb L}f_\eps \ne 0$ on $\widetilde V$. 
 
 Then in order to prove point (b) of \cref{prop2.22} in this general case, we take the previous changes into account, the straightforward changes in \eqref{eqn:DeriveF}, and we run the same computations of the proof of \cref{prop2.22} by using \eqref{DPHI2Vectorial} instead of \eqref{DPHI2} at the end of the argument. 
\end{rem}

\begin{rem}\label{rem:PointwiseAndLocalApproximability}
 	The statement (i)$\Rightarrow$(ii) of \cite[Theorem~5.1]{ASCV06}
 	claims that, in the Heisenberg groups $\mathbb H^n$, if $\sL$ is one-dimensional and $\phi\in {\rm UID}(U,\sW;\sL)$, then there exists $\omega\in C(U;\R^{2n-1})$ such that $\nabla ^\phi\phi=\omega$ in the sense of distributions and there exists a family $\{\phi_\eps\in C^1(U): \eps\in (0,1)\}$ such that $\phi_\eps\to \phi$ and $\nabla^{\phi_\eps}\phi_\eps\to \omega$ uniformly on any compact subset $K\subseteq U$. 
 
  	However, the implication (5.17)$\Rightarrow$(5.19) in its proof is imprecise. In sight of this, one could replace point (ii) with a local version of it in which the approximating functions $\{\phi_{\eps}\}$ depend on the point $a$ and the convergence is uniform in a neighborhood $B(a,\delta)$, see \cite[Proposition 4.6]{ASCV06}. This does not affect the validity of the proofs, that only refer to \cite[Lemma~5.6]{ASCV06}, which holds true with the  weakened approximation assumptions, since it has a local statement. Indeed, the proof of (ii)$\Rightarrow$(i) of \cite[Theorem 5.1]{ASCV06} just needs the local approximation. To the best of our knowledge, all the references to \cite[Theorem~5.1]{ASCV06} just require the local approximation: see point (3) \cite[Theorem~5.8]{DiDonato18}, point (3) of \cite[Theorem~8.2]{DiDonato19}, \cite[Theorem~1.3]{Cor19}, \cite[Proposition~4.4]{Cor19}, \cite[Theorem~2.7]{BSC10b}, \cite[Theorem~1.1 \& Corollary~1.4]{BSC10a}, \cite[Theorem~2.7]{BCSC14}.
  	
  	Anyway, the original formulation of (i)$\Rightarrow$(ii) of \cite[Theorem~5.1]{ASCV06} can be fixed with the approximation argument exploited in the proof of \cite[Theorem~1.2]{MV12} and we expect this to hold true in all Carnot groups of step 2. We warmly thank Francesco Serra Cassano and Davide Vittone for precious suggestions.
\end{rem}

\begin{prop}\label{prop4.12}
Let $\mathbb W$ and $\mathbb L$ be two complementary subgroups in a Carnot group $\mathbb G$, with $\mathbb L$ $k$-dimensional and horizontal. Let $(X_1,\dots, X_n)$
be an adapted basis of the Lie algebra $\mathfrak g$ 
such that $\mathbb L=\exp({\rm span}\{X_1,\dots,X_k\})$ and
$\mathbb W=\exp({\rm span}\{X_{k+1},\dots,X_n\})$. Let $\widetilde U\subseteq \mathbb W$ be an open set and let $\widetilde{\phi}\colon\widetilde{U}\to\mathbb L$ be a continuous function. 

Assume there exists $\omega \in C(U;\R^{k\times(m-k)})$ such that, for every $a\in U$, there exist $r>0$,  and a family of functions $\{\phi_\eps \in C^1(B(a,r);\R^k): 0<\eps<1\}$ satisfying
\begin{equation*}
\lim_{\eps\to0}\phi _\eps=\phi  \quad\text{in $L^\infty(B(a,r))$}\quad  \text{and} \quad  \lim_{\eps\to0}D^{{\phi _\eps}} {\phi _\eps}=\omega \quad\text{in $L^\infty(B(a,r);\R^{k\times(m-k)})$}.
\end{equation*}
Then $\phi $ is a broad* solution in $U$ of the system $ D^{{\phi}} {\phi} = \omega $.
\end{prop}

\begin{proof}
The proof closely follows the argument used in \cite[Lemma 5.6]{ASCV06}. For simplicity, we consider the case $k=1$. The case $k\geq 2$ can be reached with the same proof and some typographical effort.\\
Let $a\in U$ and fix $j\in \{2,\dots , m\}$ and $\eps >0$. 
Since $\phi_\eps\in C^1(B(a,r))$, we can find $0<\delta_2(\eps)<\delta_1<r$ and a map $E^{\phi_\eps}_j \colon \overline{B(a, \delta_2 (\eps))}\times[-\delta _2(\eps),\delta _2(\eps)] \to  \overline{B(a,\delta_1)}$ such that $E^{\phi_{\eps}}_j(b,\cdot)$ is the \textbf{unique} solution of the Cauchy problem
\begin{equation*}
\left\{
\begin{array}{l}
\gamma '=D^{\phi _\eps} _{j} \circ\gamma, \\
\gamma (0)= b,
\end{array}
\right.
\end{equation*}
 in the interval $[-\delta_2(\eps),\delta_2(\eps)]$, for every $b\in \overline{B(a,\delta_2(\eps))}$.
 By Peano's estimate on the existence time for solutions of ordinary differential equations (see e.g.\ \cite[Theorem 1]{Mustafa}) we can choose $\delta_2 (\eps)=C/\| P_j(x,\phi _\eps) \|_{L^\infty ( \overline{B(a, \delta_1 ) })} $, with $C$ only depending on $\delta _1$ and with $P_j$ a polynomial function of the coordinates $x$ of $\mathbb W$ and the components of $\phi_{\eps}$. This polynomial function depends only on the structure of $D^\phi_j$, see \eqref{eqn:CoordinateDPhiJ}. In particular, since $\phi_\eps$ is converging uniformly on $B(a,\delta_1)$, then $\delta_2(\eps)$ has a positive lower bound $M$ independent on $\eps$ and we are going to verify the conditions of \cref{defbroad*} with $\delta_2\leq M$.\\ Since $\phi_\eps$ are bounded on $B(a,\delta_1)$ uniformly in $\varepsilon>0$ and $D^{\phi_\eps}_j$ are vector fields with polynomial coefficients in $\phi_\eps$ and, possibly, in some of the coordinates, the functions $E^{\phi_\eps}_j$ are equi-Lipschitz with respect to $\varepsilon$ on the compact set $\overline{B(a, \delta_2)}\times[-\delta_2,\delta_2]$. By Arzel\'a-Ascoli Theorem, we can therefore find an infinitesimal sequence $(\eps _h)$ in $(0,1)$ such that $E^{\phi_{\eps _h}}_j$ converges to some continuous function $E_j^\phi $ uniformly on $\overline{B(a, \delta_2)}\times[-\delta_2,\delta_2]$. By definition of $E^{\phi_{\eps_h}}_j$, one has
\begin{equation*}
E^{\phi_{\eps _h}}_j(b,t)= b + \int_0^t D^{\phi _{ \eps_h} }_{j} \left(E^{ \phi_{\eps_h}}_j(b,s)\right)\de\! s \qquad \text{and}
\end{equation*}
\begin{equation*}
\phi _{\eps_h}\left(E_j^{\phi_{\eps_h}}(b,t)\right) - \phi _{\eps_h}\left(E_j^{\phi_{ \eps_h}}(b,0)\right)= \int_0^t  D^{\phi _{ \eps_h} }_{j} \phi _{ \eps_h}\left(E_j^{\phi_{\eps_h}}(b,s)\right)\de \!s,
\end{equation*}
for every $b\in \overline{B(a,\delta_2)}$ and every $t\in [-\delta_2,\delta_2]$. \\ By letting $h\to \infty$ and using that all the involved convergences are uniform, we get
\begin{equation*}
E _j^\phi(b,t)= b+ \int_0^t D^\phi _{j}\left({E_j^\phi(b,s)}\right)\de \!s,\qquad \text{and}
\end{equation*}
\begin{equation*}
\quad\phi(E_j^\phi(b,t))-\phi(E_j^\phi(b,0))= \int_0^t \omega _{j} (E _j^\phi(b,s)) \de\! s,
\end{equation*}
 for every $b\in \overline{B(a,\delta_2)}$ and every $t\in [-\delta_2,\delta_2]$, which are the conditions we were looking for to make $D^\phi\phi=\omega$ hold in the broad* sense.
\end{proof}

\begin{coroll}
Let $\mathbb W$ and $\mathbb L$ be two complementary subgroups in a Carnot group $\mathbb G$, with $\mathbb L$ horizontal and k-dimensional. Let $(X_1,\dots, X_n)$
be an adapted basis of the Lie algebra $\mathfrak g$ 
such that $\mathbb L=\exp({\rm span}\{X_1,\dots,X_k\})$ and
$\mathbb W=\exp({\rm span}\{X_{k+1},\dots,X_n\})$. Let $\widetilde U\subseteq \mathbb W$ be open and let  $\widetilde{\phi}\in \mathrm{UID}(\widetilde{U},\mathbb W;\mathbb L)$. 
Then, there exists  $\omega \in C(U ; \R^{k\times (m-k)})$ such that $\phi $ is a broad* solution in $U$ of the system $ D^{{\phi}} {\phi} = \omega $. Moreover, $\omega=\nabla^\phi\phi$.
\end{coroll}
\begin{proof}
	It is enough to choose $\omega=\nabla^\phi\phi$, which is continuous taking \cref{prop:ContinuityOfDifferentialUid} into account, since $\widetilde{\phi}\in \mathrm{UID}(\widetilde{U},\mathbb W;\mathbb L)$. The proof follows by combining item (b) of \cref{prop2.22}, which also holds in case $\mathbb L$ is $k$-dimensional, see \cref{remprop2.22}, together with \cref{prop4.12}.
\end{proof}

\subsection{Main theorem}\label{sub:Main}
Now we are in a position to give the following theorem, which is a generalization of \cite[Theorem~5.8]{DiDonato18} to all Carnot groups.
\begin{theorem}\label{thm:MainTheorem}
	Let $\mathbb W$ and $\mathbb L$ be two complementary subgroups in a Carnot group $\mathbb G$, with $\mathbb L$ horizontal and k-dimensional. Let $\widetilde U\subseteq \mathbb W$ be open and let  $\widetilde{\phi}\colon\widetilde{U}\subseteq\mathbb W\to\mathbb L$ be a continuous function. Fix a graded basis $(X_1,\dots,X_n)$  of the Lie algebra $\mathfrak g$ such that $\mathbb L=\exp({\rm span}\{X_1,\dots,X_k\})$ and $\mathbb W=\exp({\rm span}\{X_{k+1},\dots,X_n\})$.  Then, the following facts are equivalent.
	\begin{itemize}
		\item[(a)] $\widetilde\phi\in {\rm UID}(\widetilde U, \mathbb W; \mathbb L)$;
		\item[(b)] $\phi$ is vertically broad* h\"older and there exists $\omega\in C(U;\R^{k\times(m-k)})$ such that, for every $a\in U$, there exist $\delta>0$ and a family of functions $\{\phi_\eps\in C^1(B(a,\delta);\R^k):\eps\in (0,1)\}$ such that
		\[
		\lim_{\eps\to0}\phi_\eps=\phi \quad\text{and}\quad\lim_{\eps\to0}D_j^{\phi_\eps}\phi_\eps=\omega_j  \quad\text{in $L^\infty(B(a,\delta);\R^k)$},
		\]
		for every  $j=k+1,\dots,m$;
		\item[(c)] $\phi$ is vertically broad h\"older and there exists $\omega\in C(U;\R^{k\times(m-k)})$ such that $D^\phi \phi=\omega$ in the broad sense.
		\item[(d)] $\phi$ is vertically broad* h\"older and there exists $\omega\in C(U;\R^{k\times(m-k)})$ such that $D^\phi \phi=\omega$ in the broad* sense.
	\end{itemize}
Moreover, if any of the previous holds, then $\omega=\nabla^\phi\phi$.
\end{theorem}
 \begin{proof}
	(a)$\Rightarrow$(b) follows by combining \cref{prop:UidImpliesH1k}, item (b) of \cref{prop2.22} in the $k$-dimensional case following the lines of \cref{remprop2.22}. (b)$\Rightarrow$(d) follows from \cref{prop4.12}.  
	(d)$\Rightarrow$(a) follows from \cref{prop:DPhiPhi=wBroadImpliesUid}. (a)$\Rightarrow$(c) follows from \cref{prop:UidImpliesH1k}, \cref{prop:IdGradContImpliesBroad} and the continuity of $\omega$ follows from \cref{prop:ContinuityOfDifferentialUid}. The implication (c)$\Rightarrow$(d) is trivial.
\end{proof}

\begin{rem}\label{rem:VerticallyBroadHolderNonSiToglie}
	Notice that in \cite[Example~4.5.1]{Koz15}, in the setting of \cref{example:Engel}, the author constructs a function $\widetilde\phi\colon\widetilde U\subseteq \mathbb W\to\mathbb L$ such that $D^{\phi}\phi=-1$ in the broad* sense but $\widetilde\phi$ is not vertically broad* h\"older, see \cref{def:vertically*Holder}. Taking \cref{thm:MainTheorem} into account, this means that $\widetilde\phi$ cannot be {\rm UID}. This also means that, in general, in \cref{thm:MainTheorem}, one cannot drop the assumption on the vertically broad* h\"older regularity in arbitrary Carnot groups. We will show that this is possible for step 2 Carnot groups, in \cref{sec:Step2}. For more examples related to this topic, we refer the reader to \cite[Section~4.5]{Koz15}.
\end{rem}

\section{Some applications}\label{sec:Applications}
Let us begin with an observation that will motivate the first part of this section. Consider the first Heisenberg group $\mathbb H^1$ with Lie algebra $\mathfrak g\coloneqq {\rm span}\{X,Y\}\oplus{\rm span}\{Z\}$, and the only nontrivial relation $[X,Y]=Z$. By a direct application of the Baker-Campbell-Hausdorff formula one gets $\exp X\exp Y\exp (-X)\exp (-Y)=\exp Z$. 
By exploiting this formula, in \cref{prop:BroadImpliesH12Heisenberg} below, we give an alternative proof, in the case of $\mathbb H^n$ with $n\geq 2$, of \cite[Theorem 3.2]{BSC10b}.  
 The argument we use is different because we prove that being a broad* solution with a continuous datum implies being $1/2$-little H\"older continuous along vertical coordinates (see also \cref{rem:OnlyVerticalCoordinates}) that is actually simpler than proving 1/2-little H\"older continuity in all the coordinates as in \cite[Theorem~3.2]{BSC10b}. Nevertheless, this is sufficient for applying \cref{prop:DPhiPhi=wBroadImpliesUid}.  Similarly, in \cref{prop:ReObtain}, we obtain the same result of \cite[Theorem~1.2]{BSC10b}, by making use of  \cref{prop:BroadImpliesH12Heisenberg} and \cref{prop:DPhiPhi=wBroadImpliesUid}. We remark that our argument is different by the one used in \cite{BSC10b} but, on the contrary, it does not work for $n=1$.

Recently, in \cite[Proposition~4.10]{Cor19}, the author has proved a generalization of \cite[Theorem~3.2]{BSC10b} that holds for arbitrary complemented subgroups in $\mathbb H^n$, and this statement is one of the key step in order to get the main theorem \cite[Theorem~1.4]{Cor19}. By using our reasoning, we are also able to recover \cite[Proposition~4.10]{Cor19}, and thus \cite[Theorem~1.4]{Cor19}, in the case $\mathbb H^n=\mathbb W\cdot \mathbb L$ with $\mathbb L$ horizontal and $k$-dimensional, with $k<n$. However, our argument does not apply to the remaining case $k=n$, while the argument in the reference does.

 Even if these results already appeared in the literature, we think it is worth writing them down with these different proofs. Indeed,  the proof of \cref{prop:BroadImpliesH12Heisenberg} provides a useful ``toolkit'' for the forthcoming theorems: in fact, a similar idea to the one exploited in \cref{prop:BroadImpliesH12Heisenberg} will be used in \cref{big3.3.12} to prove the analogous version of \cref{prop:BroadImpliesH12Heisenberg} in the setting of free Carnot groups of step 2, when $\mathbb L$ is one-dimensional. This will be the key step  to obtain the analogous \cref{prop:ReObtain} for Carnot groups of step 2 with $\mathbb L$ one-dimensional. We refer the reader to the introduction of \cref{sec:Step2}.

In \cref{exa:UidInEngel} below we give a class of nontrivial examples of ${\rm UID}$ functions in the Engel group. We build them by making use of  \cref{thm:MainTheorem}. This class of examples is inspired by \cite[Eq.\ (3.1)]{BV10}. Moreover, a slight modification of \cite[Eq.\ (3.1)]{BV10}, gives rise to functions whose intrinsic graphs are both of class $C^1_{\rm H}$ and of class $C^1$, but they possess a characteristic point, see \cref{rem:0ischaracteristic}. We thank R. Serapioni for having discussed this example with us.

\subsection{A different proof of the propagation of broad* regularity in $\mathbb H^n$, with $n\geq 2$}

\begin{prop}\label{prop:BroadImpliesH12Heisenberg}
	Let $\mathbb W$ and $\mathbb L$ be two complementary subgroups of $\mathbb H^n$, with $n\geq2\, ($see $\cref{example:Heisenberg})$ such that $\mathbb L$ is horizontal and $k$-dimensional with $k< n$, and let $\widetilde\phi\colon\widetilde U \subseteq \mathbb W\to\mathbb L$ be a continuous function, with $\widetilde U$ open. 
	
	Then there exists $(X_1,\dots,X_{2n+1})$ an adapted basis of the Lie algebra such that $\mathbb L=\exp({\rm span}\{X_1,\dots,X_k\})$, $\mathbb W=\exp({\rm span}\{X_{k+1},\dots,X_{2n+1}\})$, and such that the only nonvanishing bracket relations are given by $[X_i,X_{n+i}]=X_{2n+1}$ for every $1\leq i\leq n$.
	
	Moreover, if there exists a continuous function $\omega\colon U\subseteq \mathbb R^{2n+1-k}\to\mathbb R^{k\times (2n-k)}$ such that $D^{\phi}\phi=\omega$ in the broad* sense, then  $\widetilde\phi$ is vertically broad* holder. Namely, since we are in $\mathbb H^n$ (see \cref{rem:OnlyVerticalCoordinates}), for every $a_0 \in U$ there exists a neighborhood $U_{a_0}\Subset U$ of $a_0$ such that
	\begin{equation}\label{eqn:ThesisBroad*}
	\lim_{\varrho\to 0}\left(\sup\left\{\frac{|\phi(\xi,z_1)-\phi(\xi,z_2)|}{|z_1-z_2|^{1/2}}: (\xi,z_1)\in U_{a_0}, (\xi,z_2)\in U_{a_0}, 0<|z_1-z_2|<\varrho\right\}\right)=0,
	\end{equation}
	where we mean that $\xi\in  U_{a_0}\cap\{z=0\}\subseteq  \R^{2n-k}$ and $z_1,z_2\in \R$ are the $(2n+1)^{th}$ coordinates of $(\xi, z_1)$ and $(\xi,z_2)$ when seen as elements of $\mathbb H^n$. 
\end{prop} 
\begin{proof}
	The existence of the basis as in the first part of the statement comes from  \cite[Lemma 3.26]{FSSC07}. Fix $a_0\in U$ and find $\delta>0$ and sufficiently small neighborhoods $V'_{a_0} \Subset V_{a_0}\Subset U$ of $a_0$ such that, for every $a\in V'_{a_0}$, and every $j=k+1,\dots,2n$, we have the existence of integral curves $E^{\phi}_j(a)\colon[-\delta,\delta]\to V_{a_0}$ satisfying the conditions of \cref{defbroad*}. Denote by $\beta\colon[0,+\infty)\to[0,+\infty)$ a modulus of uniform continuity for $\omega$ on $\overline{V_{a_0}}$. 
	
	We fix a neighborhood $U_{a_0}\Subset V'_{a_0}$ of $a_0$, $\varrho>0$ and points $(\xi,z_1),(\xi, z_2)\in U_{a_0}$ such that $|z_1-z_2|<\varrho$. The set $U_{a_0}$ has to be chosen small enough: this will be clear during the proof. We are going to prove that, for every sufficiently small $\varrho$, we have
	\begin{equation}\label{eqn:FinalToObtain}
	\frac{|\phi(\xi,z_1)-\phi(\xi,z_2)|}{|z_1-z_2|^{1/2}} \leq 2k\beta(\alpha(\varrho)),
	\end{equation}
	for a continuous function $\alpha$ with $\alpha(0)=0$ that only depends on the norms of $\phi$ and $\omega$ on $V_{a_0}$. From this fact \eqref{eqn:ThesisBroad*} would follow, concluding the proof.
	
	Recall that, by \eqref{eqn:ProjectedVectorFieldsHn1} and \textbf{since $k<n$}, one has $D^{\phi}_{k+1}=X_{k+1}$ and $D^{\phi}_{n+k+1}=X_{n+k+1}$. Assume without loss of generality that $z_2>z_1$, and set $t_0\coloneqq (z_2-z_1)^{1/2}$. We exploit the relation $[D^{\phi}_{k+1},D^{\phi}_{n+k+1}]=X_{2n+1}$ and the Baker-Campbell-Hausdorff formula to conclude that we can join $(\xi,z_1)$ and $(\xi,z_2)$ by means of a concatenation of integral curves of $D^{\phi}_{k+1}=X_{k+1}$ and $D^{\phi}_{n+k+1}=X_{n+k+1}$. In particular 
	\begin{equation}\label{eqn:ConcatenationOfCurve2}
	\begin{split}
	a\coloneqq (\xi,z_1)&\to_{E^{\phi}_{k+1}(a)} \, a_1\coloneqq E^{\phi}_{k+1}(a,t_0) \to_{E^{\phi}_{n+k+1}(a_1)} \, a_2\coloneqq E^{\phi}_{n+k+1}(a_1,t_0) \\
	&\to_{E^{\phi}_{k+1}(a_2)} \, a_3\coloneqq E^{\phi}_{k+1}(a_2,-t_0) \to_{E^{\phi}_{n+k+1}(a_3)} \,  a_4\coloneqq (\xi,z_2)= E^{\phi}_{n+k+1}(a_3,-t_0), 
	\end{split}
	\end{equation}
	and if $\varrho$ is sufficiently small with respect to $\delta$ (for example $\varrho<\delta^2$) we know that the integral lines in \eqref{eqn:ConcatenationOfCurve2} are defined in $[-t_0,t_0]$ and all the points $a_1, a_2, a_3\in V'_{a_0}$. Notice that it is precisely here that we have to take $U_{a_0}$ small enough to guarantee that the points defined in \eqref{eqn:ConcatenationOfCurve2} are in $V'_{a_0}$. This can be done since we are taking the concatenation of integral curves living up to a time whose norm is bounded above by $t_0<\varrho^{1/2}$.
	
	We set $\phi=:(\phi^{(1)},\dots,\phi^{(k)})$ in coordinates. From \cref{defbroad*} we get that, for every $a\in V'_{a_0}$, we have
	\begin{equation}\label{eqn:USELGRANGE}
	\frac{\de}{\de \! s}{|_{s=t}}\phi^{(\ell)}(E^{\phi}_{j}(a,s))=\omega_{\ell j}(E^{\phi}_{j}(a,t)), \qquad \forall j=k+1,\dots,2n+1,\quad \forall \ell=1,\dots,k, \quad \forall t\in [-\delta,\delta].
	\end{equation} 
	Thus, by using the points defined in \eqref{eqn:ConcatenationOfCurve2}, Lagrange's theorem and the triangle inequality, we get
	\begin{equation}\label{eqn:ComplEstimate}
	\begin{split}
	\frac{|\phi(\xi,z_1)-\phi(\xi,z_2)|}{|z_1-z_2|^{1/2}} &=
	 \frac{\left|\left(\phi(a)-\phi(a_1)\right)+\left(\phi(a_1)-\phi(a_2)\right)+\left(\phi(a_2)-\phi(a_3)\right)+\left(\phi(a_3)-\phi(a_4)\right)\right|}{|z_1-z_2|^{1/2}} \\
	 &\leq \sum_{\ell=1}^k \frac{1}{|z_1-z_2|^{1/2}}\cdot  \Big|\left(\phi^{(\ell)}(a)-\phi^{(\ell)}(a_1)\right)+\left(\phi^{(\ell)}(a_1)-\phi^{(\ell)}(a_2)\right)+ \\
	 &+\left(\phi^{(\ell)}(a_2)-\phi^{(\ell)}(a_3)\right)+\left(\phi^{(\ell)}(a_3)-\phi^{(\ell)}(a_4)\right)\Big|\\
	&= \sum_{\ell=1}^k \frac{|-t_0\omega_{\ell,k+1}(b_{1,\ell})-t_0\omega_{\ell,n+k+1}(b_{2,\ell})+t_0\omega_{\ell,k+1}(b_{3,\ell})+t_0\omega_{\ell,n+k+1}(b_{4,\ell})|}{t_0} \\
	&\leq \sum_{\ell=1}^k \left(|\omega_{\ell,k+1}(b_{3,\ell})-\omega_{\ell,k+1}(b_{1,\ell})|+|\omega_{\ell,n+k+1}(b_{4,\ell})-\omega_{\ell,n+k+1}(b_{2,\ell})|\right) \\
	&\leq \sum_{\ell=1}^k\left(\beta(|b_{3,\ell}-b_{1,\ell}|)+\beta(|b_{4,\ell}-b_{2,\ell}|)\right),
	\end{split}
	\end{equation}
	where, for every $\ell=1,\dots,k$, the points $b_{1,\ell},b_{2,\ell},b_{3,\ell},b_{4,\ell}$ are respectively chosen according to Lagrange's Theorem, that can be applied in view of \eqref{eqn:USELGRANGE}, on the images of the integral curves $E^{\phi}_{k+1}(a)$, $E^{\phi}_{n+k+1}(a_1)$, $E^{\phi}_{k+1}(a_2)$, $E^{\phi}_{n+k+1}(a_3)$ used in \eqref{eqn:ConcatenationOfCurve2}. By simple estimates relying on the triangle inequality, we get a constant $C_0>0$ only depending on $V_{a_0}$ such that for every $\ell=1,\dots,k$, the estimates $|b_{3,\ell}-b_{1,\ell}|\leq C_0t_0$, and $|b_{4,\ell}-b_{2,\ell}|\leq C_0t_0$ hold. Since $t_0=|z_1-z_2|^{1/2}$, and $|z_1-z_2|< \varrho$, from \eqref{eqn:ComplEstimate} we thus get \eqref{eqn:FinalToObtain} with $\alpha(\varrho)\coloneqq C_0\varrho^{1/2}$. 
\end{proof}
\begin{prop}\label{prop:ReObtain}
	In the setting of \cref{prop:BroadImpliesH12Heisenberg}, it holds 
	$$
	D^{\phi}\phi=\omega \quad \mbox{{\rm in the broad* sense}} \Rightarrow \widetilde{\phi} \in{\rm UID}(\widetilde U,\mathbb W;\mathbb L).
	$$
\end{prop}
\begin{proof}
	It is a direct consequence of \cref{prop:BroadImpliesH12Heisenberg} and \cref{prop:DPhiPhi=wBroadImpliesUid}. Indeed, taking into account \eqref{eqn:ProjectedVectorFieldsHn2} of \cref{example:Heisenberg}, the integral curves of $D^{\phi}_{2n+1}$ are vertical lines and, therefore, condition \eqref{eqn:ThesisBroad*} of \cref{prop:BroadImpliesH12Heisenberg} implies that $\phi$ is vertically broad* h\"older.
\end{proof}

\subsection{Examples of uniformly intrinsically differentiable functions}
We show a class of nontrivial examples of ${\rm UID}$ functions in the Engel group, see \cref{example:Engel}.
\begin{exa}\label{exa:UidInEngel}
	Consider the Engel group $\mathbb E$, with the splitting $\mathbb E=\mathbb W\cdot\mathbb L$ described in \cref{example:Engel}. We show that the function 
	\begin{equation}
	\phi_{\alpha}(0,x_2,x_3,x_4)\coloneqq x_4^{\alpha}\chi_{\{x_4\geq 0\}}(0,x_2,x_3,x_4),
	\end{equation}
	produces $\widetilde\phi_\alpha\in {\rm UID}(\mathbb W;\mathbb L)$ for any $\alpha> 1/3$. We first claim that $D^{\phi_{\alpha}}_{X_2}\phi_{\alpha}=\frac{\alpha}{2}x_4^{3\alpha-1}\chi_{\{x_4\geq 0\}}$ in the broad* sense. 
	Indeed, for any $\eps\in(0,1)$, the functions \[(\phi_{\alpha})_{\eps}\coloneqq \eps^{1/3}\chi_{\{x_4<0\}}+ (x_4^{3\alpha}+\eps)^{1/3}\chi_{\{x_4\geq 0\}}\] are globally $C^1$ and 
	\[
	\lim_{\eps\to0}(\phi_{\alpha})_{\eps}=\phi_{\alpha} \quad \mbox{ in } L^\infty_{\rm loc}(\R^3), \quad \text{ and }\quad D^{(\phi_{\alpha})_{\eps}}_{X_2}(\phi_{\alpha})_{\eps}=\frac{\alpha}{2}x_4^{3\alpha-1}\chi_{\{x_4\geq 0\}}, \quad \forall \eps\in(0,1),
	\]
	where the last equality comes from the particular form of $D_{X_2}^{\phi}$ in \eqref{eqn:ProjectedVectorFieldsEngel}.
	By applying \cref{prop4.12}, we get the claim. 
	
	We claim now that $\phi$ is vertically broad* holder, see \cref{def:vertically*Holder}. Indeed, since the integral curves of $D_{X_4}^{\phi_{\alpha}}$ are the vertical lines along direction $x_4$, see \eqref{eqn:ProjectedVectorFieldsEngel}, and since $\alpha>1/3$, we get that $\phi_{\alpha}$ is locally uniformly $1/3$-little H\"older continuous along these curves. We are left to prove that, locally around any $ a \coloneqq (0, x_2, x_3, x_4)$, condition \eqref{eqn:vertcicallybroad*holder} is satisfied for the integral curves of $D_{X_3}^{\phi_{\alpha}}$, whose expression is in \eqref{eqn:ProjectedVectorFieldsEngel}. 
	
	According to the sign of $x_4$ we identify three families of integral curves of the vector field $D^{\phi_{\alpha}}_{X_3}$ starting from $a$. If $x_4<0$, the only integral curve of $D_{X_3}^{\phi_{\alpha}}$ starting from $a$ is $\gamma(t)= (0,x_2,x_3+t, x_4)$, and it is well-defined for every time $t$. If $ x_4=0$, we have that an integral curve of $D^{\phi_{\alpha}}_{X_3}$, existing for all times $t$, starting from $ a$ is given by \[\gamma(t)=\left(0,x_2, x_3+t, (1-\alpha)^{1/(1-\alpha)}t^{1/(1-\alpha)}\chi_{\{t\geq 0\}}(t)\right),\] while, if $ x_4>0$, we have that an integral curve of $D^{\phi_{\alpha}}_{X_3}$ starting from $ a$ is  given by \[\gamma(t)=\left(0, x_2, x_3+t, (1-\alpha)^{1/(1-\alpha)}\left(t+\frac{( x_4)^{1-\alpha}}{1-\alpha}\right)^{1/(1-\alpha)}\right),\] defined for $t\in \left(-\frac{ x_4^{1-\alpha}}{1-\alpha},+\infty\right)$. By exploiting this explicit choice of integral curves, we notice that $\gamma(t)$, in the three cases, is constant or 
	it is of order $t^{1/(1-\alpha)}$ on the last component. Thus, we get that $\phi_{\alpha}$ is 1/2-little H\"older continuous along these integral curves, because $\alpha/(1-\alpha)>1/2$ for any $\alpha>1/3$, and this happens locally uniformly. From this, we conclude that $\phi_{\alpha}$ is vertically broad* holder. 
	Now we conclude by applying (d)$\Rightarrow$(a) of \cref{thm:MainTheorem}.
	
	Notice that, if $\alpha=1/3$, the function $\phi_{\alpha}$ is not UID in every neighborhood of the origin. Indeed, by \cref{thm:MainTheorem}, if  $\phi_{\alpha}$ is UID, then it is vertically broad* h\"older and so  1/3-little H\"older continuous along the integral curves of the vector field $D^{\phi_{\alpha}}_{X_4}$. This means that $\phi_{\alpha}$ should be  $1/3$-little H\"older continuous in the last coordinate, but this is not true when $\alpha=1/3$. 
	
	We remark here that, in the case of the Heisenberg group $\mathbb H^1$, a slight modification of this types of examples gives rise to a $C^1_{\rm H}$-hypersurface, that is also $C^1$ Euclidean, but it has 0 as a characteristic point, see \cref{rem:0ischaracteristic}.
\end{exa}
\begin{rem}\label{rem:0ischaracteristic}
Consider the first Heisenberg group $\mathbb H^1$ identified with $\mathbb R^3$ by means of exponential coordinates, and consider the splitting $\mathbb H^1=\mathbb W\cdot\mathbb L$ described in \cref{example:Heisenberg}, with $\mathbb L$ one-dimensional. Define $\phi\colon\mathbb W\equiv \mathbb R^2\to \sL\equiv\R$ by setting $\phi(0,x_2,x_3)\coloneqq {\rm sgn}(x_3)|x_3|^{2/3}$. 
	
	Since $\phi ^2(0,x_2,x_3)=|x_3|^{4/3}\in C^1$, by \cite[Corollary 5.11]{ASCV06}, then $\widetilde\phi\in{\rm UID}(\mathbb W;\mathbb L)$ and, consequently, its graph is a $C^1_{\rm H}$-hypersurface. Moreover, in coordinates, one has
	\begin{equation}\label{eqn:Parametrized}
	(0,x_2,x_3)\cdot(\phi(0,x_2,x_3),0,0)=\left({\rm sgn}(x_3)|x_3|^{2/3},x_2,x_3-\frac{1}{2}{\rm sgn}(x_3)x_2|x_3|^{2/3}\right).
	\end{equation}
	The surface $\Sigma$ parametrized by \eqref{eqn:Parametrized} is the union of two surfaces $\Sigma_1,\Sigma_2$ given by 
	\begin{equation}
	\begin{split}
	\Sigma_1&\coloneqq \left\{(x_1,x_2,x_3)\in \R^3:x_1\geq 0, x_3-x_1^{3/2}+\frac{1}{2}x_1x_2=0\right\}, \\ \Sigma_2&\coloneqq \left\{(x_1,x_2,x_3)\in \R^3:x_1\leq 0, x_3+(-x_1)^{3/2}+\frac{1}{2}x_1x_2=0\right\},
	\end{split}
	\end{equation}
	that are glued along the $x_2$-axis $\{(x_1,x_2,x_3)\in \R^3: x_1=x_3=0\}$. We thus get that, for any $x\in {\rm int}(\Sigma_1)$, one has $T_x\Sigma_1=\left(\frac{1}{2}x_2-\frac{3}{2}x_1^{1/2},\frac{1}{2}x_1,1\right)^{\perp}$ with respect to the standard Euclidean scalar product, while for any $x\in {\rm int}(\Sigma_2)$, one has $T_x\Sigma_2=\left(\frac{1}{2}x_2-\frac{3}{2}(-x_1)^{1/2},\frac{1}{2}x_1,1\right)^{\perp}$ with respect to the standard Euclidean scalar product. From these explicit expressions, $\Sigma_1$ and $\Sigma_2$ glue together in $C^1$ regular way along the $x_2$-axis, and thus $\Sigma$ is also $C^1$-Euclidean surface. Moreover, from the previous expressions, we get that $T_0\Sigma = (0,0,1)^{\perp}=\{x_3=0\}=V_1$ and hence $0$ is a characteristic point for $\Sigma$.
	
	We remark here that a similar example appeared in \cite[Remark 3.8]{FSSC07}. In that case the intrinsic graph is $C^1_{\rm H}$ regular but not $C^1$ regular.
\end{rem}
\section{Carnot groups of step 2}\label{sec:Step2}

As already underlined in \cref{rem:VerticallyBroadHolderNonSiToglie}, a co-horizontal $C^1_{\mathrm H}$-surface cannot be always characterized only by its horizontal geometry. This is however possible inside Carnot groups of step 2. Indeed, in this section, we show that the assumption on the vertical broad* h\"older regularity of \cref{thm:MainTheorem} can be dropped when $\G$ is a Carnot group of step 2 and $\mathbb{L}$ is a horizontal one-dimensional subgroup of $\G$. In particular, as main result of this section, we get \cref{thm:MainTheorem.0}, which is \cref{thm:MainTheorem.0Intro2} in the introduction. 

We describe the strategy of the proof of \cref{thm:MainTheorem.0}. The key idea is to first show that the implication
\begin{equation}\label{eqn:Propagation}
D^{\phi}\phi=\omega \; \mbox{broad*}\; \Rightarrow\; \phi \; \mbox{vertically broad* h\"older}
\end{equation}
holds in free Carnot groups of step 2, for a continuous $\omega$. This is done in two main steps. First, we explicitly write the intrinsic vector fields $D^{\phi}$ and we notice that the  nonlinearity given by $\phi$ only shows up in one vertical coordinate for each vector field, see \eqref{operatoriproiettatiinF}. Second, by using the structure of the vector fields $D^{\phi}$, one can propagate the broad* regularity from the horizontal components  of $\phi$ to the little H\"older regularity along vertical components by using a geometric trick: if $[X,Y]=Z$, and $\phi$ is $C^1$ on the integral curves of $X$ and $Y$, we expect $\phi$ to be $1/2$-little H\"older continuous on the integral curves of $Z$. 

More precisely, the first step allows us to obtain the $1/2$-little H\"older regularity on the vertical coordinates affected by the nonlinearity by means of an adaptation of \cite[Theorem~5.8]{ASCV06}. The second step allows us to obtain the $1/2$-little H\"older  regularity on all the remaining vertical coordinates. Notice that the efforts made to prove (d)$\Rightarrow$(a) in \cref{thm:MainTheorem},  asking for just the regularity along the integral curves of $D^{\phi}$ (i.e., the vertically broad* regularity), are here payed back from a crucial simplification of this proof. Indeed, to conclude the proof of \cref{thm:MainTheorem.0} in the case of free Carnot groups of step 2, namely \cref{big3.3.12}, we just use \eqref{eqn:Propagation} and apply (d)$\Rightarrow$(a) of \cref{thm:MainTheorem}. This can be done after having noticed that on Carnot groups of step 2 the vertically broad* h\"older regularity reads as the locally 1/2-little H\"older continuity along vertical coordinates, see also \cref{rem:OnlyVerticalCoordinates}.

Finally, to conclude the proof of the difficult implication (e)$\Rightarrow$(a) of \cref{thm:MainTheorem.0}, we use \eqref{eqn:Propagation} together with the fact that, on a Carnot group $\mathbb G$ of step 2, the broad* condition lifts to the free Carnot group $\mathbb F$ of step 2 and of the same rank of $\mathbb G$, see \cref{broadin GimplicabroadinF}, while having a vertically broad* h\"older property is naturally transferred  from $\mathbb F$ to $\mathbb G$, see \cref{prop2}. The resulting strategy presents some similarity to \cite{LDPS19}. 

We point out that (a)$\Leftrightarrow$(c) of \cref{thm:MainTheorem.0} is a generalization to all step 2 Carnot groups of \cite[Theorem~5.1]{ASCV06} and (a)$\Leftrightarrow$(e) is a generalization of \cite[Theorem~1.2]{BSC10b}. We refer the reader to the introduction for a more detailed discussion on the literature.

In the current section, without loss of generality, we will always work in coordinates and there will be no distinction between $\square$ and $\widetilde\square$. See \cref{remIMPORT} for details on the identifications.
\subsection{Regularity results for broad* solutions in free Carnot groups of step 2}\label{freegroup}

Free-nilpotent Lie algebras can be defined as follows (see Definition 14.1.1 in \cite{BLU07}).
\begin{defi}[Free-nilpotent Lie algebras of step 2]
	Let $m\geq 2$ be integer. We say that $\mathfrak{f}_{m,2}$ is the {\em free-nilpotent Lie algebra of step 2 with $m$ generators} $X_1,\dots , X_m$ if the following facts hold.
	\begin{itemize}
		\item[(i)] $\mathfrak{f}_{m,2}$ is a Lie algebra generated by the elements $X_1,\dots , X_m$ (i.e., $\mathfrak{f}_{m,2}$ is the smallest Lie algebra containing $\{X_1,\dots , X_m\}$);
		\item[(ii)] $\mathfrak{f}_{m,2}$ is nilpotent of step $2$ (i.e., nested Lie brackets of length $3$ are always $0$);
		\item[(iii)] for every nilpotent Lie algebra $\mathfrak g$ of step $2$ and for every map $\Psi \colon \{X_1,\dots , X_m \}\to \mathfrak g$, there exists a unique homomorphism of Lie algebras $\overline{\Psi}: \mathfrak{f}_{m,2} \to \mathfrak g$ that extends $\Psi$. 
	\end{itemize}
\end{defi}

\begin{defi}[Free Carnot groups of step 2]\label{defi:gruppiliberidefinizione}
	A {\em free Carnot group of step 2} is a Carnot group whose Lie algebra is isomorphic to a
	free-nilpotent Lie algebra $\mathfrak{f}_{m,2}$  for some $m\geq 2$. In this case, the horizontal layer of
	the free Carnot group is isomorphic to the linear span of the generators of the Lie algebra $\mathfrak{f}_{m,2}$.
\end{defi}

\begin{rem}[Free Carnot groups of step 2 in exponential coordinates]\label{rem:coordinatestep2}
	We give an explicit representation of free Carnot groups of step $2$. Fix an integer $m\geq 2$ and denote by $n\coloneqq m+\frac{m(m-1)}{2}$. In $\R^n$ denote the coordinates by $x_j$, for $1\leq j\leq m$, and by $y_{\ell s}$, for
	$1 \leq s< \ell \leq m$. Let $\partial _j$ and $\partial _{\ell s}$ denote the standard basis vectors in this coordinate system.
	We define $n$ linearly independent vector fields on $\R^n$ by setting:
	\begin{equation}\label{eqn:VectorFieldsFree}
	\begin{aligned}
	X_j &\coloneqq \partial _j+ \frac 12 \sum_{ j<\ell\leq m } x_\ell \partial _{\ell j} -\frac 12 \sum_{ 1\leq \ell<j} x_\ell \partial _{j\ell}, \quad \mbox{ if } 1\leq j\leq m,\\
	Y_{\ell s}&\coloneqq \partial _{\ell s}, \hphantom{\frac 12 \sum_{ j<\ell\leq m } x_\ell \partial _{\ell j} -\frac 12 \sum_{ 1\leq \ell<j} x_\ell \partial _{j\ell}} \qquad  \mbox{ if } 1\leq s<\ell\leq m.
	\end{aligned}
	\end{equation}
	Let  $\mathbb{F} \coloneqq (\R^{m+\frac{m(m-1)}{2}}, \cdot)$ be the coordinate representation of the step 2 Carnot group with $m$ generators whose Lie algebra is generated by the vector fields in \eqref{eqn:VectorFieldsFree}. Then $\mathbb F$ is free and its Carnot structure is given by
	\begin{equation*}
	V_1\coloneqq\mbox{span} \{ X_j \,:\, 1\leq j\leq m\}\quad \mbox{and} \quad V_2\coloneqq\mbox{span} \{ Y_{\ell s} \,:\, 1\leq s<\ell \leq m\}.
	\end{equation*}
	Moreover, for any $p$ and $q\in \mathbb F$, the product $p\cdot q$  is given by the Baker-Campbell-Hausdorff formula, and yields
	\begin{equation*}
	\begin{aligned}
	(p\cdot q)_j &=p_j+q_j,\quad \quad \qquad \qquad \qquad \,\,\,\,\, \text{ if } 1\leq j\leq m,\\
	(p\cdot q)_{\ell s} &=p_{\ell s}+q_{\ell s} +\frac 1 2 (p_\ell q_s-q_\ell p_s), \quad \text{ if } 1\leq s<\ell\leq m.\\
	\end{aligned}
	\end{equation*}
\end{rem}
\noindent It is easily verified that for $1\leq s<\ell\leq m$ and $1\leq j\leq m$, one has
\begin{equation}\label{commutatorifree}
[X_\ell , X_s] = Y_{\ell s} \quad \mbox{and} \quad  [X_j , Y_{\ell s}] = 0.
\end{equation}

\begin{rem}\label{remIMPORT} 
	Let $\mathbb{F}$ be a free Carnot groups of step $2$ and rank $m$.
	To keep the notation simpler,
	throughout this section, we identify without loss of generality $\mathbb F$ with $\R^n$ with $n\coloneqq m+\frac{m(m+1)}{2}$ by means of the coordinates described in \cref{rem:coordinatestep2}. Moreover, given two complementary subgroups $\W$ and $\mathbb L$ with $\mathbb L$ horizontal and one-dimensional, we identify them in the following way:
	\begin{equation}\label{5.2.0}
	\begin{aligned}
	\mathbb L & \coloneqq\left\{ (x_1,\dots , x_n)\in \R^n : x_2=\dots=x_n=0 \right\},\\
	\W & \coloneqq\left\{ (x_1,\dots, x_n)\in\R^n : x_1=0 \right\}.
	\end{aligned}
	\end{equation}
	Therefore, there are na\-tu\-ral identifications between $ \R^{n-1}$ and $\W$ and  between $\R$ and $\sL$. We stress that, given an arbitrary free Carnot group of step 2 and complementary subgroups $\sW$ and $\sL$, with $\sL$ horizontal and one-dimensional, we can always choose coordinates satisfying \eqref{5.2.0} and such that the identifications of \cref{rem:coordinatestep2} are satisfied.
\end{rem}

\begin{rem}[Projected vector fields on free Carnot groups of step 2]\label{derivateintrinsecheperguppiliberi}
	Let $\mathbb{F}$ be a free Carnot group of step $2$ represented in coordinates as in \cref{rem:coordinatestep2} and let $\W$ and $\mathbb L$ be complementary subgroups of $\mathbb{F}$ as in \eqref{5.2.0}. Given $V\subseteq \mathbb W$ an open set, and given a continuous map $\psi\colon V\subseteq\sW\to \sL$, according to \cref{example:Step2}, the projected vector fields are given by
	\begin{equation}\label{operatoriproiettatiinF}
	\begin{aligned}
	D_j^\psi  = \partial _{j}-\psi \partial _{j1} + \frac 1 2 \sum_{  j<\ell\leq m } x_\ell\partial _{\ell j} -  \frac 1 2 \sum_{1<s <j} x_s \partial _{js}=  {X_j}_{| V}-\psi  {Y_{j1}}_{|V},  \quad &\mbox{ for } j=2,\dots, m,\\
	D^\psi_{\ell s}  = \partial _{\ell s}  =  {Y_{\ell s}}_{\vert \W},  \hphantom{\frac 1 2 \sum_{  j<\ell\leq m } x_\ell\partial _{\ell j} -  \frac 1 2 \sum_{1<s <j} x_s \partial _{js}=  {X_j}_{\vert \W}-\psi  {Y_{j1}}_{|V}}\quad &\mbox{ for }  1\leq s<\ell\leq m.
	\end{aligned}
	\end{equation} 
	Then, for each $j=2,\dots,m$, every integral curve $\gamma_j\colon I\to \mathbb R^{n-1}$ of $D^{\phi}_j$ has vertical components $y\coloneqq(y_{\ell s})_{1\leq s<\ell\leq m}\colon I\to \mathbb R^{\frac{m(m-1)}2}$  satisfying the following equations
	\begin{equation*}
	\begin{aligned}
	\dot y _{j1} (t)&=-\psi (x_2,\dots ,x_{j-1}, x_j+t,x_{j+1},\dots ,x_m, y(t)), \\
	\dot y _{\ell j} (t)&= \frac 12 x_\ell , \quad \quad  \qquad \mbox{ if }  j<\ell\leq m,\\
	\dot y _{js} (t)&=  -\frac 12 x_s, \quad  \,\qquad \mbox{ if }  1<s <j, \\
	\dot y_{\ell s} (t)&= 0,\quad \qquad  \quad \,\quad  \mbox{ otherwise,} 
	\end{aligned}
	\end{equation*}
where the horizontal components of $\gamma_j(0)$ are $(x_1,\dots, x_m)$.
\end{rem}
We now prove that, given an open set $V\subseteq \sW$ and a continuous function $\psi\colon V\to \sL$, a broad* solution of $D^\psi \psi =\omega $ for some continuous $\omega$ is also vertically broad* h\"older (see \cref{def:vertically*Holder}). In particular, since the vector field $D^\psi_{\ell s}$ satisfies $D^\psi_{\ell s}=\partial _{{\ell s}}$ for every $1\leq s<\ell\leq m$ (see \eqref{operatoriproiettatiinF}), then \eqref{eqn:vertcicallybroad*holder} is equivalent to the following condition: for every $a_0$ in $V$, there exists a neighborhood $V'$ of $a_0$ with $V'\Subset  V$ such that for every  $\ell$ and $s$ with $1\leq  s< \ell\leq m$, one has
\begin{equation}\label{eq:equazioneprimadella66}
\lim_{ \varrho  \to0} \left(\sup\left\{\frac{|\psi(\xi,\eta)-\psi(\xi, y)|}{|\eta_{ \ell  s}-  y_{ l s}|^{1/2}}\right\}\right) =0,
\end{equation}
where the supremum is taken on all the couples $(\xi,\eta), (\xi,y)\in V'$ such that $\eta_{k\tau}=y_{k\tau}$ for any $(k,\tau)\neq (\ell,s)$ and $0<|\eta_{\ell s}-y_{\ell s}|\leq \varrho$.

In the first part of the proof  of \cref{big3.3.12} below, we use techniques that are similar to the ones exploited in \cite[Theorems 5.8 and 5.9]{ASCV06} in the context of Heisenberg groups. The third step of the proof is new. 

\begin{theorem}\label{big3.3.12}
	Let $\W$ and $\mathbb{L}$ be two complementary subgroups of the free Carnot group $\mathbb{F}$ of step 2 with $\mathbb L$ horizontal and one-dimensional. Let $ V \subseteq \W$ be open and $\psi \colon{V} \to \mathbb L$ be a continuous function and assume that $\psi$ is a broad* solution of the system $D^{\psi } \psi =\omega$ in $V$, for some $\omega\in C(V;\mathbb R^{m-1})$, with respect to the basis chosen in \cref{rem:coordinatestep2} and \cref{remIMPORT}.
	Then $\psi$ is vertically broad* h\"older.
\end{theorem}

\begin{proof}
	Working in the coordinates described in \cref{rem:coordinatestep2} and \cref{remIMPORT}, we can assume without loss of generality that $\sW$ and $\sL$ are defined as in \eqref{5.2.0}. If $m=2$, then $\mathbb{F}=\mathbb{H}^1$, and therefore the statement would follow from \cite[Theorem~1.2]{BSC10b}. 
	We therefore assume that $m>2$. 

	We prove the following stronger fact, from which \eqref{eq:equazioneprimadella66} follows. For each $a_0\in V$ there are sufficiently small neighborhoods $I\Subset  I'\Subset V$ of $a_0$ such that, for every $1\leq s<\ell\leq m$, one can find a continuous and increasing function  $\alpha _{ \ell s} \colon(0,+\infty )\to [0,+\infty )$ only depending on $I'$, $\|\psi \|_{L^\infty (I')}$, $\| \omega \|_{L^\infty (I')}$ and on the modulus of continuity of $\omega$ on $I'$ with the property that 
	\begin{equation}\label{limitealpha.2}
	\lim_{\varrho \to 0}\alpha  _{ \ell  s}(\varrho )=0,
	\end{equation}
	and 
	\begin{equation}\label{alpha88.2}
	\begin{split}
	\frac{|\psi(\xi,\eta)-\psi(\xi, y)|}{|\eta_{ \ell  s}-  y_{ \ell  s}|^{1/2}}\leq \alpha  _{  \ell  s} (\varrho ),
	\end{split}
	\end{equation}
	for every $(\xi,\eta), (\xi, y)\in I$ such that  $\eta_{ k\tau} =y_{ k\tau}$ for every $(k,\tau)\ne ( \ell, s)$ and $0< |y_{\ell  s}-  \eta_{ \ell s}|\leq \varrho$.
	
	Fix $a_0\in V$. Since $\psi $ is a broad* solution of $D^{\psi } \psi =\omega$ in $V$, there exist $0<\delta _2<\delta  _1$ and a family of maps 
	\begin{equation*}
	E_j^\psi\colon\overline{B(a_0, \delta _2) }\times[-\delta _2, \delta _2]  \to   \overline{B(a_0, \delta _1) },
	\end{equation*}
	for $j=2,\dots ,m$,  such that the conditions of \cref{defbroad*} are satisfied.
	Define $I'\coloneqq \overline{B(a_0, \delta _1) }$ and $I\coloneqq \overline{B(a_0, \delta _2) },$ and set
	$ M_1\coloneqq\| \psi \|_{L^\infty (I' )}$. Let also $\beta $ be an increasing modulus of uniform continuity of $\omega$ on $I' $. We are going to prove \eqref{alpha88.2} with $\alpha _{ \ell s}$ defined by 
	\begin{equation}\label{Ndefialfa0.2}
	\alpha _{ \ell s}(\varrho ) \coloneqq 
	\begin{cases}
	3 \delta (  \varrho ),   &  \mbox{ if } (\ell, s)=(j,1) \mbox{ and } j=2,\dots, m,\\
	G_{ \ell s} (\varrho ), & \mbox{ otherwise,}
	\end{cases}
	\end{equation}
	where $G_{ \ell s}$ will be determined later in \eqref{defiGls} and 
	\begin{equation}\label{eq:definizionedelta}
	\delta ( \varrho)\coloneqq\max \left\{  \varrho^{1/4}, \sqrt{\beta(C  \varrho^{1/4})} \right\},
	\end{equation}
	for some constant $C\coloneqq C(j) >0$ such that \begin{equation}\label{eq:costanteC} 
	 |\eta_{j1}- y_{j1}| + 2(|\xi|+M_1) |\eta_{j1}- y_{j1}|^{1/4} \leq C |\eta_{j1}- y_{j1}|^{1/4}  
	\end{equation} 
	for any $(\xi,\eta), (\xi, y) \in  I'$ with $\eta_{k \tau}=  y_{k\tau}$ for every couple $(k,\tau)\neq(j,1)$.

	\noindent \textbf{First step.} If $a=(x,y)\in I$, by using \eqref{operatoriproiettatiinF}, we have that for any $2\leq j\leq m$ and any $t \in [-\delta_2,\delta_2]$, it holds
	\begin{equation}\label{eq:sistemadiCauchyy}
	\begin{aligned}
	E_j^\psi(a,t) &=(x_2,\dots ,x_{j-1}, x_j+t,x_{j+1},\dots ,x_m, y(t)), \qquad \text{where}\\
	y_{\ell s}(t) &= 
	\begin{cases}
	y_{j1}- \int_{0}^t \psi (E_j^\psi(a,r)) \, \de \! r, & \mbox{if } (\ell,s)=(j,1),\\
	\frac{1}{2} t x_\ell  +y_{\ell j},   & \mbox{if }s=j,\text{ and }  j<\ell\leq m,\\
	-\frac{1}{2}tx_s   +y_{j s}, &\mbox{if } \ell=j,\text{ and }  1<s <j, \\
	y_{\ell s}, &  \mbox{otherwise,}
	\end{cases}
	\end{aligned}
	\end{equation}
	and consequently, since $E^{\psi}_j$ are the maps provided by \cref{defbroad*}, 
	 $t\mapsto y_{j1}(t)$ is a solution of the Cauchy problem
	\begin{equation*}\label{prcauchy.2}
	\left\{
	\begin{array}{l}
	\ddot y_{j1}(t)=\displaystyle\frac{ \de }{ \de \! t} \Big[ -\psi (E_j^\psi(a,t))\, \Big] = - \omega_j(E_j^\psi(a,t)),\qquad   t\in [-\delta _2,\delta _2],\\
	\\
	y_{j1}(0)=y_{j1}, \\
	\\
	\dot y_{j1}(0) = - \psi (a).
	\end{array}
	\right.
	\end{equation*}
	As a consequence of \eqref{eq:sistemadiCauchyy} one gets 
	\begin{equation}\label{eq:stimaderivata}
	\begin{aligned}
	\max_{ r \in [-|t|, |t|] } | \dot y (r)| & \leq  \frac{1}{2}  \left|   \sum_{  j<\ell\leq m } x_\ell \right| +  \frac{1}{2}\left|\sum_{ 1<s<j} x_s \right| +\max_{ r  \in [-|t|, |t|]} \left|  \psi (E_j^\psi(a,r)) \right|  \leq C_1(|x|+M_1),\\
	\end{aligned}
	\end{equation} 
	for every $t\in [-\delta_2,\delta_2]$, where the constant $C_1>0$ only depends on the topological dimension of $\mathbb{F}$.
	
	\noindent{\bf Second step.}  Fix $j=2,\dots , m$ and assume $a=(\xi,\eta), \hat a=(\xi, y) \in I $ with $\eta_{k\tau}= y_{k\tau}$ for all couples $(k,\tau)\neq(j,1)$. We will need to possibly shrink $I$ in a way that will be clear throughout the proof.
	We aim to show that
	\begin{equation}\label{finaleCC.2}
	\frac{|\psi (a)-\psi  (\hat a)|}{|\eta_{j1}- y_{j1}|^{1/2}  }  \leq \alpha _{j1} (|\eta_{j1}- y_{j1}| ),
	\end{equation}   
	where, according to \eqref{Ndefialfa0.2}, $\alpha _{j1}(\varrho) = 3\delta(\varrho) $. This would imply \eqref{alpha88.2} for $( \ell, s )=(j,1)$.
	
	Set $\delta \coloneqq\delta( |\eta_{j1}- y_{j1}| )$ and suppose \eqref{finaleCC.2} is not true, namely
	\begin{equation}\label{finaleCCassurdo}
	\frac{|\psi (a)-\psi  (\hat a)|}{|\eta_{j1}- y_{j1}|^{1/2} }  > 3\delta .
	\end{equation}   
	Let $E_j^\psi(a, \cdot )$ and $E_j^\psi( \hat a,\cdot) $ be the integral curves of $D^\psi_j$ given by the broad* condition. By the first step they satisfy 
	\[
	E_j^\psi(a,t)=(\xi_2,\dots, \xi_{j-1},\xi_j+t,\xi_{j+1},\dots ,\xi_m,\eta(t)), \] and
	\[
	E_j^\psi(\hat a,t)=(\xi_2,\dots ,\xi_{j-1},\xi_j+t,\xi_{j+1},\dots ,\xi_m, y(t)).
	\] 
	We claim we can find $t^*\in [-\delta _2,\delta_2]$ such that $\eta_{j1}(t^*)= y_{j1} (t^* )$, with $\psi (E_j^\psi(a,t^*) ) \ne  \psi (E_j^\psi(\hat a,t^*))$. This will lead to a contradiction, by the fact that the equality $\eta_{j1}(t^*)= y_{j1} (t^* )$ would also imply $\eta(t^*)=y(t^*)$.
	
	Without loss of generality, \textbf{assume that the initial data satisfy $\eta_{j1}> y_{j1}$}.  By the first step of this proof, for every $t\in [-\delta _2,\delta _2]$, one has
	\[
	\begin{aligned}
	\eta_{j1}(t)- y_{j1}(t) - (\eta_{j1}- y_{j1})
	&=\int _{0}^t \left[ \dot \eta_{j1}(0) -\dot y_{j1} (0)+\int _{0}^{r'} (\ddot \eta_{j1} (r) - \ddot { y }_{j1} (r))\,  \de \! r   \right] \de \! r'  \\
	& = - t(\psi (a)-\psi (\hat a))- \int _{0}^t\int _{0}^{r'} \left( \omega_j (E_j^\psi(a,r)) -\omega_j (E_j^\psi(\hat a,r))\right)\,  \de \! r   \de \! r' .\\
	\end{aligned}
	\]
 Using \eqref{eq:stimaderivata} and the fundamental theorem of Calculus, one gets a constant $C_2>0$ only depending on $C_1$ such that
	\begin{equation}\label{eq240318}
	\begin{aligned}
	|E_j^\psi(a,r) - E_j^\psi(\hat a,r)| & \leq  \vert \eta-y\vert+ |r|\left(\max_{r  \in [-|t|, |t|]} | \dot \eta (r)| + \max_{ r \in [-|t|, |t|] } | \dot {y} (r)| \right)\\
	& \leq \vert\eta-y\vert   + |t|\left(\max_{ r  \in [-|t|, |t|]} | \dot \eta(r)| + \max_{ r  \in [-|t|, |t|]} | \dot {y}(r)| \right)\\
	& \leq  C_2\left(|\eta_{j1}-y_{j1}|+2|t|(|\xi|+M_1 )\right),
	\end{aligned}
	\end{equation} 
	for every $r \in [-|t|, |t|]$ and $t\in [-\delta_2,\delta_2]$. Hence, we obtain that
	\begin{equation}\label{5.24}
	\begin{aligned}
	\eta_{j1}(t)- y_{j1}(t) - (\eta_{j1}- y_{j1})& \leq  - t(\psi (a)-\psi (\hat a))+ t^2 \max_{r  \in [-|t|, |t|] }\beta \big(\vert E_j^\psi(a,r) - E_j^\psi(\hat a,r)\vert\big)  \\
	& \leq - t(\psi (a)-\psi (\hat a))+
	t^2 \beta \left(  C_2(|\eta_{j1}- y_{j1}|  +2|t|(|\xi|+M_1 )) \right),
	\end{aligned}
	\end{equation} 
	for every $t\in [-\delta_2,\delta_2]$. Up to redefining $\beta$, we can replace without loss of generality $\beta(C_2\varrho)$ with $\beta(\varrho)$. Now, by \eqref{finaleCCassurdo} we know that
	\begin{equation}\label{phi1}
	- |\psi (a)-\psi (\hat a)|  < -3 \delta  |\eta_{j1}- y_{j1}|^{1/2}.
	\end{equation}  
	Without loss of generality, up to restricting $I$, we can assume that 
	\begin{equation}\label{eq:restrizioneI}
	\delta_2 \geq |y_{j1}-\eta_{j1}|^{1/4} \geq |y_{j1}-\eta_{j1}|^{1/2}/\delta,
	\end{equation}
	where the second inequality directly follows from \eqref{eq:definizionedelta}.
	 If $\psi(a)=\psi(\hat a)$, then \eqref{finaleCC.2} would be trivial. We study two cases: $\psi(a)-\psi(\hat a)<0$ or $\psi(a)-\psi (\hat a)>0$. If $\psi(a)-\psi(\hat a)<0$, set $t_0\coloneqq - \frac{|\eta_{j1}- y_{j1}|^{1/2}}{\delta }$ (otherwise we can choose $t_0\coloneqq\frac{|\eta_{j1}- y_{j1}|^{1/2} }{\delta }$) and evaluate $\eqref{5.24}$ in $t=t_0$. Combining it with \eqref{eq:costanteC}, \eqref{phi1}, \eqref{eq:restrizioneI}, the definition of $\delta$, and the fact that $\beta$ is increasing, we obtain (in both cases) 
	\begin{equation}\label{5.24bis}
	\begin{aligned}
	\eta_{j1}(t_0)- y_{j1}(t_0) & \leq \eta_{j1}- y_{j1} + |\eta_{j1}- y_{j1}|^{1/2} \frac{-|\psi (a)-\psi (\hat a)|}{\delta } +\\
	& \hspace{0,5 cm} + \frac{1}{\delta ^2}|\eta_{j1}- y_{j1}|  \beta  \left(  |\eta_{j1}- y_{j1}| + 2(|\xi|+M_1) \frac{|\eta_{j1}- y_{j1}|^{1/2} }{\delta} \right)    \\
	& \leq \eta_{j1}- y_{j1}  - 3 |\eta_{j1}- y_{j1}| + |\eta_{j1}- y_{j1}| \frac{ \beta  \left( C |\eta_{j1}- y_{j1}|^{1/4}\right)}{\delta ^2}    \\
	& \leq \eta_{j1}-y_{j1} - 3|\eta_{j1}- y_{j1}|+|\eta_{j1}- y_{j1}| \leq -|\eta_{j1}- y_{j1}|<0.
	\end{aligned}
	\end{equation} 
	If $ \psi (a)-\psi (\hat a)>0$  we can define  
	\begin{equation*}
	t^*\coloneqq \sup\{ \, r\in [0, \delta _2] \, :\,\eta_{j1}(s)- y_{j1}(s) >0,\, \forall s\in [0,r]\},
	\end{equation*}
	since the set $\{ \, r\in [0, \delta _2] \, :\,\eta_{j1}(s)- y_{j1}(s) >0,\, \forall s\in [0,r]\}$ is not empty; indeed, recall that we assumed without loss of generality that $\eta_{j1}>y_{j1}$ and therefore $\eta_{j1}(0)- y_{j1}(0)=\eta_{j1}- y_{j1}>0$. Moreover $0<t^*<t_0 \leq\delta _2$, and recalling that $\eta_{k\tau}=y_{k\tau}$ except for $(k,\tau)=(j,1)$, one has, by \eqref{eq:sistemadiCauchyy} that
	\begin{equation}\label{3.44big.2}
	\begin{aligned}
	\eta_{j1}(t^*) & = y_{j1}(t^*), \\
	\eta_{\ell j}(t^*)& = y_{\ell j}(t^*)= \frac{1}{2}t^* \xi_\ell  +\eta_{\ell j}, \quad \qquad \quad \mbox{ for }  j<\ell\leq m,\\
	\eta_{j s}(t^*) & = y_{j s}(t^*) =-\frac{1}{2}t^* \xi_s   +\eta_{j s}, \,\quad \quad \quad  \mbox{ for }  1<s <j, \\ \eta_{\ell s}(t^*)& =y_{\ell s}(t^*) =\eta_{\ell s},  \qquad  \qquad   \qquad \quad  \,\, \,\,  \mbox{ otherwise.}
	\end{aligned}
	\end{equation}
	Hence, by definition of $E_j^\psi(a,\cdot )$ and $E_j^\psi(\hat a,\cdot)$ in \eqref{eq:sistemadiCauchyy}, one gets $E_j^\psi(a,t^*)=E_j^\psi(\hat a, t^*)$ and therefore
	\begin{equation}\label{legame3.2}
	\begin{aligned}
	\psi (E_j^\psi(a,t^*) ) =  \psi (E_j^\psi(\hat a,t^*)).
	\end{aligned}
	\end{equation}
	In the case $ \psi (a)-\psi (\hat a)<0$, we consider $t_0= -\frac{|\eta_{j1}- y_{j1}|^{1/2}}{\delta }$ and define $t^*\coloneqq\inf \{  r\in [-\delta _2,0]: \eta_{j1}(s)- y_{j1}(s)>0,\, \forall s\in [r,0] \}$. Then  $-\delta _2\leq t_0<t^*<0$ and, also in this case, \eqref{3.44big.2} and \eqref{legame3.2} are satisfied. 
	
	We now show that this leads to a contradiction. In case $\psi (a)-\psi (\hat a)>0$, by using the properties of $E^\psi_j$, \eqref{eq240318}, \eqref{phi1}, \eqref{eq:costanteC}, the definition of $\beta$  and the fact that $t^*< t_0$, we deduce
	\begin{equation*}
	\begin{aligned}
	-(\psi (E_j^\psi(a,t^*) )  -  &\psi (E_j^\psi(\hat a,t^*)))= -(\psi (a)-\psi (\hat a))- \int _{0}^{t^*} \left( \omega_j (E_j^\psi(a,r)) -\omega_j (E_j^\psi(\hat a,r))\right)\,  \de  \! r  \\
	& \leq   -3 (\eta_{j1}-y_{j1})^{1/2}\delta + |t^*| \,\max_{r  \in [0, t^*] }  \beta \left(\vert E_j^\psi(a,r) - E_j^\psi(\hat a,r) \vert \right)  \\
	& \leq  -3 (\eta_{j1}- y_{j1})^{1/2}\delta +  |t^*| \beta  \left( |\eta_{j1}- y_{j1}|+ 2|t^*|(|\xi|+M_1) \right) \\
	& \leq -3 (\eta_{j1}- y_{j1})^{1/2}\delta  + |t^*| \, \beta  \left( |\eta_{j1}- y_{j1}|+ 2(|\xi|+M_1) \frac{|\eta_{j1}- y_{j1}|^{1/2}}{\delta} \right)    \\
	& \leq -3 (\eta_{j1}- y_{j1})^{1/2}\delta  +  (\eta_{j1}- y_{j1})^{1/2} \delta \, \frac{\beta  \left(C(\eta_{j1}- y_{j1})^{1/4} \right)}{\delta ^2}    \\
	& \leq (-3+1)(\eta_{j1}- y_{j1})^{1/2} \delta   <0.
	\end{aligned}
	\end{equation*}Similarly, if $\psi (a)-\psi (\hat a)<0$, then 
	\begin{equation*}
	\begin{aligned}
	\psi (E_j^\psi(a,t^*) )  -  \psi (E_j^\psi(\hat a,t^*))&= \psi (a)-\psi (\hat a)+ \int _{0}^{t^*} \left( \omega_j (E_j^\psi(a,r)) -\omega_j (E_j^\psi(\hat a,r))\right) \de  \! r  \\
	& \leq   -3 (\eta_{j1}- y_{j1})^{1/2}\delta + |t^*| \,\max_{r  \in [-t^*, 0] }  \beta \left(\vert E_j^\psi(a,r) - E_j^\psi(\hat a,r)\vert \right),  \\
	&<0.
	\end{aligned}
	\end{equation*}
	Hence, in both cases we have $\psi (E_j^\psi(a,t^*) ) \ne  \psi ( E_j^\psi(\hat a,t^*))$ that is in contradiction with \eqref{legame3.2}, so $\eqref{finaleCC.2}$ follows.
	
	\textbf{Third step.}
	Fix $ \ell, s$ with $1< s < \ell \leq m $,  denote by  $M_2\coloneqq\|\omega\|_{L^\infty (I' )}$ and define 
	\begin{equation}\label{defiGls}
	G_{\ell s}(\varrho)\coloneqq 2\sqrt {M_2} \alpha _{\ell 1} (4M_2 \varrho ) + 2\sqrt {M_2}  \alpha _{ s 1} (4M_2 \varrho )+2 \beta (C_0 \varrho^{1/2} ), 
	\end{equation}  
	where $\alpha_{\ell 1}$ and $\alpha_{s 1} $ are defined as in \eqref{Ndefialfa0.2} for $j=\ell$ and $j= s$, respectively, $\beta$ is an increasing modulus of uniform continuity of $\omega$ on $I' $ and $C_0>0$ is a suitable constant, only depending on the supremum norm of $\omega$ on $I'$, that will be determined later.
	 
	We want to show that 
	\begin{equation}\label{equVERIFICA.2}
	 \frac{|\psi(a)-\psi(\hat a)|}{|\eta_{\ell  s}-  y_{ \ell s}|^{1/2}}\leq G_{\ell s}(\varrho),
	\end{equation} 
	for every sufficiently small $\varrho>0$, every $a=(\xi,\eta), \hat a= (\xi, y)\in I$, such that $\eta_{ k \tau} = y_{ k\tau}$ for every $(k,\tau)\ne (\ell, s)$ and $0< |\eta_{\ell  s}-  y_{ \ell s}| \leq   \varrho$. Denote by $T_{\ell s}\coloneqq|\eta_{\ell s}- y_{\ell  s}|^{1/2}$. We will need to possibly shrink $I$ in a way that will be clear from the proof.

	\emph{Rough idea of the proof.} We build a concatenation of integral curves of the vector fields $D^\psi_{\ell}$ and $D^\psi_{s}$ that joins $\hat a$ and a suitable point $a_4$ that that can be connected to $a$ with two vertical lines on which we can use the result of the second step of this proof. We start from $\hat a$ and follow the integral curve of $D^\psi_{\ell}$ for a time $T_{\ell s}$ and we follow the integral curve of $D^\psi_{s}$ for the same time $T_{ \ell s}$; finally, we repeat the same procedure but for time $-T_{ \ell s}$. 
	At the end of this process, we obtain a point with three different vertical components with respect to $\hat a$: two increments are given by the non linear terms $-\psi \partial _{\ell 1}$ and $-\psi \partial _{ s1}  $ coming respectively from $D^\psi_{\ell}$ and from $D^\psi_{ s}$ and one increment is given by the commutator $[X_{\ell}, X_{s}]$ (which is $Y_{\ell s}$). In particular,  whenever $\eta_{\ell s}-  y_{\ell s}>0$, the $(\ell, s)$-component becomes equal to the $(\ell, s)$-component of $a$, that is $\eta_{\ell s}$. Vice versa, if $\eta_{\ell s}-y_{\ell s}<0$, one has to replace the times $\pm T_{ \ell s}$ with $\mp T_{\ell s}$.
	In the end, we complete the proof by using the estimate of the second step of this proof applied to $a_4$ and $a$. The desired estimates come by using Lagrange's Theorem.
	
	Assume $\eta_{\ell s} -y_{\ell s}>0$. Then we construct the following chain of points. 
	\begin{equation}\label{aggiustamentieq}
	\begin{split}
	\hat a &\to_{E_{\ell}^\psi (\hat a)}\, a_1\coloneqq E_{\ell}^\psi (\hat a, T_{\ell s})  \to_{E_{s}^\psi ( a_{1}) } \, a_2\coloneqq E_{ s}^\psi ( a_{1}, T_{\ell  s}) \\
	&\to_{E_{\ell}^\psi ( a_{2}) } \, a_3\coloneqq  E_{\ell}^\psi ( a_2 , -T_{\ell s}) \to_{E_{s}^\psi ( a_{3})} \,  a_4\coloneqq E_{ s}^\psi ( a_{3}, -T_{\ell  s}), \\
	\end{split}
	\end{equation}
	where we recall that $E_{\ell}^\psi$ and $ E_{s}^\psi $ are the integral curves of the vector fields $D^\psi_{\ell}$ and $D^\psi_{s}$, respectively, given by the fact that $D^{\psi } \psi =\omega$ in the broad* sense. In particular, $E_{\ell}^\psi $ and $ E_{s}^\psi $ satisfy \eqref{eq:sistemadiCauchyy}. 
	
	In case $\eta_{\ell s} - y_{\ell s} <0$ we repeat the same construction by replacing $\pm T_{\ell s}$ with $\mp T_{\ell s}$. In both cases, we have that $a_{4}\eqqcolon(\xi, y^{(a_{4})})$ is such that $y^{(a_{4})}_{\ell s} =\eta_{\ell s}$. Indeed, if $\eta_{\ell s} -y_{\ell s}>0$, then
	\begin{equation*}
	y^{(a_{4})}_{k\tau} = 
	\begin{cases}
	\eta_{\ell s}, & \mbox{if }  (k,\tau)=(\ell, s), \\
	\displaystyle\eta_{\ell 1}-\int_{0}^{T_{\ell s}} \psi ( E_{\ell}^\psi ( \hat a , t))  \de \!  t+ \int_{0}^{T_{\ell s}} \psi (E_{\ell}^\psi ( a_2 , t)) \de  \! t, & \mbox{if }  (k,\tau)=(\ell,1),\\
	\displaystyle\eta_{ s1} -\int_{0}^{T_{\ell s}} \psi (E_{ s}^\psi ( a_{1}, t))  \de  \! t+ \int_{0}^{ T_{\ell s}} \psi (E_{ s}^\psi ( a_{3}, t))  \de  \! t, &  \mbox{if }  (k,\tau)=(s,1),\\
	y_{k\tau}, &\mbox{otherwise.}
	\end{cases}
	\end{equation*}

	We can assume that $a_1,a_2, a_3, a_4\in I$. Indeed, this can be done because for a sufficiently small $\varrho$ only depending on the supremum norms of $\phi$ and $\omega$ on $I'$, we can possibly reduce $I$ to some $I_0$ so that all the curves as in \eqref{aggiustamentieq} starting in $I_0$, and living for times bounded above by $\varrho$, lie inside $I$.
	
	Let $a_5=(\xi,y^{(a_{5})})$ be a point that has the same components of $a$, except for position $(\ell, 1)$ for which $y^{(a_{5})}_{\ell1}=y^{(a_{4})}_{\ell1}$. As remarked above, we can assume without loss of generality that also $a_5 \in I$.  Moreover, we can estimate $|\psi(a)-\psi(\hat a)|$ as follows:
	\begin{equation}\label{stima2903.6}
	\begin{aligned}
	\left|\psi (a)-\psi (\hat a) \right|  & \leq   \left|\psi (a)-\psi (a_5) \right|   + \left|\psi (a_5)-\psi ( a_4) \right|  +   \left|\psi (a_4)-\psi ( \hat a)\right|. 
	\end{aligned}
	\end{equation}  
	We start by considering $|\psi(a)-\psi( a_5)|$. Evaluating \eqref{finaleCC.2} for $j=\ell$, we get that
	\begin{equation*}
	\begin{aligned}
	\left|\psi (a)-\psi (a_5) \right|  & \leq  |\eta_{\ell 1}-y^{(a_{5})}_{\ell 1}|^{1/2}  \alpha _{\ell 1} ( |\eta_{\ell 1}-y^{(a_{5})}_{\ell1}|), \\
	\end{aligned}
	\end{equation*}  
	 and we also notice that  
	\begin{equation*}
	|\eta_{\ell 1}-y^{(a_{5})}_{\ell 1}|= \left|\int_{0}^{ T_{\ell s}} \psi ( E_{\ell}^\psi ( \hat a , t))  \de \!  t  -\int_{0}^{ T_{\ell s}} \psi ( E_{\ell }^\psi ( a _2, t))  \de  \!  t \right|. 
	\end{equation*}
	Recalling that $M_2=\|\omega\|_{L^\infty(I')}$, we aim to show that
	\begin{equation}\label{stima2903.1}
	\begin{aligned}
	\left|\int_{0}^{ T_{\ell s}} \psi ( E_{\ell}^\psi ( \hat a , t))  \de  \! t  -\int_{0}^{ T_{\ell s}} \psi ( E_{\ell}^\psi ( a _2, t)) \de  \!  t \right|  \leq 4M_2 T_{\ell s}^2, \\
	\end{aligned}
	\end{equation}  
	that would imply
	\begin{equation}\label{stima2903.2}
	\begin{aligned}
	\left|\psi (a)-\psi (a_5) \right|  & \leq  2\sqrt {M_2} T_{\ell  s}  \alpha _{\ell 1} ( |\eta_{\ell 1}-y^{(a_{5})}_{\ell1}|) \leq  2\sqrt{M_2} T_{\ell s}  \alpha _{\ell 1} ( 4M_2 T_{\ell s}^2). \\
	\end{aligned}
	\end{equation}  
	We first observe that
	\begin{equation*}
	\begin{aligned}
	& \left|\int_{0}^{ T_{\ell  s}} \psi (E_{\ell}^\psi ( \hat a , t)) \de  \!  t  -\int_{0}^{ T_{\ell s}} \psi ( E_{\ell}^\psi ( a_2 , t)) \de  \!  t \right| \\& \quad = \left|\int_{0}^{ T_{\ell s}} \left(\psi ( E_{\ell}^\psi (\hat a , t)) -\psi (\hat a)\right)  \de \!   t  -\int_{0}^{ T_{\ell s}} \left(\psi (E_{\ell}^\psi ( a_2 , t)) -\psi (a_2)\right) \de  \!  t  +T_{\ell s} \left((\psi (\hat a)-\psi (a_2))\right) \right|  \\
	&\quad \leq \int_{0}^{ T_{\ell s}}  \left| \psi (E_{\ell}^\psi (\hat a , t)) -\psi (\hat a)\right| \de  \!  t  + \int_{0}^{ T_{\ell s}}  \left| \psi ( E_{\ell}^\psi ( a _2, t)) -\psi (a_2) \right| \de  \!  t  +T_{\ell s} |\psi (\hat a)-\psi (a_2)|. \\
	\end{aligned}
	\end{equation*}  
	Recalling that for every $t$ in the interval of definition of the curve $E^{\psi}_{\ell}$, one has \\ $\frac{\de }{\de \! s}_{|s=t} \psi (E_{\ell}^\psi ( A , s)) = \omega_{\ell} ( E_{\ell}^\psi ( A , t))$ for $A=\hat a, a_2$, by exploiting the fundamental theorem of Calculus the previous estimate then yields
	\begin{equation*}
	\begin{aligned}
	& \left|\int_{0}^{ T_{\ell s}} \psi ( E_{\ell}^\psi (\hat a , t)) \de  \!  t  -\int_{0}^{ T_{\ell s}} \psi ( E_{\ell}^\psi ( a_2 , t))  \de  \!  t \right| \leq 2M_2 T^2_{\ell s} +T_{\ell s} |\psi (\hat a)-\psi (a_2)|. \\
	\end{aligned}
	\end{equation*}  
	By Lagrange's Theorem one also gets
	\begin{equation*}
	\begin{aligned}
	|\psi (\hat a)-\psi (a_2)| \leq  |\psi (\hat a)-\psi (a_1)|+ |\psi (a_1)-\psi (a_2)| =T_{\ell s} |\omega_{\ell } (b^*_0)|+ T_{\ell s} |\omega_{s} (b^*_1)| \leq 2M_2T_{\ell s},   \\
	\end{aligned}
	\end{equation*} 
	where $b^*_0$ and $b^*_1$ are two points on $E_{\ell}^\psi ( \hat a , [0,T_{\ell s}])$ and $E_{s}^\psi ( a _1, [0,T_{\ell s}])$, respectively. Hence, combining together the last two estimates, one obtains \eqref{stima2903.1} and, consequently, also \eqref{stima2903.2} holds.

	Now we consider $|\psi(a_5)-\psi(a_4)|$. Since  $|a_5-a_4|= |y^{(a_{5})}_{s1}-y^{(a_{4})}_{ s1}|$, analogously to the previous case, one obtains 
	\begin{equation}\label{stima2903.3}
	\left|\psi (a_5)-\psi (a_4) \right|   \leq  2\sqrt{M_2} T_{\ell s}  \alpha _{ s 1} ( |y^{(a_{5})}_{s1}-y^{(a_{4})}_{s1}|)  \leq  2\sqrt{M_2} T_{\ell s}  \alpha _{s 1} ( 4{M_2} T_{\ell s}^2),
	\end{equation}
	where  $\alpha _{ s 1}$ is defined as in \eqref{Ndefialfa0.2} and
	\begin{equation*}
	|y^{(a_{5})}_{s1}-y^{(a_{4})}_{s1}| = \left|\int_{0}^{T_{\ell s}} \psi (E_{s}^\psi (a _1, t)) \de \!  t- \int_{0}^{ T_{\ell s}} \psi (E_{ s}^\psi ( a _3, t))  \de \!  t \right|. 
	\end{equation*}
	
	Eventually, we estimate $\psi (a_4)-\psi ( \hat a)$ in the following way:
	\begin{equation*}
	\begin{aligned}
	\left|\psi (a_4)-\psi ( \hat a)\right| & \leq    \left|(\psi (a_4)-\psi ( a_3))+ (\psi (a_3)-\psi ( a_2))+(\psi (a_2)-\psi ( a_1))+(\psi (a_1)-\psi ( \hat a)) \right|  \\
	&=| -T_{\ell s}  \omega_{ s } (a^*_3)-T_{\ell s}  \omega_{\ell } (a^*_2) + T_{\ell s}  \omega_{ s } (a^*_1)+T_{\ell  s}  \omega_{\ell } (a^*_0)| \\
	& \leq T_{\ell s} (|\omega_{ s } (a^*_1)-  \omega_{ s } (a^*_3)| +|\omega_{\ell } (a^*_2)-  \omega_{\ell } (a^*_0)|)\\
	& \leq T_{\ell  s}\left( \beta (|a^*_1- a^*_3|)+ \beta (|a^*_2- a^*_0|) \right),
	\end{aligned}
	\end{equation*}  
	where  $a^*_0 \in E_{\ell}^\psi (\hat a, [0,T_{\ell s}]), a^*_1\in E_{ s}^\psi (a_1, [0,T_{\ell s}]), a^*_2 \in E_{\ell}^\psi (a_2,[-T_{\ell s} , 0])$ and $a^*_3\in E_{ s}^\psi (a_3, [-T_{\ell s}, 0])$ are chosen to fulfill the conditions of Lagrange's Theorem. By simple estimates relying on the triangle inequality, we get a constant $C_0>0$, only depending on the supremum norm of $\omega$ on $I'$, such that $|a^*_1- a^*_3|\leq C_0T_{\ell s} $, and $|a^*_2- a^*_0|\leq C_0 T_{\ell s}$. Hence
	\begin{equation}\label{stima2903.4}
	\begin{aligned}
	\left|\psi (a_4)-\psi ( \hat a)\right|    & \leq T_{\ell s}\left( \beta (|a^*_1- a^*_3|)+ \beta (|a^*_2- a^*_0|) \right) \leq 2  T_{\ell s} \beta (C_0 T_{\ell s}). 
	\end{aligned}
	\end{equation}  
	
	By combining \eqref{stima2903.6} with \eqref{stima2903.2}, \eqref{stima2903.3} and \eqref{stima2903.4} and recalling that $T_{\ell s}= |\eta_{\ell s}- y_{ \ell  s}|^{1/2}$ and $|\eta_{\ell s}- y_{\ell s}|< \varrho$, we thus get \eqref{equVERIFICA.2}. Indeed, one has
	\[
	\begin{aligned}
	\frac{\psi(a)-\psi(\hat a)}{|\eta_{\ell s}-y_{\ell s}|^{1/2}}&\leq \frac{2\sqrt {M_2}T_{\ell s} \alpha_{\ell 1}(4M_2 T_{\ell s}^2)+2\sqrt {M_2}T_{\ell s} \alpha_{s 1}(4M_2 T_{\ell s}^2)+2T_{\ell s}\beta(C_0T_{\ell s})}{|\eta_{\ell s}- y_{\ell s}|^{1/2}}\\
	&= 2\sqrt{M_2}\Big(\alpha_{\ell 1}(4M_2|\eta_{\ell s}-y_{\ell s}|)+\alpha_{s 1}(4M_2|\eta_{\ell s}-y_{\ell s}|)\Big)+ 2\beta(C_0|\eta_{\ell s}- y_{\ell s}|^{1/2})\\
	&\leq G_{\ell s}(\varrho).
	\end{aligned}
	\]
	 Finally, since $G_{\ell s}$ is defined as sum of continuous maps that are $0$ at $0$, it follows that, if $\varrho \to 0$, then $G_{\ell  s} (\varrho)\to 0 $ for which we get \eqref{alpha88.2} also for all $( \ell, s )$ with $s\ne 1$ and the proof is complete.
\end{proof}

\subsection{Regularity results for broad* solutions in Carnot groups of step 2}\label{carnotgroupgeneral}
In this section we see how to generalize \cref{big3.3.12}, valid for free Carnot groups of step $2$, to any Carnot group of step 2. We adapt some techniques already exploited in \cite{LDPS19}. More precisely, in \cref{broadin GimplicabroadinF} we prove that the broad* condition lifts from $\mathbb G$ to the free group $\mathbb F$ with same step and rank of $\G$. In \cref{prop2} we show that the vertically broad* h\"older regularity on $\mathbb F$ implies the vertically broad* h\"older regularity on $\mathbb G$. These two facts will put us in a position to prove \cref{thm:MainTheorem.0} by exploiting \cref{thm:MainTheorem} and \cref{big3.3.12}.

We here introduce Carnot groups of step 2 and we refer the reader to \cite[Chapter~3]{BLU07}.   
We denote with $m$ the rank of $\G$ and we identify $\mathbb G$ with $(\R^{m+h}, \cdot )$. If $q\in \G$, we write $q=(x,y)$ meaning that $x\in \R^m$ and $y\in \R^h$. The group operation $\cdot$ between two elements $q=(x,y)$ and $q'=(x',y')$ is given by
\begin{equation}\label{5.1.0}
q\cdot q'= \left(x+ x',y+ y'-\frac 1 2\langle \mathcal{B} x,  x' \rangle \right),
\end{equation} 
 where $\langle \mathcal{B}x,x' \rangle \coloneqq (\langle \mathcal{B}^{(1)}x,x' \rangle, \dots , \langle \mathcal{B}^{(h)} x, x' \rangle)$ and $\mathcal B^{(i)}$ are linearly independent and skew-symmetric matrices in $\R^{m\times m}$, for $i=1,\dots, h$.

For any $i=1,\dots, h$ and any $j,\ell=1,\dots,m$, denote by $(\mathcal{B}^{(i)})_{j\ell}\eqqcolon (b_{j\ell}^{(i)})$, and define $m+h$ linearly independent left-invariant vector fields by setting
\begin{equation*}
\begin{aligned}
X'_j (p) & \coloneqq \partial _{x_j}  -\frac{1}{2 } \sum_ {i=1 }^{h} \sum_ {\ell=1 }^{m} b_ {j\ell}^{(i)} x_\ell  \,\partial _{y_i},  \quad \mbox{ for } j=1,\dots ,m,\\
Y'_i(p)  & \coloneqq \partial _{y_i }, \,\qquad \qquad \qquad \qquad \qquad \mbox{ for } i=1,\dots , h.
\end{aligned}
\end{equation*}  
 The ordered set $(X'_1,\dots,X'_m,Y'_1,\dots,Y'_h)$ is an adapted basis of the Lie algebra $\mathfrak g$ of $\G$. Using the skew-symmetry of $\mathcal B$, it easy to see that
\begin{equation}\label{commutatoripasso2}
[X'_j, X'_\ell]= \sum _{i=1}^h b_{j\ell}^{(i)} Y'_i,  \quad \mbox{and} \quad  [X'_j , Y'_{i}] = 0,\quad \forall j,\ell=1,\dots, m \;\text{ and }\; \forall i=1,\dots,h.
\end{equation}

\begin{rem}
	When $\G$ is a free Carnot group of step $2$ with coordinate representation defined as in \cref{rem:coordinatestep2}, we denote the matrices of the beginning of \cref{carnotgroupgeneral} with $\mathcal B^{(i)}\eqqcolon\mathcal{B}^{(\ell,s)}$ with $1\leq s< \ell\leq m$. The composition law $\eqref{5.1.0}$ also tells us that $\mathcal{B}^{(\ell,s)}$ has entry $1$ in  position $(\ell,s)$, $-1$ in  position $(s,\ell)$ and $0$ elsewhere. 
	
	Since the space of skew-symmetric $m$-dimensional matrices has dimension $\frac{m(m-1)}{2}$, in any Carnot group $\G$ of step 2, the dimensions $m$ of the horizontal layer and $h$ of the vertical layer are related by the inequality 
	\[
	h \leq \frac{m(m-1)}{2},
	\]
	and $\G$ is free if and only if $h=\frac{m(m-1)}{2}$.
\end{rem}

 \textbf{From now on $\G$ is a Carnot group of rank $m$ and step $2$, with the coordinate representation previously discussed, and $\mathbb{F}$ is the free Carnot group of step 2 and rank $m$ with the coordinate representation as in \cref{rem:coordinatestep2} and \cref{remIMPORT}}. We denote by $(X_1,\dots, X_m)$ a basis of the first layer of the Lie algebra of $\mathbb F$. By definition of free Lie algebra, there exists a Lie group surjective homomorphism $\pi \colon \mathbb F \to \G$ such that $\pi _*(X_\ell)=X'_\ell$ for any $\ell=1,\dots ,m$ (see e.g., \cite[Section~6]{LDPS19}). 

If we consider on $\mathbb{F}$ and $\G$ the Carnot-Carath\'eodory metrics $d_{\mathbb{F}}$ and $d_{\G}$ respectively, the map $\pi$  preserves the length of horizontal curves, so it is Lipschitz with
${\rm Lip}(\pi) = 1$. The following lemma is well-known. We refer the reader to \cite[Lemma 6.1]{LDPS19} for a proof.
\begin{lem}\label{lem6.1}
	For any $p\in \mathbb{F}$ and any $ q' \in \G$, there exists $ p' \in \pi^{-1}(q')$ such that \[d_{\mathbb{F}} (p, p') = d_{\G} (\pi(p), q').\]
\end{lem}

Recall that the dimension of $\mathbb F$ is $n=m+\frac{m(m-1)}{2}$. From the definition of $\pi $, we notice that it preserves the horizontal coordinates, namely, for any $(x,y)\in \R^{n}$, there exists $y^*\in \R^h$ such that
\begin{equation}\label{proiezione}
\pi(x,y)= (x,y^*).
\end{equation}

We denote by $\W_\G$ and $\mathbb L_\G$ two complementary subgroups of $\G$ with $\mathbb L_\G$ horizontal and one-dimensional. Similarly to \cref{remIMPORT}, by means of exponential coordinates we identify them with $\R^{m+h-1}$ and $\R$ by setting
\begin{equation}\label{5.2.0.1}
\begin{aligned}
\mathbb L_\G & \coloneqq\{ (x_1,0\dots , 0) \,:\, x_1\in\R \},\\
\W_\G & \coloneqq\{ (0,x_2, \dots,x_m,y_1,\dots, y_h) \,:\, x_i, y_k \in\R \mbox{ for }i=2,\dots, m,\, k=1,\dots h \}.
\end{aligned}
\end{equation}

\begin{rem}
	Let $ \G$  be a Carnot group of step 2 and $\W _\G$, $\mathbb L_\G$ be the complementary subgroups of $\mathbb{G}$ defined as in \eqref{5.2.0.1}. Then, according to \cref{example:Step2} the projected vector fields relative to a continuous $\phi\colon U\subseteq \sW_\G\to\sL_\G$, with $U$ open, are given by
	\begin{equation}\label{operatoriproiettatiinG}
	\begin{aligned}
	D_{j}^\phi & = \partial _{x_j}-  \sum_{ i=1 }^{h  } \left(  b_{j1}^{(i)} \phi  +\frac 1 2 \sum_{ k=2 }^{m  } x_k b_{jk}^{(i)}\right) \partial _{y_i} =  {X'_j}_{\vert U}-  \sum_{i=1}^hb_{j1}^{(i)}  \phi  {Y'_{i}}_{\vert U},  \quad \mbox{ for } j=2,\dots, m,\\
	D^\phi_{i} & = \partial_{y_i}  =  {Y'_{i}}_{\vert U},  \qquad \mbox{ for }  i=1,\dots, h.
	\end{aligned}
	\end{equation}
\end{rem}

Now let $\W_{ \mathbb{F}}$ and $\mathbb{L}_{ \mathbb{F}}$ be the complementary subgroups of $\mathbb{F}$ defined as in \eqref{5.2.0}.  Then $\pi_{|_{ \mathbb L _{ \mathbb{F}} }} \colon\mathbb L _{ \mathbb{F}}  \to \mathbb L_\G$ is an isomorphism and more precisely, with our identification, we can assume it is the identity, see \eqref{proiezione}.  Moreover, by \eqref{proiezione}, it follows that $\pi_{|_{ \mathbb W_{ \mathbb{F}} }} \colon\mathbb W_{ \mathbb{F}}  \to \mathbb W_\G$ is onto.

Since $\pi$ is a Lie group homomorphism, its differential is a Lie algebra homomorphism. Hence, for any $1\leq s<\ell\leq m$, one also has
\begin{equation*}
\pi_* (Y_{\ell s})=\pi_*  ([X_\ell, X_s])=[\pi_* (X_\ell),\pi_*  (X_s)]=[X'_\ell, X'_s] = \sum _{i=1}^h b_{\ell s}^{(i)} Y'_i,
\end{equation*}
where we used \eqref{commutatorifree}, \eqref{commutatoripasso2} and $\pi_*(X_j)=X'_j$. We can therefore write the following formula
\begin{equation}\label{definizioneesplicitaP}
\begin{split}
& \pi(x_1,\dots , x_m, y_{21},\dots , y_{m(m-1)})  =(x_1,\dots , x_m, y^*_1,\dots , y^*_h), \quad \text{where}\\
& y^*_i  =  \sum _{1\leq s<\ell \leq m}  b_{\ell s}^{(i)} y_{\ell s}, \quad \forall i=1,\dots, h.\\
\end{split}
\end{equation}

\begin{prop}\label{broadin GimplicabroadinF}
	Let $\mathbb{G}$ be a Carnot group of step 2 and let $\mathbb W_\G$ and $\mathbb L_\G$ be two complementary subgroups of $\mathbb G$, with $\mathbb L _\G$ horizontal and one-dimensional and choose coordinates such that \eqref{5.2.0.1} is satisfied. Let $\mathbb F$ be the free Carnot group of step 2, rank $m$ and let $\mathbb W_{\mathbb F}$ and $\mathbb L_{\mathbb F}$ be the complementary subgroups of $\mathbb F$ satisfying the identification \eqref{5.2.0}. \\
	Let $ U\subseteq \sW _\G$   be an open set and denote  by ${V} \subseteq\mathbb W_{\mathbb F}$ the open set defined by $  V\coloneqq\pi^{-1}( U)$. Let ${\phi}\colon{U} \to\mathbb L_\G$ be a continuous map and let $ \psi \colon{V} \to\mathbb L_{\mathbb F} $ be the map defined as 
	\[
	{\psi} \coloneqq \pi^{-1} \circ {\phi}  \circ \pi _{|_{ V}}.
	\]
	Assume there exists $\omega\in C(U; \R^{m-1})$ such that $D^\phi \phi=\omega$ in the broad* sense.
	\\
	Then $\psi $ is a broad* solution in $V$ of the system $ D^{{\psi}} {\psi} = \omega \circ \pi$.
\end{prop}

\begin{proof} 
	Fix $j=2,\dots , m$. By \eqref{operatoriproiettatiinF}, we have
	\begin{equation}\label{formula3}
	\begin{aligned}
	D_{X_j}^\psi & = \partial _{x_j}-\psi \partial _{j1} + \frac 1 2 \sum_{  j<l\leq m } x_\ell\partial _{\ell j} -  \frac 1 2 \sum_{1<s <j} x_s \partial _{js},\\
	\end{aligned}
	\end{equation}
	and by \eqref{operatoriproiettatiinG}, it follows
	\begin{equation}\label{formula4}
	\begin{aligned}
	D_{X'_j}^\phi & = \partial _{x_j}-  \sum_{ i=1 }^{h  } \left(  b_{j1}^{(i)} \phi  +\frac 1 2 \sum_{ k=2 }^{m  } x_k b_{jk}^{(i)}\right) \partial _{y_i}.\\
	\end{aligned}
	\end{equation}

	Let $ a = (0, x_2,\dots , x_m, \eta_{21},\dots , \eta_{m(m-1)}) \in V$, and denote with $ b\coloneqq  \pi (a)$ the point in $\mathbb W_\G$ with coordinates $ b =(0, x_2,\dots,  x_m, y^*_1,\dots , y^*_h)$. Let $\delta >0$ and let $ \gamma_j \colon[-\delta , \delta ] \to U$ be an arbitrary integral curve of $D_{X'_j}^\phi$ starting from $ b$ such that 
	\begin{equation}\label{eqn:PhiBroadStep2}
	\phi (\gamma _j (t))-\phi ( b)=\int_0^t \omega_{j}(\gamma_j(s))\de\!s ,  \quad \forall t\in [- \delta, \delta ].
	\end{equation}
	Recall that by \eqref{definizioneesplicitaP}, there is an explicit relation between the coordinates of $ a$ and the coordinates of $b$. We are going to prove that we can lift $\gamma_j$ to an integral curve $\zeta_j$ of $D^{\psi}_{X_j}$ starting from $ a$, defined on $[-\delta,\delta]$ and with values in $V$, that satisfies
	 \begin{equation*}
	\psi (\zeta _j (t))-\psi ( a)=\int_0^t (\omega_{j} \circ \pi ) (\zeta _j (s))\,\de\!s ,  \quad \forall t\in [- \delta, \delta ].
	\end{equation*}
	
	Indeed, let $\zeta _j \colon[-\delta , \delta ] \to \mathbb W_{\mathbb F}$ be defined as
	\begin{equation}\label{definizione di eta}
	\begin{split}
	\zeta _j(t) &=(0, x_2,\dots , x_{j-1}, x_j+t,  x_{j+1},\dots , x_m,\eta_{21}(t),\dots, \eta_{m(m-1)}(t)), \quad \text{where}\\
	\eta_{\ell s}(t)& = \begin{cases}
	 \displaystyle \eta_{j1}- \int_{0}^t \phi (\gamma _j(r))  \de \! r,  & \mbox{for } (\ell,s)=(j,1),\\
	\displaystyle\frac{1}{2} t x_\ell  + \eta_{\ell j},   &\mbox{for }s=j\,\text{ and }\,  j<\ell\leq m,\\
	\displaystyle-\frac{1}{2}t x_s   + \eta_{js},  &\mbox{for } \ell=j\,\text{ and }\, 1<s <j, \\
	 \eta_{\ell s},& \mbox{otherwise,}
	\end{cases}
	\end{split}
	\end{equation}
	for $t\in [-\delta , \delta]$. By definition, one immediately gets $ \zeta _j(0) = a$. We are left to prove the following facts. 
	\begin{itemize}
		\item[(i)] $ \pi  \circ \zeta _j = \gamma _j$;
		\item[(ii)]  $\zeta _j \colon[-\delta , \delta ] \to V$ is  an integral curve of $ D_{X_j}^\psi$;
		\item[(iii)] For every $t\in [-\delta, \delta]$ one has \begin{equation*}
		\psi (\zeta _j (t))-\psi ( a)=\int_0^t (\omega_{j} \circ\pi) (\zeta _j (s))\de\!s.
		\end{equation*}
	\end{itemize}

	(i). By \eqref{formula4}, and the fact that $\gamma _j (0)= b=(0, x_2,\dots, x_m, y^*_1,\dots , y^*_h)$, we can explicitly write $ \gamma _j$, exploiting the fact that it is an integral curve of $D^{\phi}_{X_j'}$, and get
	\begin{equation}\label{defitildegammaj}
	\begin{split}
	\gamma _j(t) &=(0, x_2,\dots , x_{j-1}, x_j+t, x_{j+1},\dots ,  x_m, y^*_1(t),\dots, y^*_h(t)), \quad \text{where}\\
	y^*_{i}(t)& =  y^*_{i}  -\frac 1 2 t\left( \sum_{ k=2 }^{m  } x_k b_{jk}^{(i)} \right) -b_{j1}^{(i)}  \int_{0}^t \phi (\gamma _j(r)) \de \! r,  \qquad \forall i=1,\dots, h,\text{ and }\,\forall t\in[-\delta , \delta].\\
	\end{split}
	\end{equation}
	Using the explicit formula of $\pi$ given in \eqref{definizioneesplicitaP}, we obtain that
	\begin{equation}\label{defigamma.1}
	\begin{split}
	& \pi (\zeta _j(t )) =(0, x_2,\dots , x_{j-1}, x_j+t,  x_{j+1},\dots , x_m, \eta^*_{ 1}(t),\dots, \eta^*_{ h}(t)), \quad \text{where}\\
	& \eta^*_{i} (t) =  \sum _{1\leq s<\ell \leq m}  b_{\ell s}^{(i)} \eta_{\ell s}(t ), \qquad \forall i=1,\dots, h,\text{ and }\,\forall t\in[-\delta , \delta].\\
	\end{split}
	\end{equation}
	In particular, by \eqref{definizione di eta}, we have that
	\begin{equation*}
	\begin{split}
	\eta^*_{i} (t)&  = b_{j1}^{(i)}\left( \eta_{j1}- \int_{0}^t \phi (\gamma _j(r))\de \! r\right) +  \sum _{j<\ell\leq m } b_{\ell j}^{(i)}\left( \frac{1}{2} t x_\ell  +  \eta_{\ell j} \right)   \\
	&\quad +  \sum _{1<s <j } b_{js}^{(i)} \left( -\frac{1}{2}t x_s   + \eta_{j s}\right) +\sum _{\substack{1\leq s<\ell \leq m \\ s\ne j, \ell\ne j}}  b_{\ell s}^{(i)} \eta_{\ell s},
	\end{split}
	\end{equation*}
	for every $i=1,\dots, h$ and $t\in [-\delta , \delta]$.  Using again \eqref{definizioneesplicitaP}, we also notice that $ y^*_i  =  \sum _{1\leq s<\ell \leq m}  b_{\ell s}^{(i)}\eta_{\ell s}$, and, by the skew-symmetry of $\mathcal{B}^{(i)}$, we deduce that
	\begin{equation*}
	\begin{split}
	\eta^*_i (t) &  = y^*_{i}  -\frac 1 2 t\left( \sum_{ k=2 }^{m  } x_k b_{jk}^{(i)} \right) -b_{j1}^{(i)}  \int_{0}^t \phi (\gamma _j(r)) \de \! r, \qquad \forall i=1,\dots, h \text{ and }\,\forall t\in [-\delta , \delta].  \\
	\end{split}
	\end{equation*}  
	Comparing \eqref{defitildegammaj}, \eqref{defigamma.1} and the last equality, we get $\pi \circ \zeta _j = \gamma _j$, as desired. 
	
	(ii). We now check that $\zeta_j$ is an integral curve of $ D_{X_j}^\psi$. Recalling \eqref{formula3}, and by the definition of the components of $\zeta_j$ in \eqref{definizione di eta}, it suffices to check its $(j,1)$-coordinate, since for all the others is trivial. Observe that, from (i), $(\phi \circ \gamma _j) (t) = (\phi \circ \pi\circ \zeta_j )(t)= (\psi \circ \zeta _j)(t), $ for every $t\in [-\delta , \delta]$, and so 
	\[\eta_{j1} (t ) = \eta_{j1}- \int_{0}^t \phi (\gamma _j(r)) \de \! r= \eta_{j1}- \int_{0}^t \psi (\zeta _j(r)) \, \de \! r, \qquad \forall t\in [-\delta , \delta],\] as desired. Notice that we have used that $\pi_{|_{\mathbb L_{\mathbb F}}}$ identifies $\mathbb L_{\mathbb F}$ with $\mathbb L_{\mathbb G}$.
	
	(iii). We first notice that, since ${U} \supseteq \gamma _j ([-\delta , \delta ]) =( \pi  \circ \zeta _j)  ([-\delta , \delta ])$ and $ V=\pi^{-1}( U)$, then also  $\zeta _j ([-\delta , \delta ])\subseteq {V}$.\\
	Using  $\pi  \circ \zeta _j = \gamma _j$ and the fact that $\gamma_j$ satisfies \eqref{eqn:PhiBroadStep2}, we finally obtain
	\begin{equation}\label{equaintegrale}
	\psi (\zeta _j (t))-\psi (a)=\int_0^t (\omega_{j} \circ \pi ) (\zeta _j (s))\de\!s ,  \qquad  \forall t\in [- \delta, \delta ].
	\end{equation}

	We thus showed that every integral curve of $D^\phi_{X'_j}$ satisfying \eqref{eqn:PhiBroadStep2} can be lifted to an integral curve of $D^{\psi}_{X_j}$ satisfying \eqref{equaintegrale}, with a procedure that is patently local. Thus, by using \cref{defbroad*}, it follows that, if $D^{\phi}\phi=\omega$ holds in the broad* sense on $U$, then $D^{\psi}\psi=\omega\circ \pi$ holds in the broad* sense on $V$.\qedhere
\end{proof}

\begin{prop}\label{prop2}
	Let $\mathbb{G}$ be a Carnot group of step 2 and let $\mathbb W_\G$ and $\mathbb L_\G$ be two complementary subgroups of $\mathbb G$, with $\mathbb L _\G$ horizontal and one-dimensional and choose coordinates such that \eqref{5.2.0.1} is satisfied. Let $\mathbb F$ be the free Carnot group of step 2, rank $m$ and let $\mathbb W_{\mathbb F}$ and $\mathbb L_{\mathbb F}$ be the complementary subgroups of $\mathbb F$ satisfying the identification \eqref{5.2.0}. Let $ U\subseteq \sW _\G$   be an open set and denote  by ${V} \subseteq\mathbb W_{\mathbb F}$ the open set defined by $  V\coloneqq\pi^{-1}( U)$. Let ${\phi}\colon{U} \to\mathbb L_\G$ be a continuous map and let $ \psi \colon{V} \to\mathbb L_{\mathbb F} $ be the map defined as 
	\[
	{\psi} \coloneqq \pi^{-1} \circ {\phi}  \circ \pi _{|_{ V}}.
	\]
	
	\noindent Then, if ${\psi} $ is vertically broad* h\"older, also $\phi$ is vertically broad* h\"older. 
\end{prop}

\begin{proof}  
	We observe that in the case we are dealing with, the vertically broad* h\"older condition is equivalent to the locally 1/2-little H\"older continuity along vertical coordinates, see the discussion before the statement of \cref{big3.3.12}, that holds verbatim for arbitrary Carnot groups of step 2. Fix $b_0\in U$ and let $a_0 \in\pi^{-1}(b_0)$. Since $\psi$ is vertically broad* h\"older, there exist two neighborhoods $V'$ and $V''$ of $a_0$ with $V'\Subset  V ''\Subset  V$ and an increasing map $\alpha \colon(0,+\infty )\to [0,+\infty )$ only depending on $V ''$,  such that $\lim_{\varrho \to0}\alpha(\varrho)=0$ and
	\begin{equation}\label{eq:disalfa}
	\frac{|\psi(\xi,\eta)-\psi(\xi,y)|}{|\eta_{\ell s}-y_{\ell s}|^{1/2}}\leq \alpha(\varrho)
	\end{equation}
	for every $(\ell,s)$ such that $1\leq s<\ell\leq m$, every sufficiently small $\varrho>0$ and every $(\xi,\eta), (\xi,y)\in V'$ with $\eta_{k\tau}=y_{k\tau}$ for every $(k,\tau)\neq (\ell,s)$ and $0<|\eta_{\ell s}-y_{\ell s}|\leq \varrho$.
	
	Set $U' \coloneqq\pi(V')$ so that $b_0\in U'$ and clearly $U'\Subset U.$
	We aim to prove that there exists an increasing function  $\beta \colon(0,+\infty )\to [0,+\infty )$ only depending on $U''\coloneqq \pi(V'')$ such that $\lim_{\varrho \to 0}\beta(\varrho )=0$ and 
	\begin{equation}\label{betafin.4}
	\frac{|\phi(\xi, \eta)-\phi(\xi, y)|}{|\eta_{i}-  y_{i}|^{1/2}}\leq \beta (\varrho ),
	\end{equation}
	for every $i=1,\dots , h$, every sufficiently small $\varrho>0$ and every $(\xi,\eta), (\xi, y)\in U'$ with $y_{k} =y_{k}$ for every $k \ne i$ and $0< |\eta_{i}- y_{i}|\leq \varrho$.
	Fix $i=1,\dots ,h$ and $\varrho>0$ sufficiently small and consider $b=(\xi,\eta)$ $\hat b=(\xi, y)$ in $U'$ such that $\eta_{k} = y_{k}$ for all $k \ne i$ and $0<|y_i-\eta_i|\leq \varrho$.
	
	Applying \cref{lem6.1} to the points $b_0$ and $b$ 
	 we find $a= (\xi,\eta^*)\in \pi^{-1}(b)$ such that $d_\G(b_0, b)=d_{\mathbb F}(a_0,a)$ and, since $\pi$ is continuous, we can also assume, up to possibly reducing $U'$, that $a\in V'$. 
	 Applying again \cref{lem6.1} to the points $b$ and $ \hat b$, we find $\hat a = (\xi,y^{*})\in \pi^{-1}(\hat b)\cap V'$ such that $d_\G(b, \hat b)=d_{\mathbb F}(a,\hat a)$. Since the horizontal components of the points $b$ and $\hat b$ are equal and the norm induced by the distance $d_\G$ is equivalent to the  anisotropic norm on $ \mathbb{G}$, we have that $d_\G (b, \hat b) $ is equivalent to $| \eta_i -y_i |^{1/2}$. Similarly, notice that $a, \hat a$ have the same horizontal components and by the fact that the norm induced by $d_{ \mathbb{F}} $ is equivalent to the anisotropic norm on $ \mathbb{F}$, it follows that $|\eta^* - y^* |^{1/2}$ is equivalent to $ d_{ \mathbb{F}} (a, \hat a) $. In particular, we can find a geometric constant $C_1>0$ such that $|\eta^* -y^* |^{1/2} \leq C_1| \eta_i - y_i |^{1/2}$. 
	
	We can then make the following estimates:
	\begin{equation*}
	\begin{aligned}
	& \frac{|\phi (\xi,\eta)- \phi  (\xi,y)|}{| \eta_i- y _{i}|^{1/2}  } 
	= \frac{|\psi (\xi,\eta^*)- \psi (\xi,y^* )|}{|\eta^* - y^*|^{1/2} }\frac{| \eta^* -  y^*|^{1/2} }{| \eta_i - y_i|^{1/2} }\\
	& \leq C_1\left(\frac{|\psi (\xi,\eta^*)- \psi (\xi, y^*_{21},\eta^*_{31}, \dots , \eta^*_{m(m-1)})|}{|\eta^*_{21}- y^*_{21}|^{1/2} } +\dots\right.\\&\left.\qquad\qquad\dots+ \frac{| \psi (\xi, y^*_{21},\dots , y^*_{m(m-2)} , \eta^*_{m(m-1)}) - \psi (\xi,y^*)|}{|\eta^*_{m(m-1)}- y^*_{m(m-1)}|^{1/2} } \right)\\
	& \leq C_1\left(\alpha(|\eta^*_{21} - y^*_{21}|) + \dots + \alpha(|\eta^*_{m(m-1)} -y^*_{m(m-1)}|) \right),
	\end{aligned}
	\end{equation*}
	where we used \eqref{eq:disalfa} and assumed without loss of generality that all the considered points in the chain belong to $V'$.
	Observe that, for any $(\ell,s)$ such that $1\leq s<l\leq m$, one has $|\eta^*_{\ell s} -y^*_{\ell s}| \leq |\eta^* -y^* |\leq C_1^2| \eta_i - y_i | \leq C_1^2\varrho$, and then we can define
	\begin{equation*}
	\begin{aligned}
	\beta(\varrho)\coloneqq C_1 \frac{m(m-1)}{2}\alpha (C_1^2\varrho).
	\end{aligned}
	\end{equation*}
	The previous computations have shown that
	\begin{equation*}
	\frac{|\phi (\xi,\eta)- \phi  (\xi,y)|}{| \eta_i- y _{i}|^{1/2} }  \leq \beta(\varrho),
	\end{equation*}
	and hence \eqref{betafin.4} holds, completing the proof.
\end{proof}

\begin{theorem}\label{bigtheoremstep2}
	Let $\W$ and $\mathbb{L}$ be complementary subgroups of a Carnot group  $\mathbb{G}$ of step 2 with $\mathbb{L}$ horizontal and one-dimensional and choose coordinates such that \eqref{5.2.0.1} is satisfied. Let $ U\subseteq \sW$ be an open set and let $ \phi\colon  U \to \mathbb L$ and $\omega\colon U\to \R^{m-1}$ be two continuous functions such that $D^\phi \phi=\omega$ in the broad* sense.
	Then $\phi$ is vertically broad* h\"older.
\end{theorem}

\begin{proof}
	 Let $\mathbb F$ be the free Carnot group of step 2, rank $m$ and let $\mathbb W_{\mathbb F}$ and $\mathbb L_{\mathbb F}$ be the complementary subgroups of $\mathbb F$ satisfying the identification \eqref{5.2.0}. By \cref{broadin GimplicabroadinF}, we know that $\psi \coloneqq \pi^{-1}\circ \phi \circ \pi$ is a broad* solution of $ D^{{\psi}} {\psi} = \omega \circ \pi$ in $V=\pi^{-1}(U)$. Then, by \cref{big3.3.12}, $\psi$ is vertically broad* h\"older and finally, by using \cref{prop2}, we obtain the thesis.
\end{proof}

We state here some corollaries of the previous results.

\begin{coroll}\label{coroll:step2.2}
	Let $\W$ and $\mathbb{L}$ be complementary subgroups of a Carnot group  $\mathbb{G}$ of step 2 with $\mathbb{L}$ horizontal and one-dimensional and choose coordinates on $\G$ such that \eqref{5.2.0.1} is satisfied. Let $ U\subseteq \sW$ be an open set and let $ \phi\colon U \to \mathbb L$ and $\omega\colon U\to\R ^{m-1}$ be two continuous functions such that $D^\phi \phi=\omega$ in the broad* sense on $U$.
	Then $\phi\in {\rm UID}( U, \mathbb W; \mathbb L)$.
\end{coroll}

\begin{proof}
	 It is enough to combine \cref{prop:DPhiPhi=wBroadImpliesUid} and \cref{bigtheoremstep2}.
\end{proof}

\begin{coroll}
	Let $\W$ and $\mathbb{L}$ be complementary subgroups of a Carnot group  $\mathbb{G}$ of step 2 with $\mathbb{L}$ horizontal and one-dimensional and choose coordinates on $\G$ such that \eqref{5.2.0.1} is satisfied. Let $ U\subseteq \sW$ be an open set and let $ \phi\colon U\to \mathbb L$ and $\omega\colon U \to\R ^{m-1}$ be two continuous functions such that $D^\phi \phi=\omega$ in the broad* sense on $U$. Then, the intrinsic graph of $\phi$ is a surface of class $C^1_{\mathrm H}$.
\end{coroll}
\begin{proof}
 It is enough to combine \cref{coroll:Koz} and \cref{bigtheoremstep2}.
\end{proof}

\begin{coroll}\label{coroll:Broad*ImpliesDistributionallyStep2}
	Let $\W$ and $\mathbb{L}$ be complementary subgroups of a Carnot group  $\mathbb{G}$ of step 2 with $\mathbb{L}$ horizontal and one-dimensional and choose coordinates on $\G$ such that \eqref{5.2.0.1} is satisfied. Let $U\subseteq \mathbb W$ be an open set, and let $\phi\colon U\subseteq\mathbb W\to \mathbb L$ and $\omega\colon U\to\R^{m-1}$ be two continuous functions. Assume that $D^\phi \phi=\omega$ in the broad* sense on $U$. Then $D^\phi\phi=\omega$ in the sense of distributions on $U$.
\end{coroll}

\begin{proof}
	It is enough to combine \cref{bigtheoremstep2} and \cref{coroll:Broad*ImpliesDistributionally}. 
\end{proof}

\begin{rem}
	 The converse implication of \cref{coroll:Broad*ImpliesDistributionallyStep2} is, up to now, only known for Heisenberg groups, see \cite{BSC10a}. This implication in the general step-2 case will be a subject of further investigations by means of the techniques exploited in this section.
\end{rem}

\subsection{Main theorem in Carnot groups of step 2}\label{sub:Main:step2}
Now we are in a position to give the following theorem (stated in \cref{thm:MainTheorem.0Intro2}), which shows that the assumption on the vertically broad* h\"older regularity in \cref{thm:MainTheorem} can be dropped if we are inside a Carnot group of step 2 and $ \mathbb L$ is one-dimensional. We use the same conventions as in \cref{thm:MainTheorem}, following the notation of \cref{def:coordinateconletilde}.

\begin{theorem}\label{thm:MainTheorem.0}
	Let $\W$ and $\mathbb{L}$ be complementary subgroups of a Carnot group  $\mathbb{G}$ of step 2 with $\mathbb{L}$ horizontal and one-dimensional. Let $\widetilde U\subseteq \sW$ be an open set and let $\widetilde{\phi}\colon\widetilde{U} \to\mathbb L$ be a continuous function. Then the following conditions are equivalent:
	\begin{itemize}
		\item[(a)] $\widetilde\phi\in {\rm UID}(\widetilde U, \mathbb W; \mathbb L)$;
		\item[(b)] $\widetilde\phi\in {\rm ID}(\widetilde U, \mathbb W; \mathbb L)$ and $\de^{\phi}\!\phi$ is continuous on $\widetilde U$;
		\item[(c)] there exists $\omega\in C(U;  \R ^{m-1} )$ such that, for every $a\in U$, there exist $\delta>0$ and a family of functions $\{\phi_\eps\in C^1(B(a,\delta)):\eps\in (0,1)\}$ such that
		\[
		\lim_{\eps\to0}\phi_\eps=\phi \quad\text{and}\quad\lim_{\eps\to0}D_j^{\phi_\eps}\phi_\eps=\omega _j  \quad\text{in $L^\infty(B(a,\delta))$},
		\]
		for every  $j=2,\dots,m$;
		\item[(d)]  there exists $\omega\in C(U; \R ^{m-1})$ such that $D^\phi \phi=\omega$ in the broad sense on $U$;
		\item[(e)]  there exists $\omega \in C(U; \R ^{m-1} )$ such that $D^\phi \phi=\omega$ in the broad* sense on $U$ with the choice of coordinates of \eqref{5.2.0.1}.
	\end{itemize}
\end{theorem}

\begin{proof}
	(a)$\Rightarrow$(b) is trivial, by item (b) of \cref{prop:ContinuityOfDifferentialUid} and (b)$\Rightarrow$(a) follows from \cref{prop:IdGradContImpliesBroad}, \cref{bigtheoremstep2} and \cref{coroll:IDContinuousImpliesUid}.
	
	(a)$\Rightarrow$(c) follows from (a)$\Rightarrow$(b) of \cref{thm:MainTheorem} and (c)$\Rightarrow$(a) follows by combining \cref{prop4.12} and \cref{coroll:step2.2}. 
	
	(d)$\Rightarrow$(e) is trivial, by \cref{defbroad*}. (e)$\Rightarrow$(a) follows from \cref{coroll:step2.2}. (a)$\Rightarrow$(d) follows from (a)$\Rightarrow$(c) of \cref{thm:MainTheorem}.
\end{proof}

\end{document}